\documentclass{article}
\usepackage{epstopdf}
\usepackage[utf8]{inputenc}
\usepackage{nicematrix}
\usepackage{amssymb,xcolor,tikz-cd,mathtools,bm,lscape,marvosym}
\usepackage{hyperref}  
\usepackage{scalerel,stackengine}
\usepackage[normalem]{ulem}   
\usepackage{amsthm}           
\usepackage[all]{xypic}            
\stackMath
\graphicspath{{./tex2/xfig-bratelli5}} 
\graphicspath{{./tex2/xfig-gramdetbyGST}} 

\newcommand\reallywidehat[1]{%
\savestack{\tmpbox}{\stretchto{%
  \scaleto{%
    \scalerel*[\widthof{\ensuremath{#1}}]{\kern.1pt\mathchar"0362\kern.1pt}%
    {\rule{0ex}{\textheight}}
  }{\textheight}%
}{2.4ex}}%
\stackon[-6.9pt]{#1}{\tmpbox}%
}

\newcommand{\baa}[1]{\!\!\begin{array}{c} #1 \end{array}\!\!}
\newcommand{\ppm}[1]{\textcolor{red}{#1}}
\newcommand{\ppmm}[1]{#1}  
\newcommand{\bajm}[1]{\textcolor{blue}{#1}}
\newcommand{\bajmm}[1]{#1}

\newtheorem{theorem}{Theorem}[section]
\newtheorem{corollary}{Corollary}[theorem]
\newtheorem{conjecture}{Conjecture}[theorem]
\newtheorem{lemma}[theorem]{Lemma}
\newtheorem{lem}[theorem]{Lemma}     
\newtheorem{proposition}[theorem]{Proposition}
\theoremstyle{definition}

\newtheorem{remark}[theorem]{Remark}

\newcommand{\beq}{\begin{equation}}
\newcommand{\eq}{\end{equation}}

\newcommand{\ignore}[1]{}

\newcounter{minidef}[section]

\newcommand{\mdef}{\refstepcounter{theorem}       \medskip \noindent ({\bf \thetheorem}) }

\newcounter{minicapt}

\newcommand{\tp}[2]{\begin{tikzpicture}[scale=#1] #2 \end{tikzpicture}} 
\newcommand{\flip}[1]{\reflectbox{\rotatebox[]{180}{#1}}}

\newcommand{\sigone}[1]{\tp{#1}{
\draw (1,0) -- (2,-1); 
\draw (2,0)--(1,-1);
\draw (3,0)--(3,-1);
\draw (4,0)--(4,-1);
}}

\newcommand{\sigtwo}[1]{\tp{#1}{
\draw (1,0) -- (1,-1); 
\draw (2,0)--(3,-1);
\draw (3,0)--(2,-1);
\draw (4,0)--(4,-1);
}}

\newcommand{\onecupp}[8]{ 
\begin{tikzpicture}[scale=0.368]
     \draw (0,0) -- (#1,0) ;        
     \foreach \x in {#4,#5,#6,#7,#8}
          \draw (\x,0) -- (\x,-1);  
     \draw (#2,0) .. controls (#2,-.5) and (#3,-.5) .. (#3,0) ;
 \end{tikzpicture}
 }

\newcommand{\onecupppppo}[6]{ 
\begin{tikzpicture}[scale=0.268]
     \draw (0,0) -- (#1,0) ;        
     \draw (#4,0) .. controls (#4,-1) .. (2,-2);  
     \draw (#5,0) .. controls (#5,-1) .. (3,-2);  
     \draw (#6,0) .. controls (#6,-1) .. (4,-2);  
     \draw (#2,0) .. controls (#2,-.5) and (#3,-.5) .. (#3,0) ;    
 \end{tikzpicture}
 }

\newcommand{\onecuppppp}[6]{ 
\begin{tikzpicture}[scale=0.268]
     \draw (0,0) -- (#1,0) ;        
     \draw (#4,0) .. controls (#4,-1) .. (2,-2);  
     \draw (#5,0) .. controls (#5,-1) .. (3,-2);  
     \draw (#6,0) .. controls (#6,-1) .. (4,-2);  
\draw (2,-2) -- (2,-2.5) ; 
\draw (3,-2) -- (3,-2.5) ; 
\draw (4,-2) -- (4,-2.5) ; 
\draw (1.79,-2.5) -- (4.1,-2.5) -- (4.1,-2.92) -- (1.79,-2.92) -- (1.79,-2.5);   
     \draw (#2,0) .. controls (#2,-.5) and (#3,-.5) .. (#3,0) ;    
 \end{tikzpicture}
 }

\newcommand{\onecupppppx}[6]{ 
\begin{tikzpicture}[scale=0.268]
     \draw (0,0) -- (#1,0) ;        
     \draw (#4,0) .. controls (#4,-1) .. (2,-2);  
     \draw (#5,0) .. controls (#5,-1) .. (3,-2);  
     \draw (#6,0) .. controls (#6,-1) .. (4,-2);  
\draw (2,-2) -- (2,-2.5) ; 
\draw (3,-2) -- (4,-2.5) ; 
\draw (4,-2) -- (3,-2.5) ; 
\draw (1.79,-2.5) -- (4.1,-2.5) -- (4.1,-2.92) -- (1.79,-2.92) -- (1.79,-2.5);   
     \draw (#2,0) .. controls (#2,-.5) and (#3,-.5) .. (#3,0) ;    
 \end{tikzpicture}
 }

\newcommand{\ptarmigan}{Cheby}
\newcommand{\ramping}{ramping}
\newcommand{\Ramping}{Ramping}

\newcommand{\mat}[1]{\left(\begin{array}{#1}}
\newcommand{\tam}{\end{array}\right)}
\newcommand{\PP}{{\mathtt P}}  
\newcommand{\xx}{\alpha}       
\newcommand{\Deltaa}{\Delta}   
\newcommand{\Deltaaa}[2]{\Deltaa^{#1}_{#2}}
\newcommand{\Delt}[2]{\Delta^{#1}_{(#2)}}  
\newcommand{\Deltt}{{\mathtt D}}   
\newcommand{\Delttt}[2]{\Deltt^{#1}_{(#2)}}  
\newcommand{\Sym}{\Sigma}  
\newcommand{\Q}{{\mathbb{Q}}}  

\def\uu{{\mathsf u}} 
\def\SS{\mathcal{S}} 
\def\BB{\mathcal{B}} 
\def\nn{\nonumber}
\def\a{\alpha}

\def\N{\mathbb{N}}

\def\e{\epsilon}
\def\l{\lambda}
\def\K{\mathbb{K}}

\def\C{\mathbb{C}}
\def\Z{\mathbb{Z}}

\def\R{\mathbb{R}}
\def\th{\theta}

\numberwithin{equation}{subsection}

\title{
On semisimplicity criteria and non-semisimple representation theory for 
the Kadar-Yu algebras}
\author{B Morris and P Martin}   

\begin{document}
\maketitle
\begin{abstract}
    The Kadar--Yu algebras are a physically motivated sequence of towers of algebras interpolating between the Brauer algebras and Temperley--Lieb algebras. 
    The complex representation theory of the Brauer and Temperley--Lieb algebras is now fairly well understood, with each connecting in a different way to Kazhdan--Lusztig theory. 
    The semisimple representation theory of the KY algebras is also understood, and thus interpolates, for example, between the double-factorial and Catalan combinatorial realms. However the non-semisimple representation theory has remained largely open, being harder overall than the (already challenging) Brauer case. 
    In this paper we determine generalised Chebyshev-like forms for the determinants of gram matrices 
    of contravariant forms for standard modules. 

    This generalises the root-of-unity paradigm for Temperley--Lieb algebras (and many related algebras); interpolating in various ways between this and the `integral paradigm' for Brauer algebras. 
    
    The standard module gram determinants give 
    a huge amount of information about morphisms between standard modules, making thorough use of the powerful homological machinery of towers of recollement (ToR), with 
    appropriate gram determinants providing the ToR `bootstrap'. As for the Brauer and TL cases the representation theory has a strongly alcove-geometric flavour, but the KY cases guide an intriguing generalisation of the overall geometric framework.  
\end{abstract}
\newpage 
\tableofcontents

\section{Introduction}   \label{ss:intro}

Practical computation for statistical mechanical models is often built around transfer matrix methods (see e.g. \cite{Baxter81,Baxter83,LiebMattis,Martin91,Wannier,Brown67} 
and references therein). 
Transfer matrix methods in turn rely on a mixture of fourier analytic and algebraic methods (aside: the mathematical tendency to consider also the machinery of integrability is not 
universally  
appropriate  
- most physically interesting models are not integrable, while most do benefit from translation symmetry/fourier and spectral/representation theoretic methods). 
Thus it is, for example, that XXZ spin chains, 2d Potts models, six-vertex models and so on have algebras such as Temperley--Lieb algebras \cite{TemperleyLieb71} in their computational machinery. 
More recently motivations such as topological quantum computation have driven interest in topological phases, and statistical mechanical models of systems with boundary doping 
(see e.g. \cite{Fradkin,GuidrySun} and references therein).  
Just as for homogeneous systems, one can consider the algebraic layer in these settings. There are various interesting examples, but here 
in this paper 
we focus on the construction given in \cite{KadarYu} - hence `Kadar--Yu' algebras. 
In this construction there is a tower of algebras $J_l = J_l(\xx)$ for each integer $l \geq -1$;
and then for each value of the parameter $\xx$ taken from the ground field. 
(Here the default ground field will be $\C$.) 
In statistical mechanical terms the index $l$ sets how far the doping will penetrate into the bulk, but controlled topologically rather than metrically. In algebraic terms it controls which subset of a set of Brauer generators are present.

From the representation theory perspective the KY algebras are also intrinsically interesting. We have already mentioned Temperley--Lieb, and the Brauer algebras \cite{Brauer} and partition algebras \cite{Martin91} also arise both as transfer matrix algebras 
(see e.g. \cite{Gainutdinov_2014,Ramgoolam07,Ramgoolam} and references therein)
and objects of direct study in representation theory. 
For example, the non-semisimple representation theories of Temperley--Lieb and Brauer algebras can be seen as controlled by appropriate different cases of Kazhdan--Lusztig theory. The resultant theories are very different in terms of their fundamental invariants, and so the KY algebras provide an opportunity to interpolate between these cases, and hence 
study  
geometrical representation theory itself, as we will see. 

In this paper we 
focus on representation theory. 
A key technical result is Theorem~\ref{thm:main}.
In principle this reduces the computation of 
`bootstrap' 
gram matrices for each tower $J_l$ 
to the computation of one fixed and two initial polynomials per $\l\vdash l+2$, 
all of which are then routine to determine. 
We then illustrate this with a comprehensive set of examples (see \S\ref{ss:examples0}). 
The Theorem and examples raise two key questions: 
firstly how to characterise roots (analogous to the characterisation for classical 
Chebyshev in terms of roots of unity) - see \S\ref{s:roots};
and then how to run the Tower of Recollement programme from this bootstrap to obtain representation theory
- see \S\ref{ss:repthy}. 
For this in particular we need also in principle to determine, in the cases where a `one-cup' standard module {\em has} a submodule (flagged by a vanishing determinant), 
what is the structure of this submodule. 
For the algebras we study here we show, in \S\ref{ss:bootcase0}-\ref{ss:keybootl1}, that 
in almost all cases 
\ppmm{no module can map in} except the no-cup module with the same labelling partition; and then that this 
is a composition factor with multiplicity 1. 

An intriguing question is how to determine not just the bootstrap determinants but all standard determinants. 
Indeed this was one of the initial aims of the paper. 
This is important here 
for various reasons. Firstly 
because these algebras $J_{l,n}$ exhibit a large variety of homological complexities, and so 
provide a good laboratory to understand the passage from gram determinants to Cartan decomposition matrices, 
including homological grading (manifestations of Loewy depth of factors in modules and so on). 
Ab initio the determinants are, collectively, a lot of information, and it is a challenge even to start to present this in an accessible way - this is itself interesting, raising the question of how global limit properties of representation theory unify local linear data in the fibre over the limit.  
In the $l=-1$ (Temperley--Lieb) case there is indeed a striking property that enables one to determine 
not just the bootstrap determinants, but all standard determinants (we review it in \S\ref{ss:gdetdecrolletl-1}). 
This leads us to consider {\em marginal vertex functions} in the general case. 
And indeed one then observes empirically that the presentation of results is hugely simplified. 
And furthermore the results boil down to expressions in terms of the same polynomials as before. 
In \S\ref{ss:gdetdecrollet} we review the $l=-1$ case and summarize the data showing how this generalises. 
\\
(Aside: Since we have the bootstrap determinants, another approach in principle is to reverse-engineer other gram determinants from the representation theory determined via the tower of recollement - potentially linking directly to generalised Kazhdan--Lusztig theory. To achieve the overall aim the ultimate challenge would be to understand the problem approached from {\em both} directions, however we will demote this level of completeness to  a later work.)

\medskip 

Another very intriguing set of questions arises from technical aspects of our computations. 
Key to these computations are properties of Chebyshev series here called `\ramping'. 
These are important for representation theory since they determine properties of the (discrete but rich) subset of parameter values for non-semisimplicity. 
They thus generalise 
(necessarily quite profoundly)
the `root of unity' phenomenon from classical cases. This is a key part of the aimed-for interpolation between Temperley--Lieb-like and Brauer-like representation theories. But the Chebyshev aspect is intrinsically interesting, as we shall see in \S\ref{s:roots}. 

\medskip

\subsection{Glossary of notations and terms}

See Table~\ref{tab:glossary}. 

\newcommand{\Brp}{{\mathcal B}}   

\begin{table}[]
    \centering
    \begin{tabular}{l|l|l}
    notation & short description & where  \\  
     & & defined \\ \hline
    $\Brp(n,m) $ & full set of $(n,m)$ Brauer partitions  &  \ref{de:Brp} \\ 
    $p \mapsto p^*$ & natural map $*:\Brp(n,m) \rightarrow \Brp(m,n)$ (diagram vertical flip) & \ref{de:flip} \\ 
    $J_{l,n}$   &  Kadar--Yu algebra at rank $n$ and height $l$ & \ref{de:Jln} \\  
    $J_l$     &  Kadar--Yu category at height $l$  &   \ref{de:Jln} \\ 
     $\Deltaaa{n}{(p,\l)} $   &  standard module of $J_{l,n}$ for fixed $l$ and index $(p,\l)$ & \ref{de:stdmods} \\
     $ \Gamma(\Delt{n}{p,\l})  $ &  gram matrix of $\Deltaaa{n}{p,\l}$ wrt given basis & \ref{de:gramdet} \\ 
     $\Delttt{(n)}{x}$  & gram determinant of $ \Gamma(\Delt{n}{p,\l})  $ wrt given basis &\ref{de:gramdet} \\ 
     $\Lambda_n$ & set of integer partitions of $n$ &  \ref{de:Lambdan}  \\ 
     $\Lambda_{l,n}$ & index set for standard modules of $J_{l,n}$  & \ref{de:labelset}  \\ 
     $\Lambda_{l,\infty}$ & index set for standard modules in global/large $n$ limit  & \ref{def:indexset}  \\ 
     $\SS_\l$ & Specht module of $\Sym_n$ (for some $n$) with index $\l \in \Lambda_n$  &  \ref{eq:Specht0} \\ 
     $\Sym_n$ & symmetric group of rank $n$\\
     $\uu_{ij}$ & fix $n$. cup (at $\{i,j\}$) element of $\Brp(n,n-2)$  & \ref{de:uu} \\ 
     $\uu_{ij} \uu_{i'j'}^*$ & fix $n$. cup-cap (at $\{i,j\},\{ i',j' \}$) element of $\Brp(n,n)$  & \ref{de:uu} \\ 
    \end{tabular}
    \caption{Glossary of notations 
    } 
    \label{tab:glossary}
\end{table}

\newpage

\section{Preambles}

\subsection{
The algebras $J_{l,n}$ and some key properties}
In this subsection, we give a brief {review}
of the algebras, $J_{l,n}$, constructed in \cite{KadarYu} and recall some properties of these algebras which will be of use in later sections. For a more detailed description of the $J_{l,n}$ we refer readers to the original reference \cite{KadarYu}. 

\mdef   \label{de:Brp}
Fix integers $n,m\geq 0$. 
A Brauer partition of type $(n,m)$ is a pair partition of the set $\{1,\dots, n\}\sqcup \{1',\dots, m'\}$ (we use primed labels to distinguish elements of the two sets). Let $\mathcal{B}(n,m)$ denote the set of all such partitions (observe that $\mathcal{B}(n,m)$ is empty unless $n\equiv m\mod 2$). 

We write $\Brp^{r}(n,m)$ for the subset of  $\Brp^{}(n,m)$ of partitions with at most $r$ pairs of the form $\{i,j'\}$ 
(called propagating lines); and $\Brp^{=r}(n,m)$ for the subset with exactly $r$ propagating lines. 

\mdef   \label{pa:left-right}
A partition, $p \in \mathcal{B}(n,m)$, may be represented by a Brauer diagram as follows. 
A Brauer diagram consists of a rectangular region with unprimed (primed) labels placed on the top (bottom) edge 
increasing from left to right, and inscribed curves joining labels which belong to the same part. Curves in a Brauer diagram only intersect transversely, on their interiors, finitely many times, and do not have points of triple intersection. We frequently abuse notation, representing by a Brauer diagram its underlying partition. 

\mdef  \label{de:flip}
We write $p \mapsto p^*$ for the map $*: \Brp(n,m) \rightarrow \Brp(m,n)$ corresponding to 
vertical flip of a diagram, hence swapping primed for unprimed. Examples:
\medskip 

{\centering
\sigone{.41}  \hspace{.32cm}$ \stackrel{*}{\mapsto}$\hspace{.32cm}  \sigtwo{.41}

\sigtwo{.41}  \hspace{1cm}   \sigone{.41}

\medskip }

Observe in particular that there is a copy of the symmetric group $\Sym_n$ (so far just as a set) in $\Brp(n,n)$, indeed it is precisely $\Brp^{=n}(n,n)$, and here the flip takes an element to its inverse in the group.

\mdef    \label{pa:ground-ring}
Fix a commutative ring $\K$ - by default we will take $\K=\C$. 
Fix $\a \in \K$. Let $Br(n,m)$ denote the free $\K$-module with basis 
$\mathcal{B}(n,m)$. Define the Brauer category $Br=Br(\a)$ to be the triple $(\N_{0}, Br(\_, \_), \circ )$, where for $(p_1,p_2)\in Br(n,m)\times Br(m,l)$, composition is defined by linearly extending the formula
\[p_1\circ p_2 = \a^{d_{p_1,p_2}} p_1\# p_2.\]
Here $p_1\# p_2 \in Br(n,l)$ is the resultant Brauer partition determined by vertically juxtaposing any two diagrams for $p_1$ and $p_2$,
with $p_1$ over $p_2$, 
and $d_{p_1,p_2}$ is the number of closed curves in any such juxtaposition. For brevity, we adopt the multiplicative notation $p_1 \circ p_2=p_1 \, p_2$. 

\mdef \label{de:flip2}
The *-map extends to a contravariant functor $(\_)^*:Br(\a)\to Br(\a)$ with $n^*=n$ for objects, and $(\_)^*:Br(n,m)\to Br(m,n)$ is given by linearly extending the map from \ref{de:flip}.

\mdef 
Note that the category $Br$ has a monoidal structure given by $n\otimes m=n+m$ on objects, and where $p_1\otimes p_2$ is the horizontal juxtaposition of diagrams. 

\mdef The {\em height} of a Brauer diagram is defined to be the maximal height of all points of intersection of embedded curves as in \cite[\S2.2]{KadarYu}. If a diagram has no points of intersection, its height is defined to be $-1$.

\mdef The height of a Brauer partition, $p\in \mathcal{B}(n,m)$, is defined to be the minimal height over all diagrams representing $p$. Write $ht(p)$ for the height of $p$.

\mdef   \label{de:Jln}
Fix an integer $l\geq -1$. Let $J_{l}(n,m)$ denote the free $\K$-module with basis 
$$
\Brp_l(n,m) \; := \; \{ p \in \mathcal{B}(n,m) \ | \ ht(p)\leq l\}    .
$$ 
In \cite[Thm. 3.4]{KadarYu} it is proved that for any $(p_1, p_2)\in \mathcal{B}(n,m)\times \mathcal{B}(m,l)$,  $ht(p_1\# p_2)\leq \max(ht(p_1),ht(p_2))$, and thus 
$J_l:=(\N_0,J_l(\_, \_), \circ )$ defines a subcategory of $Br$ for each $l$. 

Following \ref{de:Brp}, let $J_l^p(n,m)$ denote the submodule of $J_{l}(n,m)$ spanned by diagrams with at most $p$ propagating curves, and let $J_l^{=p}(n,m)$ denote the submodule spanned by those exactly $p$ propagating curves.   

Let $J_{l,n}:=J_{l}(n,n)$ denote the endomorphism algebras of $J_l$, with $J^{p}_{l,n}=J^{p}_{l}(n,n)$ and likewise $J^{=p}_{l,n}=J^{=p}_{l}(n,n)$. 
We refer to the $\K$-algebras $J_{l,n}$ as the Kadar-Yu algebras.

\mdef\label{pa:Jlngens} In \cite{KY2} it is proved that $J_{l,n}$ is generated as an algebra by the usual ``cup cap" Temperley-Lieb generators $e_i=\uu_{i,i+1} \uu_{i,i+1}^*$ for $i=1,\dots, n-1$ (using the notation of \ref{cupbasis}), and the elementary transpositions $s_j=(j \ j+1)$ for $j=1,\dots,l+1$ (\textit{i.e.} $s_j \in \Sigma_{l+2}$).

\begin{remark} When $l=-1$, the Category $J_{-1}$ may be identified with the Temperley-Lieb category, since in this case all pair partitions may be realised by diagrams without intersections. 
\end{remark}

\mdef 
Note that the monoidal product $\otimes$ on category $Br$ does not descend to one on $J_{l}$ (unless $l=-1$) since, for instance, $ht(id_n \otimes p)=ht(p)+n$ provided $ht(p)\geq 0$. However, it will be convenient for our aims to utilise tensor notation. 
For example, given $x\in J_l(a,b)$ we have $x\otimes id_n \in J_l(a+n,b+n)$. 

\medskip 

In what follows, we define two adjunctions crucial in studying the representation theory of  the algebras $J_{l,n}$. 



\mdef Observe there is a natural algebra inclusion $J_{l,n} \hookrightarrow J_{l,n+1}$ given by 
$x\mapsto x\otimes id_1$. We let $\; Res^{J_{l,n+1}}_{J_{l,n}}: \; J_{l,n+1}-\text{mod} \; \; \rightarrow \;\; J_{l,n}-\text{mod}\;$ denote the associated ``restriction" functor, and $\; Ind^{J_{l,n+1}}_{J_{l,n}} \; :\; J_{l,n}-\text{mod}\;\; \rightarrow\;\; J_{l,n+1}-\text{mod}\;$ its adjoint ``induction" functor.

\mdef  \label{de:globloc0}
There is a ``globalisation" functor 
$\; G:\; J_{l,n}-mod \;\; \rightarrow\;\; J_{l,n+2}-mod \;$ given by 
\[ G(M)=J_l(n+2,n)\otimes_{J_{l,n}} M.\]
It has adjoint ``localisation" $\; F:\; J_{l,n+2}-mod \;\; \rightarrow\;\; J_{l,n}-mod \;$ given by
\[ F(M)=J_l(n,n+2)\otimes_{J_{l,n+2}} M.\]
Observe that $F\circ G=id$, since $J_l(n,n+2)\circ J_l(n+2,n)=J_{l,n}$. 
Thus $G$ is an embedding. 
On the other hand 
$$
J_l(n+2,n)\circ J_l(n,n+2)=J_l^n(n+2,n+2)\subsetneq J_{l,n+2}  .
$$
\ppmm{
Note that this is a proper two-sided ideal of $J_{l,n+2}$ and that the quotient 
$J_{l,n+2} / J_l^n(n+2,n+2) \cong \K \Sym_{\min(l+2,n+2)}$. }

\mdef \label{de:labelset} 
The following can be proved for instance by using the functors from $\ref{de:globloc0}$ 
(see \cite{KadarYu}); and we will recall the explicit construction in \ref{de:stdmods}. 
\\
Up to isomorphism,  
standard  
$J_{l,n}$-modules can be indexed by the set
\beq
\Lambda_{l,n} = \{  ( p, \l)\; | \; p \in \N_0, \; 
n-p \in 2\N_0 ,  
\; \l\vdash   \min(p,l+2) \}.\label{eq:simples}
\eq
and for $\xx \in \C^\times$ simple modules are indexed by the same set.

\subsection{On Standard modules for $J_{l,n}$ 
}
\label{ss:sdmods}




In order to fix conventions here we recall the definition of standard modules and their contravariant form. 
For this we will first need to fix conventions for $\Sym_n$ Specht modules, where $\Sym_n$ is the symmetric group. 
Observe that the category of symmetric group $\K$-algebras is a monoidal subcategory of $Br$. 

\newcommand{\nE}{{\mathtt{E}}}  
\newcommand{\nF}{{\mathtt{F}}}  
\newcommand{\nC}{{\mathtt{C}}}  
\newcommand{\cc}{{\mathsf{c}}}

Fix $l,n$. Fix $p$ so that $n-p=2k$ for some $k\geq 0$, and $r \in \{0, 1, ..., min(p,l+2)\}$. 
Let $\cc_\l$ be any element of $\Q\Sym_{r}$ such that the 
$\C\Sym_r$ Specht module $\SS_\l \cong \C\Sym_{r} \cc_\l$
(here $\Sym_0$ is the trivial group). 
Then $\cc_\l \otimes 1_{p-r} \in \Q\Sym_p$ and can be regarded as an element of $J_l(p,p)$ by the obvious inclusion, 
\ppmm{provided that $\Q\subset\K$}. 
%
%
As in \cite{KadarYu}  
\[
\Deltaa^{n}_{(p,\lambda)} \;\; \cong  \;\; J_l(n,p) ( \cc_\l \otimes 1_{p-r}) / J_l^{p-2}(n,p)
\]
where $(p,\lambda) \in\Lambda_{l,n}$ as per \ref{eq:simples}; 
$J_l^{p-2}(p,n)$ is as per \ref{de:Jln}.

It will now be convenient to fix $ \Deltaa^{n}_{(p,\lambda)}$  precisely by choosing $\cc_\l$.

\subsubsection{Specht module conventions recalled}

\mdef \label{pa:Specht}
Let $\nE_n$ denote the usual symmetriser for the symmetric group $\Sym_n$ (the unique idempotent in 
 $\Q\Sym_n$ inducing the trivial module); 
and $\nF_n$ the antisymmetriser. 
We will write $E_n$ and $F_n \in \Z\Sym_n$ for the unnormalised versions. 
Then for 
$\l = (\l_1 , \l_2, ..., \l_m) \vdash n$ define 
\[
\nE_\l = \nE_{\l_1} \otimes \nE_{\l_2} \otimes ... \otimes \nE_{\l_m}
\]
and, with 
$\l' = (\l_1' , \l_2' , ...,\l_{m'}')$ denoting the conjugate partition, define 
\[
\nF_\l = \nF_{\l'_1} \otimes \nF_{\l'_2} \otimes ... \otimes \nF_{\l'_{m'}}
\]
Note that $\nE_\l^* = \nE_\l$ and $\nF_\l^* = \nF_\l$ using $*$ from (\ref{de:flip2}). 
Recall that there is a (shortest) $w' \in \Sym_n$ such that 
\beq  \label{eq:preform}  
E_\l \C \Sym_n  F_\l = \C E_\l w' F_\l
\eq 
and similarly for $F_\l w^{} E_\l$ (with $w' = w^{-1}$). 
Note that these one-dimensionality properties hold true under replacement of $E_\l , F_\l$ with any conjugates 
or scalar multiples. 
Let $C_\l =  F_\l w E_\l $. 

Observe from \eqref{eq:preform} that for any $a \in \C \Sym_n$, there is a unique $f(a) \in \C$ defined by the equation
\[E_\l w' F_\l a F_\l w E_\l = f(a) E_\l w' F_\l w E_\l.\]
Note here that $f$ does depend on the normalisations. 
If we want to work in $\Z\Sym_n$ then $E_\l$ and $F_\l$ should be the unnormalised versions. 

For each choice of normalisations there is a form $(-,- ): \C\Sym_n C_\l \times \C\Sym_n C_\l \rightarrow \C $ 
defined by the equation
\beq (a C_\l)^* b C_\l=C_\l^* a^* b C_\l = ( a C_\l, b C_\l ) E_\l w' F_\l w E_\l, \label{eq:form}\eq
which is symmetric, bilinear, and contravariant with respect to $*$ since $( gaC_\l,bC_\l) = ( aC_\l,g^* bC_\l )$. 

Let  $\SS_\l$ denote the $\C\Sym_n$-module (and indeed left ideal) 
\beq  \label{eq:Specht0}
\SS_\l = \C \Sym_n  F_\l w E_\l   =   \C \Sym_n  C_\l
\eq 
- this is the Specht module (up to isomorphism). Let $d_\l$ denote the hook length of the partition $\l$ so that, in particular, $\dim(\SS_\l)=d_\l$.

There are various isomorphic incarnations. Some are better suited to our aims that others. 
For example 
we can start with the Young module for $\l=(21)$ in the form 
$Y_{21} \cong \C \Sym_3 E'_{21}$ where $E'_{21} =E_{12} := (E_1 \otimes E_2)$. 
(This just means that $F_{21} E'_{21}$ is congruent to $F_{21} w E_{21}$ without needing the $w$.) 
Thus $F_{21} E'_{21}$ induces a left subideal, which is easily seen to be proper. 
Indeed $\C\Sym_3  F_{21} E'_{21} = \C\{ F_{21} E'_{21} , (23) F_{21} E'_{21} \} $. 
The version of the form in terms of this basis is given firstly by 
$ ( F_{21} E'_{21} )^*  F_{21} E'_{21} =  E'_{21}  F_{21}^2 E'_{21}  
          = 2 E'_{21} F_{21} E'_{21}$
and then 
$ ( F_{21} E'_{21} )^*  (23)   F_{21} E'_{21} =  E'_{21}  F_{21} (23) F_{21} E'_{21}  
        =   E'_{21}  (1-(12)) (23) F_{21} E'_{21}   =  E'_{21} F_{21} E'_{21}$ 
since $ E'_{21} (12)(23) F_{21} =0$. 
If we choose outputs normalised by $f(a) E'_{21} F_{21} E'_{21} $ 
then here the gram matrix is $\Gamma$ as below: 
\[
\Gamma = \mat{cccc} 2&1 \\ 1&2 \tam     \hspace{3cm} 
\Gamma_x = \mat{cccc} 2&0 \\ 0&6x^2 \tam 
\]
Note that if we want basis elements to be orthogonal we can replace 
$(23) F_{21} E'_{21} $  with  $(x.1+y.(23)) F_{21} E'_{21} $ with $y=-2x$. 
This gives the second gram matrix above. 
Note that the overall factor of 2 can be removed as a convention; but normalisation here 
requires an extension of the ground ring from $\Z$, for example to $\C$. 

Recall that a basis for the Young module 
$Y_{21}$  
can be written as the set of words: 
211, 121, 112.  
In the simplest version these are $E'_{21}$, $(12)E'_{21} $ and $(32)(12)E'_{21} $ respectively. 
We think of the words 211 and so on graphically as stepping along directions $e_1$ and $e_2$ in the plane. Specifically we orient the axes with $e_1$ pointing up-right and $e_2$ down-right. 
Thus, with a first convenient change of basis we have:  
\[
E'_{21} \mapsto 211 \mapsto 
\tp{.31}{\draw (0,0) -- (1,-1); \draw (1,-1) -- (2,0) -- (3,1);}
\hspace{2.1cm} 
F_{21}E'_{21} \; =\; (1-(12))E'_{21} \; \mapsto 121 
\mapsto \tp{.31}{\draw (0,0) -- (1,1); \draw (1,1) -- (2,0) -- (3,1); \draw [red] (0,0) -- (1,-1) -- (2,0);}  
\]
- adding a box, as it were, in the (12) position (note that the factor $(x1+y(12))$ can 
in principle have any $x,y$ with $y \neq 0$ for independence, the chosen coefficients make the element induce a submodule); 
\[ 
(23)(1-(12))E'_{21} \mapsto 112    
\mapsto \tp{.31}{\draw (0,0) -- (1,1); \draw (1,1) -- (2,2) -- (3,1);}  
\ignore{{
\hspace{1cm} 
(1-(12))E'_{21} \mapsto 121 
\mapsto \tp{.31}{\draw (0,0) -- (1,1); \draw (1,1) -- (2,0) -- (3,1);}  
\hspace{1cm} 
E'_{21} \mapsto 211 \mapsto 
\tp{.31}{\draw (0,0) -- (1,-1); \draw (1,-1) -- (2,0) -- (3,1);}
}}
\]

\mdef   \label{pa:21ortho}
It will be convenient to have a basis for the Specht module that is orthonormal with respect to the form that we will use in \S\ref{ss:gram-mat}. For an example let us consider the basis 
(see also \eqref{eq:21ortho})  
$\{  C_{12} , \frac{1}{\sqrt{3}}(-1.1+2.(23))C_{12} \}$. 
Note from $\Gamma_x$ that this is indeed orthogonal. 

\mdef For a positive integer $n$ and partition $\l\vdash n$ we define an equivalent convention for the Young symmetriser:
\beq \nC_\l := a_\l \, E_\l w' F_\l w E_\l \in \C \Sym_{n},\label{eq:ccl}\eq
with $a_\l \in \Q$ chosen so that $\nC_\l^2=\nC_\l$. Then $\C \Sym_n \nC_\l$ is isomorphic to the Specht module $\C \Sym_n \nC_\l \simeq \SS_\l$. Furthermore, we note that $\nC_\l^*=\nC_\l$. We will use this convention for the remainder of the paper.


\medskip  

\subsubsection{Standard module conventions}

\mdef 
Now fix $l$ and consider $\l \vdash r \leq l+2$. Observe that for $p \geq r$   
we have $\nC_\l \otimes id_{p-r} \in J_{l,p}$, with $\nC_\l$ as in \eqref{eq:ccl}. 
Fixing $p$ define 
\beq   \label{de:cl}
c_\l = \nC_\l \otimes id_{p-r}.
\eq
For example when $p\geq 3$ 
\[
c_{(2,1)} = \frac{1}{6}(e+(1 2))(e-(1 3) )(e+(1  2))\otimes id_{p-3}\in J_{l=1,p}.
\]
using cycle notation for permutations.

\mdef\label{de:stdmods} 
\ppmm{Essentially}
as in \cite{KadarYu} 
\ppmm{(here we require $\Q\subset\K$, whereas \cite{KadarYu} does not)} 
we define the standard modules of $J_{l,n}$ (for fixed $l$) by 
\[
\Deltaa^{n}_{(p,\lambda)} \; := \; J_l(n,p)  c_\l / J_l^{p-2}(n,p)
\]
where $(p,\lambda) \in\Lambda_{l,n}$ as per \ref{eq:simples}; 
$J_l^{p-2}(p,n)$ is as per \ref{de:Jln};
and $c_\l$ is defined as in (\ref{de:cl}). 

\mdef \label{pa:convs}
For any $x\in J_l(m, p)$ with $m\geq p$, we often abuse notation and write simply $x \, c_\l$ for the corresponding element in the module $\Deltaa^{m}_{(p,\l)}$, 
that is, $x \, c_\l$ denotes the class $x \, c_\l+J_{l}^{p-2}(m,p)$. 
In particular, for any $\xi \in \Delta^{n}_{(p,\l)}$ 
and $x\in J_{l}(m,n)$ we have $x \xi \in \Delta^{m}_{(p,\l)}$, since we may write $\xi = y c_\l$ for some $y \in J_l(n,p)$, and thus $x \xi = (xy)c_\l$.

With this notation we observe, for example, that for $(p,\l)\in \Lambda_{l,p}$, 
\beq
\Deltaa^{p}_{(p,\l)}=J_{l,p} c_\l= (\C\Sigma_{r}\nC_\l)\otimes id_{p-r} 
  \stackrel{\C\Sym_r-mod}{\simeq}  \SS_\l
\label{eq:conv}\eq
where $\l\vdash r$. 
The latter is an isomorphism of $\C \Sym_r$-modules.

\mdef \label{pa:spbasis}
Let $\mathcal{B}=\{b_1,\dots,b_{d_\l}\}$ be a basis for the $\Sym_r$ Specht module $\SS_\l$ (thus $\l\vdash r$). By \eqref{eq:conv}, this determines a basis for $\Delta^p_{(p,\l)}$ by $b_i\mapsto b_i \otimes id_{p-r}$; we abuse notation and write simply $b_i$ for the element $b_i\in \Delta^p_{(p,\l)}$ (likewise, for any element $x\in \C \Sym_r$ we write $x$ for the element $x\otimes id_{p-r} \in J_{r-2, p}$, and therefore can make sense of the equation $\C \Sym_{r}\subset J_{r-2,p}$). Thus we may regard a basis for $\SS_\l$ as one for $\Delta^p_{(p,\l)}$ (for any $p$). We will make use of this fact repeatedly in \S\ref{s:onecup0}. 

\mdef \label{pa:diagbas} \bajmm{Following \cite{KadarYu}, we give a basis for a standard module as follows:} Fix $\l \vdash r\leq p$, and let $\mathcal{B}$ be a basis for $\SS_\l$, regarded as one for $\Delta^p_{(p,\l)}$. For any $n\geq p$ with $n\equiv p\mod 2$,  the module $\Deltaa^{n}_{(p,\lambda)}$ has a basis given by
\beq
\{ u\, b_i \ | \ b_i \in \mathcal{B}, \ u \in J_l^{||}(n,p)\}\label{eq:diagbas}
\eq
where $J_l^{||}(n,p)$ is the set of ``half-diagrams" of type $(n,p)$, consisting of those diagrams from $J_l(n,p)$ with $p$ pairwise non-crossing, propagating lines (\textit{cf.} \cite[4.7]{KadarYu} \footnote{Converting from the notation here to that in \cite{KadarYu}, $J^{||}_l(n,p)$ becomes $J_l^{||}(n,p,p)$}).  \bajmm{A half-diagram in $J_l(p+2k,p)$ for $k\geq 0$, is determined by $k$ ``cups" (that is, parts or curves joining unprimed labels at the top of a diagram). As such, we refer to a module $\Deltaa^{n=p+2k}_{(p,\lambda)}$ as a $k$-cup module. We refer to a basis of the form \eqref{eq:diagbas} as a diagram basis. Note that all diagram bases are related by a change of basis (over $\C$) since they depends only on the choice of $\C$-basis for $\SS_\l$.}
\medskip


\subsubsection{Contravariant form on standard modules}

From here until \S\ref{ss:repthy} we take $\K = \C[\xx]$.   

\mdef   \label{de:grammat}
The  form 
$ \langle\_,\_\rangle : \Deltaaa{n}{p,\l} \times \Deltaaa{n}{p,\l} \rightarrow \C[\xx]$ 
is defined by the equation
\beq 
(a c_\l)^* b c_\l=c_\l a^* b c_\l   \;= \;\; \langle a c_\l, b c_\l \rangle c_{\l}.  \label{eq:Gform}
\eq
This form is symmetric, bilinear and contravariant with respect to *, that is, for any $x\in J_{l,n}$ we have $\langle x a c_\l,  b c_\l \rangle=\langle a c_\l, x^* b c_\l \rangle$. 

\mdef 
Notice when $p=n$, we have 
$\Delta^{p}_{p,\l} \stackrel{\C}{\simeq} \SS_{\l}$, 
and the form here may differ from  the form $(\_,\_)$ as per \eqref{eq:form}. 
In particular, observe that $\langle c_\l, c_\l \rangle=1$. 
Furthermore,  since $\SS_\l$ is simple (as a $\Sym_r$-module) over the complex numbers, $\Delta^p_{(p,\l)}$ is simple (as a $J_{r-2,n}$ module) and hence the form $\langle\_,\_\rangle$ is non-degenerate. It makes sense, therefore, to choose a basis for $\SS_\l$ (which we regard as one for $\Delta^p_{(p,\l)}$), which is orthonormal with respect to this form. This is what we will mean, by an orthnormal basis for $\SS_\l$.

\mdef   \label{de:gramdet}
We write $\Gamma(\Delt{n}{p,\lambda})$ 
for the Gram matrix (in a given diagram basis), and,  following \textit{e.g.} \cite{ShBr},
we will write 
\beq
\Delttt{n}{p,\lambda}  \; := \; \det( \Gamma(\Delt{n}{p,\lambda} ))   
\eq 
for the Gram determinant. Since entries of $\Gamma(\Delt{n}{p,\lambda})$ are polynomials in $\a$, it follows that $
\Delttt{n}{p,\lambda}\in \C[\a]$. We sometimes write $
\Delttt{n}{p,\lambda}(\a)$ to emphasise this. By \ref{pa:diagbas}, $\Delttt{n}{p,\lambda}(\a)$ is defined up to a complex scalar.

\mdef   \label{exa:keyn5la21} Consider the $J_{1,5}$-module $\Deltaaa{{5}}{({3},(21))}$. By choosing a basis $\{\nC_{(2,1)}, (2 \ 3) \nC_{(2,1)}\}$ for $\SS_{(2,1)}$, we can give the following diagram basis for this module where the rectangular box ``receiving''  
three propagating lines denotes the idempotent $\nC_{(2,1)}$. 
\begin{align}
&\hspace{-1cm} 
\onecuppppp{6}{1}{2}{3}{4}{5}  \; 
\onecuppppp{6}{1}{3}{2}{4}{5}  \;
\onecuppppp{6}{1}{4}{2}{3}{5}  \;
\onecuppppp{6}{2}{3}{1}{4}{5} \;
\onecuppppp{6}{2}{4}{1}{3}{5}  \;
\onecuppppp{6}{3}{4}{1}{2}{5}  \;
\onecuppppp{6}{4}{5}{1}{2}{3}  
\nn\\
&\hspace{1cm} 
\onecupppppx{6}{1}{2}{3}{4}{5}  \;
\onecupppppx{6}{1}{3}{2}{4}{5}    \;
\onecupppppx{6}{1}{4}{2}{3}{5}  \;
\onecupppppx{6}{2}{3}{1}{4}{5}  \;  
\onecupppppx{6}{2}{4}{1}{3}{5}  \;
\onecupppppx{6}{3}{4}{1}{2}{5}  \;
\onecupppppx{6}{4}{5}{1}{2}{3}  \label{eq:21basis}
\end{align}
The value of the form on two diagram basis elements is determined by applying * to one of the diagrams (flipping it upside down), vertically juxtaposing the diagrams, and computing the associated scalar. Schematically then, we compute 
\begin{align*}&\reflectbox{\rotatebox[]{180}{\onecuppppp{6}{1}{2}{3}{4}{5}}}=\a,
&&\reflectbox{\rotatebox[]{180}{\onecuppppp{6}{1}{2}{3}{4}{5}}}=1.\\[-0.822em]
&\ \onecuppppp{6}{1}{2}{3}{4}{5}
&&\ \onecuppppp{6}{1}{3}{2}{4}{5}\end{align*}

\noindent 
and so on. In particular 
\ignore{{
\bajm{[I am now convinced we actually need to switch to the EFE convention for the symmetriser - this is because the explicit matrices we are writing will be wrong otherwise... for example, when $\l=(2,1)$, computing $\langle (12),e\rangle$ gives 1 in the EFE formalism since $(1 2) E=E$, yet it gives $-1$ in the FE formalism since $(12)F=-F$... For example, using $C_\l$ as written here we find  ]}
\begin{align*}&\reflectbox{\rotatebox[]{180}{\onecuppppp{6}{1}{2}{3}{4}{5}}}=C_{\l}^* (12) C_{\l}=-C_\l^* C_\l \Rightarrow \left\langle \onecuppppp{6}{1}{2}{3}{4}{5} , \onecuppppp{6}{1}{4}{2}{3}{5} \right\rangle =-1
\\[-1.7em]
&\ \onecuppppp{6}{1}{4}{2}{3}{5}\end{align*}
\bajm{[However, this should correspond to the (1,4) entry of the Gram matrix below for $\Deltaaa{{5}}{({3},(21))}$, which is 1... ]}\bajm{[I am now proceeding with the updated convention since this gives the correct matrix!]}\ppm{[agreed.  (I haven't got my head around the diff yet, but your way is not wrong.)]}
\medskip 
}}
\begin{align*}& \reflectbox{\rotatebox[]{180}{\onecuppppp{6}{1}{4}{2}{3}{5}}}=-\tfrac{1}{2},
&&\reflectbox{\rotatebox[]{180}{\onecuppppp{6}{1}{4}{2}{3}{5}}}=-\tfrac{1}{2},
&&\reflectbox{\rotatebox[]{180}{\onecuppppp{6}{1}{4}{2}{3}{5}}}=0,
\\[-0.822em]
& \ \onecupppppx{6}{1}{3}{2}{4}{5}
&&  \ \onecupppppx{6}{2}{4}{1}{3}{5}
&& \ \onecupppppx{6}{2}{3}{1}{4}{5}\\
&\reflectbox{\rotatebox[]{180}{\onecuppppp{6}{1}{4}{2}{3}{5}}}=-\tfrac{1}{2},
&&\reflectbox{\rotatebox[]{180}{\onecuppppp{6}{1}{4}{2}{3}{5}}}=-\tfrac{1}{2}.
\\[-0.822em]
& \ \onecupppppx{6}{3}{4}{1}{2}{5}
&& \ \onecupppppx{6}{4}{5}{1}{2}{3}
\end{align*}
\noindent 
We give the Gram matrix for $\Deltaaa{{5}}{({3},(21))}$ below, where we use the orthonormal basis as per \ref{eq:21ortho} for $\SS_{(2,1)}$ (in the same order as \eqref{eq:21basis} but with $(2 \ 3) \nC_{(2,1)}$ replaced by $\frac{2}{\sqrt{3}}((2 \ 3)+\frac{1}{2}e )\nC_{(2,1)}$): 

\[
\left(
\begin{array}{ccccccc ccccccc}
 \alpha  & 1 & 1 & 1 & 1 & 0 & 0 & 0 & 0 & 0 & 0 & 0 & 0 & 0 \\
 1 & \alpha  & 1 & 1 & 0 & 1 & 0 & 0 & 0 & 0 & 0 & 0 & 0 & 0 \\
 1 & 1 & \alpha  & 0 & 1 & 1 & -\frac{1}{2} & 0 & 0 & 0 & 0 & 0 & 0 & -\frac{\sqrt{3}}{2} \\
 1 & 1 & 0 & \alpha  & 1 & 1 & 0 & 0 & 0 & 0 & 0 & 0 & 0 & 0 \\
 1 & 0 & 1 & 1 & \alpha  & 1 & -\frac{1}{2} & 0 & 0 & 0 & 0 & 0 & 0 & \frac{\sqrt{3}}{2} \\
 0 & 1 & 1 & 1 & 1 & \alpha  & 1 & 0 & 0 & 0 & 0 & 0 & 0 & 0 \\
 0 & 0 & -\frac{1}{2} & 0 & -\frac{1}{2} & 1 & \alpha  & 0 & 0 & \frac{\sqrt{3}}{2} & 0 & \frac{\sqrt{3}}{2} & 0 & 0 
 \\ 
 0 & 0 & 0 & 0 & 0 & 0 & 0 & \alpha  & 1 & -1 & 1 & -1 & 0 & 0 \\
 0 & 0 & 0 & 0 & 0 & 0 & 0 & 1 & \alpha  & 1 & 1 & 0 & -1 & 0 \\
 0 & 0 & 0 & 0 & 0 & 0 & \frac{\sqrt{3}}{2} & -1 & 1 & \alpha  & 0 & 1 & -1 & -\frac{1}{2} \\
 0 & 0 & 0 & 0 & 0 & 0 & 0 & 1 & 1 & 0 & \alpha  & 1 & 1 & 0 \\
 0 & 0 & 0 & 0 & 0 & 0 & \frac{\sqrt{3}}{2} & -1 & 0 & 1 & 1 & \alpha  & 1 & \frac{1}{2} \\
 0 & 0 & 0 & 0 & 0 & 0 & 0 & 0 & -1 & -1 & 1 & 1 & \alpha  & 1 \\
 0 & 0 & -\frac{\sqrt{3}}{2} & 0 & \frac{\sqrt{3}}{2} & 0 & 0 & 0 & 0 & -\frac{1}{2} & 0 & \frac{1}{2} & 1 & \alpha 
\end{array}
\right)
\]
By direct computation, we find:
\begin{align*}\Delttt{5}{3,(21)}&= (\alpha -2)^3 \alpha  (\alpha +2) (\alpha +4)(\alpha ^4-7 \alpha ^2+3)^2,\\
\Delttt{6}{4,(21)}&=(\alpha -2)^3 \alpha  (\alpha +2) (\alpha +4)(\alpha -1)^2 \alpha^2(\alpha +1)^2 \left(\alpha ^2-7\right)^2.
\end{align*}

\ignore{\bajm{[-------------------------I would propose to end \S \ref{ss:sdmods} here-----------------------]}

Gram matrix for $\Deltaaa{7}{(5,(21))}$: 
\[
\left(
\begin{array}{ccccccc|cc|ccccccc|cc}
 \alpha  & 1 & 1 & 1 & 1 & 0 & 0 & 0 & 0 & 0 & 0 & 0 & 0 & 0 & 0 & 0 & 0 & 0 \\
 1 & \alpha  & 1 & 1 & 0 & 1 & 0 & 0 & 0 & 0 & 0 & 0 & 0 & 0 & 0 & 0 & 0 & 0 \\
 1 & 1 & \alpha  & 0 & 1 & 1 & -\frac{1}{2} & 0 & 0 & 0 & 0 & 0 & 0 & 0 & 0 & -\frac{\sqrt{3}}{2} & 0 & 0 \\
 1 & 1 & 0 & \alpha  & 1 & 1 & 0 & 0 & 0 & 0 & 0 & 0 & 0 & 0 & 0 & 0 & 0 & 0 \\
 1 & 0 & 1 & 1 & \alpha  & 1 & -\frac{1}{2} & 0 & 0 & 0 & 0 & 0 & 0 & 0 & 0 & \frac{\sqrt{3}}{2} & 0 & 0 \\
 0 & 1 & 1 & 1 & 1 & \alpha  & 1 & 0 & 0 & 0 & 0 & 0 & 0 & 0 & 0 & 0 & 0 & 0 \\
 0 & 0 & -\frac{1}{2} & 0 & -\frac{1}{2} & 1 & \alpha  & 1 & 0 & 0 & 0 & \frac{\sqrt{3}}{2} & 0 & \frac{\sqrt{3}}{2} & 0 & 0 & 0 & 0 \\ \hline 
 0 & 0 & 0 & 0 & 0 & 0 & 1 & \alpha  & 1 & 0 & 0 & 0 & 0 & 0 & 0 & 0 & 0 & 0 \\
 0 & 0 & 0 & 0 & 0 & 0 & 0 & 1 & \alpha  & 0 & 0 & 0 & 0 & 0 & 0 & 0 & 0 & 0 \\ \hline 
 0 & 0 & 0 & 0 & 0 & 0 & 0 & 0 & 0 & \alpha  & 1 & -1 & 1 & -1 & 0 & 0 & 0 & 0 \\
 0 & 0 & 0 & 0 & 0 & 0 & 0 & 0 & 0 & 1 & \alpha  & 1 & 1 & 0 & -1 & 0 & 0 & 0 \\
 0 & 0 & 0 & 0 & 0 & 0 & \frac{\sqrt{3}}{2} & 0 & 0 & -1 & 1 & \alpha  & 0 & 1 & -1 & -\frac{1}{2} & 0 & 0 \\
 0 & 0 & 0 & 0 & 0 & 0 & 0 & 0 & 0 & 1 & 1 & 0 & \alpha  & 1 & 1 & 0 & 0 & 0 \\
 0 & 0 & 0 & 0 & 0 & 0 & \frac{\sqrt{3}}{2} & 0 & 0 & -1 & 0 & 1 & 1 & \alpha  & 1 & \frac{1}{2} & 0 & 0 \\
 0 & 0 & 0 & 0 & 0 & 0 & 0 & 0 & 0 & 0 & -1 & -1 & 1 & 1 & \alpha  & 1 & 0 & 0 \\ 
 0 & 0 & -\frac{\sqrt{3}}{2} & 0 & \frac{\sqrt{3}}{2} & 0 & 0 & 0 & 0 & 0 & 0 & -\frac{1}{2} & 0 & \frac{1}{2} & 1 & \alpha  & 1 & 0 \\ \hline 
 0 & 0 & 0 & 0 & 0 & 0 & 0 & 0 & 0 & 0 & 0 & 0 & 0 & 0 & 0 & 1 & \alpha  & 1 \\
 0 & 0 & 0 & 0 & 0 & 0 & 0 & 0 & 0 & 0 & 0 & 0 & 0 & 0 & 0 & 0 & 1 & \alpha  \\
\end{array}
\right)\]

Gram matrix for $\Deltaaa{8}{(6,(21))}$:
\[ \hspace{-1cm} 
\left(
\begin{array}{ccccccc|ccc|ccccccc|ccc}
 \alpha  & 1 & 1 & 1 & 1 & 0 & 0 & 0 & 0 & 0 & 0 & 0 & 0 & 0 & 0 & 0 & 0 & 0 & 0 & 0 \\
 1 & \alpha  & 1 & 1 & 0 & 1 & 0 & 0 & 0 & 0 & 0 & 0 & 0 & 0 & 0 & 0 & 0 & 0 & 0 & 0 \\
 1 & 1 & \alpha  & 0 & 1 & 1 & -\frac{1}{2} & 0 & 0 & 0 & 0 & 0 & 0 & 0 & 0 & 0 & -\frac{\sqrt{3}}{2} & 0 & 0 & 0 \\
 1 & 1 & 0 & \alpha  & 1 & 1 & 0 & 0 & 0 & 0 & 0 & 0 & 0 & 0 & 0 & 0 & 0 & 0 & 0 & 0 \\
 1 & 0 & 1 & 1 & \alpha  & 1 & -\frac{1}{2} & 0 & 0 & 0 & 0 & 0 & 0 & 0 & 0 & 0 & \frac{\sqrt{3}}{2} & 0 & 0 & 0 \\
 0 & 1 & 1 & 1 & 1 & \alpha  & 1 & 0 & 0 & 0 & 0 & 0 & 0 & 0 & 0 & 0 & 0 & 0 & 0 & 0 \\
 0 & 0 & -\frac{1}{2} & 0 & -\frac{1}{2} & 1 & \alpha  & 1 & 0 & 0 & 0 & 0 & \frac{\sqrt{3}}{2} & 0 & \frac{\sqrt{3}}{2} & 0 & 0 & 0 & 0 & 0 
 \\  \hline  
 0 & 0 & 0 & 0 & 0 & 0 & 1 & \alpha  & 1 & 0 & 0 & 0 & 0 & 0 & 0 & 0 & 0 & 0 & 0 & 0 
 \\
 0 & 0 & 0 & 0 & 0 & 0 & 0 & 1 & \alpha  & 1 & 0 & 0 & 0 & 0 & 0 & 0 & 0 & 0 & 0 & 0 
 \\
 0 & 0 & 0 & 0 & 0 & 0 & 0 & 0 & 1 & \alpha  & 0 & 0 & 0 & 0 & 0 & 0 & 0 & 0 & 0 & 0 
 \\  \hline 
 0 & 0 & 0 & 0 & 0 & 0 & 0 & 0 & 0 & 0 & \alpha  & 1 & -1 & 1 & -1 & 0 & 0 & 0 & 0 & 0 \\
 0 & 0 & 0 & 0 & 0 & 0 & 0 & 0 & 0 & 0 & 1 & \alpha  & 1 & 1 & 0 & -1 & 0 & 0 & 0 & 0 \\
 0 & 0 & 0 & 0 & 0 & 0 & \frac{\sqrt{3}}{2} & 0 & 0 & 0 & -1 & 1 & \alpha  & 0 & 1 & -1 & -\frac{1}{2} & 0 & 0 & 0 \\
 0 & 0 & 0 & 0 & 0 & 0 & 0 & 0 & 0 & 0 & 1 & 1 & 0 & \alpha  & 1 & 1 & 0 & 0 & 0 & 0 \\
 0 & 0 & 0 & 0 & 0 & 0 & \frac{\sqrt{3}}{2} & 0 & 0 & 0 & -1 & 0 & 1 & 1 & \alpha  & 1 & \frac{1}{2} & 0 & 0 & 0 \\
 0 & 0 & 0 & 0 & 0 & 0 & 0 & 0 & 0 & 0 & 0 & -1 & -1 & 1 & 1 & \alpha  & 1 & 0 & 0 & 0 \\
 0 & 0 & -\frac{\sqrt{3}}{2} & 0 & \frac{\sqrt{3}}{2} & 0 & 0 & 0 & 0 & 0 & 0 & 0 & -\frac{1}{2} & 0 & \frac{1}{2} & 1 & \alpha  & 1 & 0 & 0 
 \\ \hline 
 0 & 0 & 0 & 0 & 0 & 0 & 0 & 0 & 0 & 0 & 0 & 0 & 0 & 0 & 0 & 0 & 1 & \alpha  & 1 & 0 \\
 0 & 0 & 0 & 0 & 0 & 0 & 0 & 0 & 0 & 0 & 0 & 0 & 0 & 0 & 0 & 0 & 0 & 1 & \alpha  & 1 \\
 0 & 0 & 0 & 0 & 0 & 0 & 0 & 0 & 0 & 0 & 0 & 0 & 0 & 0 & 0 & 0 & 0 & 0 & 1 & \alpha  \\
\end{array}
\right)
\]

Gram matrix for $\Deltaaa{9}{(7,(21))}$:
\[
\hspace{-.2in}
\left(
\begin{array}{ccccccc|cccc|ccccccc|cccc}
 \alpha  & 1 & 1 & 1 & 1 & 0 & 0 & 0 & 0 & 0 & 0 & 0 & 0 & 0 & 0 & 0 & 0 & 0 & 0 & 0 & 0 & 0 \\
 1 & \alpha  & 1 & 1 & 0 & 1 & 0 & 0 & 0 & 0 & 0 & 0 & 0 & 0 & 0 & 0 & 0 & 0 & 0 & 0 & 0 & 0 \\
 1 & 1 & \alpha  & 0 & 1 & 1 & -\frac{1}{2} & 0 & 0 & 0 & 0 & 0 & 0 & 0 & 0 & 0 & 0 & -\frac{\sqrt{3}}{2} & 0 & 0 & 0 & 0 \\
 1 & 1 & 0 & \alpha  & 1 & 1 & 0 & 0 & 0 & 0 & 0 & 0 & 0 & 0 & 0 & 0 & 0 & 0 & 0 & 0 & 0 & 0 \\
 1 & 0 & 1 & 1 & \alpha  & 1 & -\frac{1}{2} & 0 & 0 & 0 & 0 & 0 & 0 & 0 & 0 & 0 & 0 & \frac{\sqrt{3}}{2} & 0 & 0 & 0 & 0 \\
 0 & 1 & 1 & 1 & 1 & \alpha  & 1 & 0 & 0 & 0 & 0 & 0 & 0 & 0 & 0 & 0 & 0 & 0 & 0 & 0 & 0 & 0 \\
 0 & 0 & -\frac{1}{2} & 0 & -\frac{1}{2} & 1 & \alpha  & 1 & 0 & 0 & 0 & 0 & 0 & \frac{\sqrt{3}}{2} & 0 & \frac{\sqrt{3}}{2} & 0 & 0 & 0 & 0 & 0 & 0 
 \\  \hline 
 0 & 0 & 0 & 0 & 0 & 0 & 1 & \alpha  & 1 & 0 & 0 & 0 & 0 & 0 & 0 & 0 & 0 & 0 & 0 & 0 & 0 & 0 \\
 0 & 0 & 0 & 0 & 0 & 0 & 0 & 1 & \alpha  & 1 & 0 & 0 & 0 & 0 & 0 & 0 & 0 & 0 & 0 & 0 & 0 & 0 \\
 0 & 0 & 0 & 0 & 0 & 0 & 0 & 0 & 1 & \alpha  & 1 & 0 & 0 & 0 & 0 & 0 & 0 & 0 & 0 & 0 & 0 & 0 \\
 0 & 0 & 0 & 0 & 0 & 0 & 0 & 0 & 0 & 1 & \alpha  & 0 & 0 & 0 & 0 & 0 & 0 & 0 & 0 & 0 & 0 & 0 
 \\ \hline 
 0 & 0 & 0 & 0 & 0 & 0 & 0 & 0 & 0 & 0 & 0 & \alpha  & 1 & -1 & 1 & -1 & 0 & 0 & 0 & 0 & 0 & 0 \\
 0 & 0 & 0 & 0 & 0 & 0 & 0 & 0 & 0 & 0 & 0 & 1 & \alpha  & 1 & 1 & 0 & -1 & 0 & 0 & 0 & 0 & 0 \\
 0 & 0 & 0 & 0 & 0 & 0 & \frac{\sqrt{3}}{2} & 0 & 0 & 0 & 0 & -1 & 1 & \alpha  & 0 & 1 & -1 & -\frac{1}{2} & 0 & 0 & 0 & 0 \\
 0 & 0 & 0 & 0 & 0 & 0 & 0 & 0 & 0 & 0 & 0 & 1 & 1 & 0 & \alpha  & 1 & 1 & 0 & 0 & 0 & 0 & 0 \\
 0 & 0 & 0 & 0 & 0 & 0 & \frac{\sqrt{3}}{2} & 0 & 0 & 0 & 0 & -1 & 0 & 1 & 1 & \alpha  & 1 & \frac{1}{2} & 0 & 0 & 0 & 0 \\
 0 & 0 & 0 & 0 & 0 & 0 & 0 & 0 & 0 & 0 & 0 & 0 & -1 & -1 & 1 & 1 & \alpha  & 1 & 0 & 0 & 0 & 0 \\
 0 & 0 & -\frac{\sqrt{3}}{2} & 0 & \frac{\sqrt{3}}{2} & 0 & 0 & 0 & 0 & 0 & 0 & 0 & 0 & -\frac{1}{2} & 0 & \frac{1}{2} & 1 & \alpha  & 1 & 0 & 0 & 0 
 \\ \hline  
 0 & 0 & 0 & 0 & 0 & 0 & 0 & 0 & 0 & 0 & 0 & 0 & 0 & 0 & 0 & 0 & 0 & 1 & \alpha  & 1 & 0 & 0 \\
 0 & 0 & 0 & 0 & 0 & 0 & 0 & 0 & 0 & 0 & 0 & 0 & 0 & 0 & 0 & 0 & 0 & 0 & 1 & \alpha  & 1 & 0 \\
 0 & 0 & 0 & 0 & 0 & 0 & 0 & 0 & 0 & 0 & 0 & 0 & 0 & 0 & 0 & 0 & 0 & 0 & 0 & 1 & \alpha  & 1 \\
 0 & 0 & 0 & 0 & 0 & 0 & 0 & 0 & 0 & 0 & 0 & 0 & 0 & 0 & 0 & 0 & 0 & 0 & 0 & 0 & 1 & \alpha  \\
\end{array}
\right)\]

\bajm{[I want to write some determinants here too]}
\ignore{{   
\[\left(
\begin{array}{cccccccccccccccccccccc}
 \alpha  & 1 & 1 & 1 & 1 & 0 & 0 & 0 & 0 & 0 & 0 & 0 & 0 & 0 & 0 & 0 & 0 & 0 & 0 & 0 & 0 & 0 \\
 1 & \alpha  & 1 & 1 & 0 & 1 & 0 & 0 & 0 & 0 & 0 & 0 & 0 & 0 & 0 & 0 & 0 & 0 & 0 & 0 & 0 & 0 \\
 1 & 1 & \alpha  & 0 & 1 & 1 & -\frac{1}{2} & 0 & 0 & 0 & 0 & 0 & 0 & 0 & 0 & 0 & 0 & -\frac{\sqrt{3}}{2} & 0 & 0 & 0 & 0 \\
 1 & 1 & 0 & \alpha  & 1 & 1 & 0 & 0 & 0 & 0 & 0 & 0 & 0 & 0 & 0 & 0 & 0 & 0 & 0 & 0 & 0 & 0 \\
 1 & 0 & 1 & 1 & \alpha  & 1 & -\frac{1}{2} & 0 & 0 & 0 & 0 & 0 & 0 & 0 & 0 & 0 & 0 & \frac{\sqrt{3}}{2} & 0 & 0 & 0 & 0 \\
 0 & 1 & 1 & 1 & 1 & \alpha  & 1 & 0 & 0 & 0 & 0 & 0 & 0 & 0 & 0 & 0 & 0 & 0 & 0 & 0 & 0 & 0 \\
 0 & 0 & -\frac{1}{2} & 0 & -\frac{1}{2} & 1 & \alpha  & 1 & 0 & 0 & 0 & 0 & 0 & \frac{\sqrt{3}}{2} & 0 & \frac{\sqrt{3}}{2} & 0 & 0 & 0 & 0 & 0 & 0 \\
 0 & 0 & 0 & 0 & 0 & 0 & 1 & \alpha  & 1 & 0 & 0 & 0 & 0 & 0 & 0 & 0 & 0 & 0 & 0 & 0 & 0 & 0 \\
 0 & 0 & 0 & 0 & 0 & 0 & 0 & 1 & \alpha  & 1 & 0 & 0 & 0 & 0 & 0 & 0 & 0 & 0 & 0 & 0 & 0 & 0 \\
 0 & 0 & 0 & 0 & 0 & 0 & 0 & 0 & 1 & \alpha  & 1 & 0 & 0 & 0 & 0 & 0 & 0 & 0 & 0 & 0 & 0 & 0 \\
 0 & 0 & 0 & 0 & 0 & 0 & 0 & 0 & 0 & 1 & \alpha  & 0 & 0 & 0 & 0 & 0 & 0 & 0 & 0 & 0 & 0 & 0 \\
 0 & 0 & 0 & 0 & 0 & 0 & 0 & 0 & 0 & 0 & 0 & \alpha  & 1 & -1 & 1 & -1 & 0 & 0 & 0 & 0 & 0 & 0 \\
 0 & 0 & 0 & 0 & 0 & 0 & 0 & 0 & 0 & 0 & 0 & 1 & \alpha  & 1 & 1 & 0 & -1 & 0 & 0 & 0 & 0 & 0 \\
 0 & 0 & 0 & 0 & 0 & 0 & \frac{\sqrt{3}}{2} & 0 & 0 & 0 & 0 & -1 & 1 & \alpha  & 0 & 1 & -1 & -\frac{1}{2} & 0 & 0 & 0 & 0 \\
 0 & 0 & 0 & 0 & 0 & 0 & 0 & 0 & 0 & 0 & 0 & 1 & 1 & 0 & \alpha  & 1 & 1 & 0 & 0 & 0 & 0 & 0 \\
 0 & 0 & 0 & 0 & 0 & 0 & \frac{\sqrt{3}}{2} & 0 & 0 & 0 & 0 & -1 & 0 & 1 & 1 & \alpha  & 1 & \frac{1}{2} & 0 & 0 & 0 & 0 \\
 0 & 0 & 0 & 0 & 0 & 0 & 0 & 0 & 0 & 0 & 0 & 0 & -1 & -1 & 1 & 1 & \alpha  & 1 & 0 & 0 & 0 & 0 \\
 0 & 0 & -\frac{\sqrt{3}}{2} & 0 & \frac{\sqrt{3}}{2} & 0 & 0 & 0 & 0 & 0 & 0 & 0 & 0 & -\frac{1}{2} & 0 & \frac{1}{2} & 1 & \alpha  & 1 & 0 & 0 & 0 \\
 0 & 0 & 0 & 0 & 0 & 0 & 0 & 0 & 0 & 0 & 0 & 0 & 0 & 0 & 0 & 0 & 0 & 1 & \alpha  & 1 & 0 & 0 \\
 0 & 0 & 0 & 0 & 0 & 0 & 0 & 0 & 0 & 0 & 0 & 0 & 0 & 0 & 0 & 0 & 0 & 0 & 1 & \alpha  & 1 & 0 \\
 0 & 0 & 0 & 0 & 0 & 0 & 0 & 0 & 0 & 0 & 0 & 0 & 0 & 0 & 0 & 0 & 0 & 0 & 0 & 1 & \alpha  & 1 \\
 0 & 0 & 0 & 0 & 0 & 0 & 0 & 0 & 0 & 0 & 0 & 0 & 0 & 0 & 0 & 0 & 0 & 0 & 0 & 0 & 1 & \alpha  \\
\end{array}
\right)\]
}}}

\subsection{On Rollet graphs  }

\newcommand{\Roll}{{\mathcal{R}}}  
Recall that a tower of recollement \cite{ToR} 
is a sequence of 
inclusions of algebras $A_n \hookrightarrow A_{n+1}$ together with a sequence of idempotent inclusions $e A_{n+j} e \cong A_n$ for some $j$ - typically $j=2$ - so that there is a  global limit of representation categories - 
in our case 
along the functors from (\ref{de:globloc0}). (Recall also that this setup arises in various contexts. Originally it is the algebra analogue of the thermodynamic limit in corresponding statistical mechanical models; it is also an image of Lie geometric representation theory on the dual side when there is a Schur--Weyl duality involving a Lie structure; and it can also be seen as a non-semisimple generalisation of a Jones Basic Construction.) 
The global limit partitions into $j$ components - here indeed 2 components.  

In this section we recall the Rollet graphs of our towers of algebras, and their combinatorial connection 
to $\Deltaa$-modules. 

\mdef   \label{de:Lambdan}
Let us write $\Lambda_n$ for the set of integer partitions of $n$. 
(Which set thus indexes the complex-irreducible representations of $S_n$ up to isomorphism.)

\newcommand{\Lambdax}{\ppm{\Lambda}} 

\mdef   \label{def:indexset}
The usual \cite{KadarYu} index set for standard modules of $J_{l,n}$ is 
\[
\Lambda_{l,n} = \{  ( p, \l)\; | \; p \in \N_0, \; 
n-p \in 2\N_0 ,  
\; \l\vdash   \min(p,l+2) \}
\]
Observe that $\Lambda_{l,n} \hookrightarrow \Lambda_{l,n+2}$. 
In particular, for large enough $n$ we have 
$$
\Lambda_{l,n+2} \; = \; \Lambda_{l,n} \cup \left(\Lambda_{l+2}\times \{n+2\}\right)  .
$$
We define $\Lambda_{l,\infty}=\cup_{n=1}^\infty \Lambda_{l,n}$. 

\mdef  Our algebras are non-semisimple in general (and indeed in the cases we are interested in). But each forms a Brauer modular system with the triangle of commutative rings (integral - field of fractions - modular) taken as follows. 
The integral ring is complex polynomials in indeterminate $\alpha$ (this can be lifted to consider integral polynomials and a double modularity, but we will stay with the complex case here). The field of fractions is the rational \bajmm{functions}, $\C(\a)$. 
And the modular ring is the complex field obtained from the integral by evaluating $\alpha$ at a given complex number. 
We note that our algebras are semisimple over the field of fractions; and hence generically. Thus in particular the (integrally defined) $\Delta$-modules are generically simple; and are a basis for the Grothendieck group in general. 
In this sense our tower has a Bratteli diagram for restriction of $\Delta$-modules 
(indicating filtration/character factors rather than direct summands in general).

\mdef   \label{de:Rollet} 
Recall \cite{ToR} that a tower of recollement (ToR) is a tower of algebras related both by inclusion and by idempotent localisation; and that the corresponding Rollet graph is the projection of the Bratelli diagram  with respect to the localisation. 
Let us write $\Roll_{l}$ for the Rollet graph (projection of  Bratelli diagram)  of the tower of algebras $J_{l,-} = ( J_{l,n})_{n \in \N_0}$ with respect to the usual idempotent localisation from (\ref{de:globloc0}).

The vertex set of the graph $\Roll_l$ is the set of pairs $(p,\l)  \in \Lambda_{l,\infty}$. 
There is an edge $(p,\l)-(p', \l')$ if and only if (up to interchanging the pairs) $p'=p+1$, and either $\l'$ can be obtained from $\l$ by adding a box, or $\l'=\l$.


\medskip 

The `Rollet diagram' is simply the Rollet graph drawn with vertex weights 
$p$ 
increasing 
from left to right; and then writing simply $\l$ for $(p,\l)$ (since the $p$ can now be recovered from context). 

\newcommand{\arm}{arm}
\newcommand{\hip}{shoulder}

\mdef \label{pa:rollarms}
It will be convenient to refer to the subgraph that agrees with the Young graph as the {\em head} of the graph; 
the remainder of the graph as the {\em \arm s} of the graph, 
thus a $\l$-arm for each $\l\vdash l+2$;
and the last vertices in the head - that are the beginnings of the arms - as the {\em \hip s} of the graph.

\mdef Example. 
The  
Rollet diagrams  
$\Roll_{-1}$ for the tower $J_{-1,n}$,  
$\Roll_1$ for the tower $J_{1,n}$ 
and
$\Roll_2$ for the tower $J_{2,n}$ 
are given in Fig.\ref{fig:R2}

\newcommand{\boxx}{\!\!\!\!\qed\;}  

\begin{figure}
$$
\emptyset - \boxx -\boxx -\boxx - \boxx - \boxx - ...
$$
\vspace{.91cm}

\includegraphics[width=9cm]{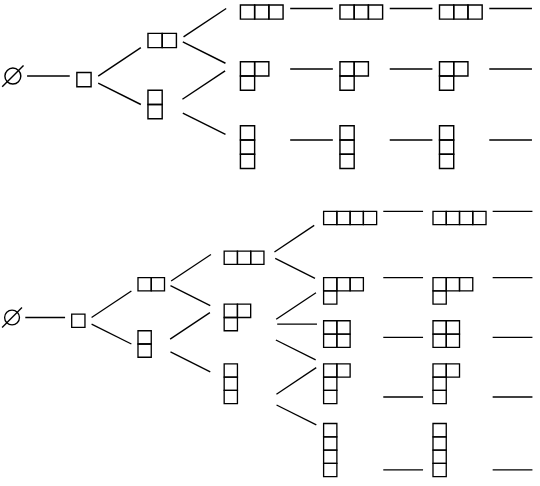}

\caption{\label{fig:R2}  The Rollet graphs/diagrams  $\Roll_{-1}$, 
$\Roll_1$ and   $\Roll_2$.  
(See \ref{ss:gdetl0} for $\Roll_0$.)
\\
In $\Roll_2$, for example, the head of the graph is the first five columns; 
with the \hip\ being the 5th column; and columns five and above being the \arm s. 
}
\end{figure} 



\medskip \medskip 

The Rollet graph encodes standard module restriction rules - specifically filtration factors. 
Thus in particular we have the following. 

\mdef   \label{pr:walks0}
The dimension of the standard module of $J_{l,n}$ with label 
$v=(p,\lambda)$ is given by the number of walks of length $n$ from $\emptyset$ to $v$ 
on the graph $\Roll_l$.

\mdef 
\ignore{{ 
For completeness (and also to facilitate the making of a useful point) let us include the case 
$\Roll_{-1}$, 
for $J_{-1,n}$:

\medskip 

\beq 
\emptyset - \boxx -\boxx -\boxx - ...
\eq 
\medskip

\noindent 
}}
In the $l=-1$ case, $p$ determines $(p,\lambda)$, so we can use just $p$ for the label.
Walks of length $n$ from 0 (i.e. $(0,\emptyset)$) to $p$  (i.e. $(p,\boxx)$)
on this $\Roll_{-1}$ graph can be used to give a basis for the corresponding 
Temperley--Lieb standard module (indeed there are multiple distinct useful such bases). 
From this, or directly, we can give a corresponding matrix representation from which, in turn, we can start to see the role of alcove geometry in this setting (alcove geometry, particularly type affine-$A_1$ --- affine reflection groups on the line, arises naturally in Lie representation theory and hence can be considered to arise here via Schur--Weyl duality, but we can also understand it directly). 

The construction is as follows  \cite{Martin91}. 
First we consider our graph as embedded in $\R$ so that the vertices are a subset of the integral points:
then walks on the embedded graph from 0 to $\l$ are put in correspondence with bases (in various ways - see \S\ref{ss:gdTL}) for the $\l$ standard module. 
In a suitable bases there is a decoupling of walks exceeding a geometrical bound when $\xx$ is the root of a Chebyshev polynomial - a submodule, and hence giving a map in from another standard module, say $\l'$. The map in can be organised as the label $\l'$ being the image of $\l$ under the action of a suitable realisation (depending on $\xx$) of the affine reflection group. See \S\ref{ss:repthy} for visualisations. 



\ignore{{
\mdef  \ppm{[DELETE this para, or at least delete ref to C1, for now.]}
Later, in \S\ref{ss:walkbases0}, 
we will describe bases for $\Deltaa$-modules in terms of walks on Rollet graphs,
and hence more `geometrical' than diagram bases. And hence potentially better suited to extending the corresponding known bases in case $l=-1$.  
}}



\subsubsection{Gram-determinant-decorated Rollet graph - preliminaries}
As noted in the Introduction (see 
\S\ref{ss:intro}),  
our first main aim is 
to determine the fundamental invariants of Kadar--Yu representation theory, i.e. the Cartan decomposition matrices. 
This is not easy!%
\footnote{(In the 20th century, Brauer posed the corresponding problem for the symmetric groups over fields of finite characteristic and then, coming to suspect that this problem is impossible in general, he posed the problem for the Brauer algebras over the complex field - a problem which  then remained open for over 50 years, and was only solved in the 21st century. And, viewed holistically, the Kadar-Yu problem is certainly harder.)} 
So we have some intermediate aims. 
In particular, partial data on decomposition matrices is given by morphisms from costandard to standard modules. 
And partial data on these is given 
by the corresponding gram matrices 
$\Gamma(\Deltaa^{n}_{(\lambda,p)})$
of contravariant forms (in some suitable basis). 
And finally just by the determinants:
\beq
 \Delttt{n}{p,\lambda}  \; := \;  \det(\Gamma(\Deltaa^{n}_{(\lambda,p)}))\label{eq:detnot}
\eq

Even the gram determinants represent a huge swathe of complex data, raising the practical problem 
of how to organise and present what is known. 

\medskip 


\newcommand{\V}{{\mathcal V}}
\newcommand{\RollV}{\Roll^{{\mathcal V}}}
\newcommand{\RollD}{\Roll^{{\Deltaa}}}

In general a `gram-det-decorated' Rollet graph 
$\RollD_l$ 
is a Rollet graph $\Roll_l$ 
as in \ref{de:Rollet}
with the fibre of gram determinants attached to each vertex. 
And we shall include results presented in this way.
But this is quite information heavy. 
We will see from the $l=-1$ case that there the data may be simplified enormously by only attaching (for given~$l$)
the fibre of `marginal vertex functions' defined as follows. 
The marginal vertex function is 
%
\beq  \label{eq:MVF1}
\V^{\bajmm{n}}_{(p.\lambda)} \; \bajmm{:}= \; 
\frac{ \Delttt{(n=p+2m)}{p,\l} }{ \prod_{{v\; adjacent\; \atop to \; (p,\lambda)\; in\;\Roll_l }} \Delttt{(n-1)}{v}}
\eq 
- recall that $p$ is the number of propagating lines.  
In the new decorated graph the fibre of determinants is replaced by this $\V$ as $m$ varies - if indeed it depends on $m$.
\ppmm{So we will also include some results presented in this way.} 
Let us write $\RollV_l$ for this particular decorated Rollet graph. 

\mdef   \label{de:Roll-etc}
Finally, and even less information-heavy, we may also present a certain section through $\RollV_l$,
corresponding to fixed $m$ 
(fixed number of `cups' in the diagrams of the diagram basis) - which we denote by $\Roll^{(m)}_l$.

The corresponding section through $\RollD_l$ is denoted $\Roll^{m}_l$. 

\medskip 

Observe that if we can determine the marginal vertex functions then we can determine the \bajmm{G}ram determinants by a recursion. 
So a next step is to study the marginal functions. 
In order to give them in closed form it will be convenient to use Chebyshev polynomials. 
And in particular this brings us to the `bootstrap' one-cup cases. 
So in 
\S\ref{ss:onecup0} and in particular in 
\S\ref{ss:cheby0} we collect the manipulations that we will need. 

Then in \S\ref{ss:e-next!} we will address the general case. 

\section{Gram-determinants of one cup standard modules}  \label{s:onecup0}

In this section, we consider the gram determinants of the $J_{l,n}$ standard modules $\Deltaa_{(n-2,\l)}^{n}$ for $n\geq l+4$. The main result is Theorem \ref{thm:main} in \S\ref{ss:non1dsp} which allows us to obtain closed forms for $\Deltaa_{(n-2,\l)}^{n}$ as a series in $n$ by computing the first two cases (\textit{i.e.} $n=l+4,l+5$); these cases uniquely determine a Chebyshev series of polynomials, $P^{(\l)}_{n}$, with which we can express each determinant. We apply this theorem to compute closed forms of this determinant for all $\l\vdash l+2$ with $l\leq 2$ and additionally in the cases $\l=(l+1,1), (l+1,1)^T$ for $l\leq 7$. The polynomials $P^{(\l)}_{n}$ exhibit some remarkable properties which we begin to document in \S\ref{ss:examples0}. We study the distribution of roots of these polynomials in \S\ref{s:roots}.

\subsection{On Chebyshev polynomials (for gram determinants)}  \label{ss:cheby0}

\bajmm{We say that a series of functions, $P=\{ P_n\}_{n\in \Z}$, is a ``Chebyshev series" if it satisfies the Chebyshev recursion:}
\beq \label{eq:Cheby0} 
P_{n+1}(x) = x P_{n}(x) +P_{n-1} (x)   ,
\eq  
(see e.g. \cite[\S6.3.3]{Martin91} and references therein; cf. also e.g. \cite{KarlinStudden}).
This means of course that the series is 
determined in principle by any two successive values, $P_i$ and $P_{i+1}$. 
Classically these would be with $i=0$, as we will recall shortly. 
Two series $P,P'$ are `parallel' if $P_n = P'_{n+i}$ for some $i$. 


\mdef  \label{pr:cheb1}
Observe that if any two successive functions $P_i$ and $P_{i+1}$ are (integral) polynomial then all $P_n$ are. If in addition $\deg(P_{i+1})=\deg(P_i) + 1$ \ppmm{for some $i$,} then $\deg(P_{k+1})=\deg(P_{k}) + 1$ for all $k>i$ 
and further, if in addition
$P_{i+1}$ is monic 
then so is $P_k$ for all $k>i$. 

\mdef We call a Chebyshev series of polynomials $P$ ``reduced" if there is no common non-trivial (\textit{i.e.} degree $>0$) polynomial factor 
among all $P_i$ (equivalently, among any two successive 
terms). 
\\ 
Any polynomial series $P$ is a (polynomial) multiple 
of a reduced one $P'$; $P'$ can be obtained from $P$ by removing a maximal common factor among two successive terms. 


\mdef  \label{de:Chebyab}
Let us write $\PP^{a,b} $ for the Chebyshev series with initial conditions 
$\PP^{a,b}_0 = a$ and $\PP^{a,b}_1 = b$.  

We may abuse notation slightly and, provided $i$ is fixed, write $\PP^{a,b} $ 
for the parallel series that starts with $\PP^{a,b}_i =a$. 

As we will see, explicit usage of the $\PP^{a,b}$ notation is often cumbersome in practice for general series (since $a$ may be a lengthy expression), and we will largely restrict the notation to abstract manipulations. Instead we use more specific notations for interesting cases, as in \S\ref{ss:onecup0}.

\mdef 
The classical initial conditions are $P_0(x) =0$ and $P_1(x) =1$ giving then 
$P_2(x) =x $; $P_3(x) = x^2-1$; $P_4(x) = x(x^2 -2)$; and so on. In our $\PP^{a,b}$ notation this classical series is $\PP^{0,1}$. 

\mdef 
For convenience we set $P^U:=\PP^{0,1}$ since this series is related to the classical series of ``U-type" Chebyshev polynomials, $U$, which can be defined by the identity
\beq U_n(\cos(\th))\sin(\th)= \sin((n+1)\th),\label{eq:trig}\eq
by $P^U_n(x)=U_{n-1}(x/2)$. The polynomials $U_n$ and $P^U_n$ are well defined for $n<0$ with the above formulae still holding here; indeed, one can use these to determine
\beq P^U_{-n}(x)=-P^U_{n}(x),\label{eq:negcheb} \eq
for any $n\in \Z$. 

\mdef   \label{pr:Cbasis} Note that $\{ P^U_1(x),P^U_2(x),\dots \}$ is a $\C$-basis for the polynomial ring $\C[x]$.

\ignore{{ 
\newcommand{\ptarmigan}{Cheby}
\newcommand{\ramping}{ramping}
\newcommand{\Ramping}{Ramping}
}} 

\subsubsection{The \Ramping\ Property}


Here we discuss a property that Chebyshev series may 
achieve eventually. 

\newcommand{\rampsfrom}{ramps from}   

\mdef\label{def:chebycond} 
Let $P$ be a 
Chebyshev series. For some integer $N$, we say that the series $P$ 
has the ``\ramping\ property 
for $n\geq N$'' if it is polynomial and reduced, and if, furthermore, for $n\geq N$, $P_n$ is monic and $\deg(P_{n+1})=\deg(P_{n})+1$. 
Such a series is uniquely determined by $N$ and the ratio $P_{N+1}(x)/P_{N}(x)$, a rational function.


\mdef \label{pa:chebslv} 
Suppose that a series of polynomials $P$ has the \ramping\ property for $n\geq N$, and suppose that $d=\deg(P_N)$. 
Then there exist unique coefficients $a_k, b_i, b'_j \in \C$, 
such that
\begin{align*}
    P_N(x)&=\sum_{k=-(d+1)}^{d+1} a_{k} P^U_{k}(x)=\sum_{i=1}^{d+1} b_i P^U_{i}(x),\\
    P_{N+1}(x)&=\sum_{k=-(d+1)}^{d+1} a_{k} P^U_{k+1}(x)=\sum_{j=1}^{d+2} b'_j P^U_{j}(x).\end{align*}
The $a_k$ can be determined as follows: first solve for the coefficients $b_i$ and $b_j'$ - these are unique by \ref{pr:Cbasis}. Then by using \eqref{eq:negcheb} we obtain the following full rank linear system relating the $a_k$ to the $b_i$ and $b_j'$:
\begin{align*}
    b_k&=a_k-a_{-k}, & b'_{d+2}&=a_{d+1}, & b'_{d+1}&=a_{d}, &b'_{k'}&=a_{k'-1}-a_{-(k'+1)},
\end{align*}
for $k=1,\dots,d+1$ and $k'=1,\dots, d$. In particular, monicity of $P_N$ and $P_{N+1}$ imply that $a_{d+1}=1$ and $a_{-(d+1)}=0$.

Since the relation \eqref{eq:Cheby0} is linear, it thus follows that for any $l\in \Z$, we have
\beq 
P_{N+l}(x)=\sum_{k=-(d+1)}^{d+1} a_{k} P^U_{k+l}(x)=P^U_{d+1+l}(x)+\sum_{k=-d}^{d} a_{k} P^U_{k+l}(x). 
\eq

\ignore{
\mdef 
If we start with $P'_2(x)=2x$ and $P'_3(x) = (x-1)(x+2) \;$ 
(see e.g. \cite[\S5.2]{ShBr})
then  $P'_4(x) = x(x^2+x-4)$,
\beq   \label{eq:l=0+}
P'_5(x) = x^4 + x^3 -5x^2 -x +2, \qquad 
P'_6(x) =  x(x-1)(x^3 +2x^2 -4x -6)   \qquad 
\eq  
and so on. 
A better notation is perhaps $P'_n(x) = P^{(2)}_n(x)$, 
since this sequence is related to the $\lambda = (2)$ branch - as we will explain shortly. 
\\ \ppm{
Suppose now we start with $P''_2(x) = 2$ and $P''_3(x) = x-1$  
- a better notation here is $P''_n(x)  = P^{(1^2)}_n(x)$.
Then ... }

\mdef  \label{pa:Cheby1}
Here are tables of values for various initial conditions 
(we say `initial' conditions, but such Chebyshev sequences are determined by any two successive terms; other than for $P$, the starting point for $n$ values is not yet locked down):
\[
\begin{array}{c|c|c|ccc}
n\setminus  & P_n(x)        & P'_n =P^{(2)}_n      & P''_n = P^{(1^2)}_n \\
2      &  x                 & 2x               & 2 \\
3      & x^2 -1             & (x-1)(x+2)        & x-1  \\
4      & x(x^2-2)           &  x(x^2+x-4)      & (x+1)(x-2)  \\ 
5      & (x^2-x-1)(x^2+x-1) & x^4+x^3-5x^2 -x +2  & (x^3-x^2-3x+1)  \\
6       & x(x^2-1)(x^2-3)   & x(x-1)(x^3+2x^2-4x-6) & (x-1)(x^3-4x-2) 
\end{array}
\]
and continuing (with some initial conditions suggested by $l=1$, cf. \ref{pa:l1}):
\[ \hspace{-.861cm} 
\begin{array}{c|c|c|ccc}
n\setminus  & P^{(3)}_n             & P^{(21)}_n      &  P^{(1^3)}_n \\
1    &    &  3(x^2-1) & \\
2      &  3(x+2)             & x(x^2-4)               & 3 \\
3      &    x(x+4)            & (x^4-7x^2+3)         & x-2  \\
4      & (x+1)(x^2+3x-6)        &  x(x^2-1)(x^2-7)      & (x+1)(x-3)  \\ 
5      & x(x^3+4x^2-4x-10)        & (x^6-9x^4+14x^2-3)  & (x^3-2x^2-4x+2)  \\
6       & (x^5+4x^4-5x^3-14x^2+3x+6) & 
x \left(x^{6}-10 x^{4}+22 x^{2}-10\right)
& (x^4-2x^3-5x^2+4x+3) \\ 
7   &  \dots  & 
\left(x^2 -1\right) \left(x^{6}-10 x^{4}+21 x^{2}-3\right)
& \dots 
\end{array}
\]


\mdef 
In general there is a Chebyshev sequence for each `branch' for each $l$, 
hence for each $\lambda\vdash l+2$, 
thus a sequence $\PP^\lambda$ 
\ppm{defined specifically by ...
}

The first perspective on this is that the gram matrices on a branch are `related' by the Chebyshev recursion. This is true on the nose if we look at the slice through the Bratelli diagram corresponding to 1-cup modules $\Deltaaa{n}{(n-2,\lambda)}$ with $\lambda\in \{ (3), (1^3) \}$. 
... }

\subsection{Gram Determinants of One Cup Standard Modules}\label{ss:onecup0}
Here we consider the Gram determinants of the standard modules $\Deltaa_{(p, \l)}^{n}$ for $n=p+2$. 

\subsubsection{The one-cup gram matrix}    \label{ss:gram-mat}


As in \cite[Prop. 4.22]{KadarYu}, we may construct a basis for 
the left module 
$\Deltaaa{n}{(p, \l)}$ 
for given $l$ 
from a basis for the corresponding Specht module, $\mathcal{S}^{\l}$, and the set of half diagrams $J^{||}_{l}(n,p)$ (\textit{cf.} \ref{de:stdmods}). 
Indeed this holds for any $p$, but if $p=n-2$ as here then the latter half diagrams are completely determined by the position of a single ``cup" hence our referring to the $\Deltaa_{(p, \l)}^{n}$ as ``one cup modules". 

\mdef   \label{de:uu} 
For fixed $l$ and $n$, write $\uu_{ij}$, with $i<j$ for the half diagram $\uu_{ij}\in J_l^{||}(n,n-2)$ 
with a cup joining the sites $i$ and $j$.  
For example, when $n=7$ we have 
\[
\uu_{24} =  \onecupp{8}{2}{4}{1}{3}{5}{6}{7}      \hspace{.31cm} \in J_{1}^{||}(7,5). 
\]
Note that if $j=i+1$ then $ht(\uu_{ij})=-1$, otherwise $ht(\uu_{ij})=j-3$.


\newcommand{\BBB}{{\mathsf B}}  

\mdef 
Write $\BBB_{(l,n)}$ as shorthand for the set of half diagrams $J_{l}(n,n-2)$. One has that 
\beq
\BBB_{(l,n)}=\{\uu_{12},\uu_{23},\dots,\uu_{n-1,n}\}\cup \left(\bigcup_{k=3}^{\min{n,l+3}} \BBB_k\right), \label{cupbasis}
\eq
where $\BBB_k:=\{\uu_{j,k} \ | \ j=1,\dots, k-2\}$. 
Observe that $\BBB_k$ is precisely those one-cup half-diagrams with height $k-3$. 

\mdef  \label{pa:cuprec}
Observe for $n\geq l+3$, we have that $\BBB_{(l,n+1)}= \BBB_{(l,n)}\cup \{\uu_{n,n+1}\}$.

\mdef 
Fix $l$ and $\l \vdash l+2$. When $n\geq l+4$, consider a basis $\mathcal{B}=\{b_1,\dots, b_{d_{\l}}\}$ for 
$\mathcal{S}_{\l}$ regarded as one for $\Delta_{(n,\l)}^{n}$ (\textit{c.f.} \ref{pa:spbasis})
with $b_i=x_i c_{\l}$ for some $x_i\in \C \Sym_{l+2}\subset J_{l,n}$ and $c_{\l}$ as per \eqref{de:cl}. Then a basis for $\Deltaaa{n}{(n-2,\l)}{}$ can be obtained as 
\beq
\BBB^n_{(n-2,\l)} \; := \; \{ \uu_{jk}b_i \ | \ \uu_{jk} \in \BBB_{(l,n)}, \ b_i \in \mathcal{B}  \}. \label{eq:basis} \eq

\mdef \label{pa:gramrels}
For $n\geq l+4$, we now consider the form $\langle\_,\_\rangle$ on $\Deltaa^{n}_{(n-2,\l)}$. The value of the form $\langle \uu_{jk} b_i , \uu_{j'k'} b_{i'} \rangle$ is determined by the equation
 \[
 c_{\l }   x_{i'}^* \uu_{j'k'}^* \uu_{jk} x_i  c_{\l}= \langle b_i \uu_{jk}, b_{i'} \uu_{j'k'} \rangle c_{\l}.
\]
There are three cases to consider, namely $|\{j,k,j',k'\}|=2,3,4$. In the first case $ \uu_{j'k'}^*\uu_{jk}=\a e$, in the second $\uu_{j'k'}^* \uu_{jk} =\sigma$ for some $\sigma\in \Sym_{l+2}$. In the last case,  $\uu_{j'k'}^* \uu_{jk} $ has $p-2$ propagating lines. We thus find
\[ \langle  \uu_{jk} b_i ,\uu_{j'k'} b_{i'}  \rangle =\begin{cases}
    \a \langle x_i, x_{i'}\rangle, & |\{i,j,i',j'\}|=2,\\
    \langle \sigma x_i, x_{i'}\rangle, & |\{i,j,i',j'\}|=3, \text{ where } \sigma=\uu_{j'k'}^*\uu_{jk} ,\\
    0, & \text{otherwise.}
\end{cases} 
\]
where $\langle\_,\_\rangle$ also here denotes the form on $\Delta^{n}_{(n-2,\l)} c_{\l} \simeq \SS_\l$ (\textit{c.f} \ref{de:stdmods}). Passing to an orthonormal basis for $\SS_\l$, with respect to this form, we observe that $\a$ only appears on the leading diagonal of the matrix. In particular, $\Gamma(\Delta^{n}_{(n-2,\l)})$ now has the form
\beq \Gamma(\Delta^{n}_{(n-2,\l)})(\a)=\a \, \text{I} +\Gamma(\Delta^{n}_{(n-2,\l)})', \label{eq:adep} \eq
where $\Gamma(\Delta^{n}_{(n-2,\l)})'=\Gamma(\Delta^{n}_{(n-2,\l)})|_{\a=0}$ is a real symmetric matrix. As observed for example in \cite{KY2}, this implies that the form $\langle \_ , \_\rangle$ is non-degenerate for $\a$ away from the real line.

For $l+3<j<n$ basis elements of the form $\uu_{j,j+1} b_i $ 
only have non-zero pairing with those of the form $\uu_{j',j'+1} b_{i'}$ for $j'=j,j\pm 1$, where we may compute
\beq \langle \uu_{j,j+1} b_i  , \uu_{j',j'+1}b_{i'} \rangle = \delta_{ii'}(\a \delta_{jj'}+\delta_{j+1,j'}+\delta_{j-1,j'}). \label{cheby} \eq

Equation \eqref{cheby} is the source of an important property of our Gram matrices $\Gamma(\Deltaa^{n}_{(p,\lambda)})$. 

\subsubsection{When $\SS_\l$ is one dimensional}\label{ss:1dspecht}

The underlying theory for this case is treated in \cite{KadarYu,ShBr}. 
This section provides a review in preparation for the general case in \S\ref{ss:non1dsp}, 
and examples extended beyond those previously computed. 

\medskip 

Suppose firstly 
then, 
that $\SS_\l$ is one dimensional, \textit{i.e.} $\l=(p)$ or $(1^p)$, and so $\Deltaa_{(p,\l)}^{n}$ has basis 
$\{ \uu_{ij} c_{\l} \}_{ij}$ with $\uu_{ij}\in B_{(n,l)}$. 
Furthermore, it is naturally divided into two ``regimes" whereon the form $\langle \_ , \_ \rangle$ displays distinct behaviour; the first we term the ``head" consisting of basis elements $ \uu_{ij} c_{\l}$ with $j\leq l+4$, and the second the ``tail" consisting of those same elements with $j>l+4$. The gram matrix restricted to the head is a dense matrix, whereas for the tail it is tridiagonal (when appropriately ordered). Furthermore, the pairing between elements of the head and the tail only has one non-zero case, namely $\langle \uu_{l+3,l+4} c_{\l}, \uu_{l+4,l+5} c_{\l} \rangle=1$.

\mdef 
To illustrate this behaviour, we give the gram matrix for the $J_{1,13}$-module $\Deltaa^{(13)}_{(11,\lambda)}$, for $\l =(3)$ or $\l=(1^3)$ 
(by taking $\e=1$ or $\e=-1$, respectively)
below:
\beq \left(
\begin{array}{ccc ccc c| cccccccc}
 \alpha  & 1 & \e & 1 & \e & 0 & 0 & 0 & 0 & 0 & 0 & 0 & 0 & 0 & 0 \\
 1 & \alpha  & 1 & 1 & 0 & \e & 0 & 0 & 0 & 0 & 0 & 0 & 0 & 0 & 0 \\
 \e & 1 & \alpha  & 0 & 1 & \e & 1 & 0 & 0 & 0 & 0 & 0 & 0 & 0 & 0 \\
 1 & 1 & 0 & \alpha  & 1 & 1 & 0 & 0 & 0 & 0 & 0 & 0 & 0 & 0 & 0 \\
 \e & 0 & 1 & 1 & \alpha  & 1 & \e & 0 & 0 & 0 & 0 & 0 & 0 & 0 & 0 \\
 0 & \e & \e & 1 & 1 & \alpha  & 1 & 0 & 0 & 0 & 0 & 0 & 0 & 0 & 0 \\
 0 & 0 & 1 & 0 & \e & 1 & \alpha  & 1 & 0 & 0 & 0 & 0 & 0 & 0 & 0 \\\hline
 0 & 0 & 0 & 0 & 0 & 0 & 1 & \alpha  & 1 & 0 & 0 & 0 & 0 & 0 & 0 \\
 0 & 0 & 0 & 0 & 0 & 0 & 0 & 1 & \alpha  & 1 & 0 & 0 & 0 & 0 & 0 \\
 0 & 0 & 0 & 0 & 0 & 0 & 0 & 0 & 1 & \alpha  & 1 & 0 & 0 & 0 & 0 \\
 0 & 0 & 0 & 0 & 0 & 0 & 0 & 0 & 0 & 1 & \alpha  & 1 & 0 & 0 & 0 \\
 0 & 0 & 0 & 0 & 0 & 0 & 0 & 0 & 0 & 0 & 1 & \alpha  & 1 & 0 & 0 \\
 0 & 0 & 0 & 0 & 0 & 0 & 0 & 0 & 0 & 0 & 0 & 1 & \alpha  & 1 & 0 \\
 0 & 0 & 0 & 0 & 0 & 0 & 0 & 0 & 0 & 0 & 0 & 0 & 1 & \alpha  & 1 \\
 0 & 0 & 0 & 0 & 0 & 0 & 0 & 0 & 0 & 0 & 0 & 0 & 0 & 1 & \alpha  \\
\end{array}
\right).\eq
The basis $\uu_{ij} c_{\l}$ is ordered lexicographically, 
thus
\[
\hspace{-.5cm}
\onecupp{8}{1}{2}{3}{4}{5}{6}{7} ,
\onecupp{8}{1}{3}{2}{4}{5}{6}{7} ,  
\onecupp{8}{1}{4}{2}{3}{5}{6}{7} ,  
\onecupp{8}{2}{3}{1}{4}{5}{6}{7} , \cdots
\]
The lines in the matrix are separating the head and the tail.  

\mdef 
More generally, for any $n>l+4$, and $\l=(l+2), (1^{l+2})$, by ordering the basis in the same way the matrix $\Gamma(\Deltaa^{n+1}_{(n-1,\lambda)})$ has the form
\beq
\begin{pNiceArray}{cccc|cc}
\Block{4-4}<\large>{\Gamma(\Deltaa^{n-1}_{(n-3,\lambda)})} &&&& 0&0\\
&&&& \vdots & \vdots\\
&&&& 0 & 0 \\
&&&& 1 &0\\ \hline
0 & \dots &0 &1 & \a &1 \\
0 & \dots &0 &0 & 1 &\a \\
\end{pNiceArray}.
\eq


To compute the determinant $\Delttt{(n+1)}{n-1,\lambda}$ for such $\l$ then,
we can expand along the last row obtaining the recursive formula
\begin{align}\Delttt{(n+1)}{n-1,\lambda}(\a)&= \a  \ \left| \ \begin{NiceArray}{cccc|c}
\Block{4-4}<\large>{\Gamma(\Deltaa^{(n-1)}_{(n-3,\lambda)})} &&&& 0\\
&&&& \vdots \\
&&&& 0  \\
&&&& 1 \\ \hline
0 & \dots &0 &1 & \a 
\end{NiceArray}\right| \ - \  \left| \ \begin{NiceArray}{cccc|c}
\Block{4-4}<\large>{\Gamma(\Deltaa^{(n-1)}_{(n-3,\lambda)})} &&&& 0\\
&&&& \vdots \\
&&&& 0  \\
&&&& 0 \\ \hline
0 & \dots &0 &1  &1 \\
\end{NiceArray}\right | \nn \\ 
&=\a\, \Delttt{(n)}{n-2,\lambda}(\a)-\Delttt{(n-1)}{n-3,\lambda}(\a).\label{eq:detrec1}
\end{align}
That is, the determinants, $\Delttt{(n)}{n-2,\lambda}$, satisfy the Chebyshev recursion in $n$ \eqref{eq:Cheby0}. In other words, from \S\ref{ss:cheby0} we have
\[
\Delttt{(n)}{n-2,\lambda}   = \PP^{a,b}_n
\]
with initial conditions $a=\Delttt{(l+4)}{l+2,\lambda}$ and $b=\Delttt{(l+5)}{l+3,\lambda}$ at $n=l+4$ and $n=l+5$ respectively. 
We will apply this in \S\ref{ss:examples0} to give explicit formulas for all Gram determinants in the cases $\l=(l+2), (1^{l+2})$ for $0\leq l \leq 12$.

\ignore{\subsubsection{\ppm{(temporary sec. marker)} Applications and examples}
\bajm{[I agree this may be more natural]}
\ppm{[could move this subsec. to after Theorem to avoid inverted referencing. not sure.]} \ppm{[-could even integrate into the main such section now?]}

\mdef \label{pa:1dgdets}
For $l=0$  
direct computation yields
\begin{align*}
\Delttt{(4)}{2,(2)}&=(\alpha -1) \alpha  \left(\alpha ^2+\alpha -4\right),& \Delttt{(5)}{3,(2)}&=(\alpha -1) \left(\alpha ^4+\alpha ^3-5 \alpha ^2-\alpha +2\right),\\
    \Delttt{(4)}{2,(1^2)}&=(\alpha -2) (\alpha -1) (\alpha +1) (\alpha +2),&
   \Delttt{(5)}{3,(1^2)}&=(\alpha -1) (\alpha +2) \left(\alpha ^3-\alpha ^2-3 \alpha +1\right),
\end{align*}  
For $l=1$, we consider the row and column partitions $\l$ here. 
In these cases, direct computation yields
\begin{align*}
    \Delttt{(5)}{3,(3)}&=(\alpha -2)^2 \alpha ^2 (\alpha +1) \left(\alpha ^2+3 \alpha -6\right),& \Delttt{(6)}{4,(3)}&=(\alpha -2)^2 \alpha ^3 \left(\alpha ^3+4 \alpha ^2-4 \alpha -10\right),\\
    \Delttt{(5)}{3,(1^3)}&=(\alpha -3) (\alpha -2)^2 (\alpha +1) (\alpha +2)^3,&
    \Delttt{(6)}{4,(1^3)}&=(\alpha -2)^2 (\alpha +2)^3 \left(\alpha ^3-2 \alpha ^2-4 \alpha +2\right).
\end{align*}
Thus, using the notation of Theorem \ref{thm:main}, 
\ppm{[-at the mo this is a forward ref!]}
we have here $\Delttt{(n)}{n-2,\l}=C^{(\l)}(\a) P^{(\l)}_{n}(\a)$ where
\begin{align*}
    C^{(2)}(\a)&=(\a-1), &    C^{(1^2)}(\a)&=(\a-1)(\a+2),\\
    C^{(3)}(\a)&=(\a-2)^2\a^2, &    C^{(1^3)}(\a)&=(\a-2)^2(\a+2)^3,
\end{align*}
and the sequences $P_n^{(\l)}(\a)$ are completely determined by the Chebyshev recursion, and the ``initial conditions":
\begin{align*}
    P_4^{(2)}(\a)&=\a(\a^2+\a-4), &        P_5^{(2)}(\a)&=\a^4+\a^3-5\a^2-\a+2,\\
    P_4^{(1^2)}(\a)&=(\a-2)(\a+1), &        P_5^{(1^2)}(\a)&=\alpha ^3-\alpha ^2-3 \alpha +1,\\
    P_5^{(3)}(\a)&=(\alpha +1) \left(\alpha ^2+3 \alpha -6\right), &        P_6^{(3)}(\a)&=\alpha  \left(\alpha ^3+4 \alpha ^2-4 \alpha -10\right),\\
    P_5^{(1^3)}(\a)&=(\alpha +1) \left(\a-3\right), &        P_6^{(1^3)}(\a)&=  \left(\alpha ^3-2 \alpha ^2-4 \alpha +2\right).
\end{align*}

\mdef   \label{pa:able}
By making use of \ref{pa:chebslv} we may write each polynomial $P_n^{(\l)}$ explicitly. 
We do this in more generality; we are able to compute 
\ppm{[-let's reword to show the proof.]}\bajm{[what this means?]}\ppm{[ah I see, more generality does not mean generality. ok I'll just think of a wording a bit different from ``are able to''.]}
(also in the notation of Theorem \ref{thm:main}) 
\ppm{[-forward ref.]}
for $0\leq l\leq 12$, that
\begin{align*}
    C^{(l+2)}(\a)&=(\a+(l-1))^{l+1}(\a-2)^{\frac{1}{2}l(l+3)},\\
    C^{(1^{l+2})}(\a)&=(\a-(l+1))^{l+1}(\a+2)^{\frac{1}{2}(l+2)(l+1)},
\end{align*}
and
\begin{align*}P^{(l+2)}_{l+4+k}&=P_{k+4}^U(\a)+(3 l+1) P_{k+3}^U(\a)+
\left(2 l^2-l-2\right)P_{k+2}^U(\a)\\
&\quad+(l+1)((2l-1)P^U_{k+1}(\a)-P^U_{k}(\a)),\\
P^{(1^{l+2})}_{l+4+k}&=P_{k+3}^U(\a)-(l+1)(P_{k+2}^U(\a)+P_{k+1}^U(\a)).\end{align*}
Observe in each case:
\begin{align}
    P^{(l+2)}_{l+4}&=(\alpha +l) \left(\alpha ^2+(2l+1)\alpha -2(l+1)\right), \label{eq:row0}\\
    P^{(1^{l+2})}_{l+4}&=(\alpha +1) (\alpha- (l+2) ). \label{eq:col0}
\end{align}}

\subsubsection{On Gram determinants $\Delttt{(n)}{n-2,\l}$ 
for general $\l$
}\label{ss:non1dsp}

When $\SS_\l$ is not (necessarily) one dimensional we may proceed as follows. 
In this section we give a formal factorisation of the gram determinants. 
In \S\ref{ss:examples0} we determine the factors.

\mdef 
Assume that $\BB=\{b_1,\dots, b_{d_\l}\}$ is an orthonormal basis for 
$\SS_\l$, regarded as one for $\Delta^n_{(n,\l)}$ (\textit{cf.} \ref{pa:spbasis}),
with respect to our contravariant form $\langle \_ , \_ \rangle$. Observe that this requires that we are working over a splitting field for the symmetric group. And in fact we will go further and work here over $\C$. More precisely we are working over $\C[\xx]$ so we cannot, for example, take an orthonormal basis for $\Deltaa^{n}_{(p,\l)}$.

The formula \eqref{cheby} indicates that we can divide our basis for $\Deltaa^{n}_{(p,\l)}$ into a head and a tail for each basis element $b_i \in \BB$, with similar behaviour. 
In this case, there may be some non-zero values of the form between elements belonging to distinct heads. 
For example, when $\l=(2,1)$, by choosing the basis 
\beq   \label{eq:21ortho}
{\mathcal B}_{(21)} \; 
   = \;\;\{b_1=c_{\l}, \; b_2=\frac{2}{\sqrt{3}}((23)+\frac{1}{2} e )c_{\l}\}
\eq  
for $\SS_\l$, and ordering the 
$ \uu_{j,k} b_i$ 
in lexicographic order now with respect to the triple $(i,j,k)$, we obtain the 
expression for the gram matrix $\Gamma(\Deltaa^{8}_{(6,\l)})$ 
in Fig.\ref{fig:n9la21}.

\begin{figure} \vspace{-.43in} 
\input{tex/grammn9la21}
    \caption{Gram matrix  $\Gamma(\Deltaa^{8}_{(6,(2,1))}) $ 
    \label{fig:n9la21}}
\vspace{-.0251in} 
\end{figure}

\medskip  

\newcommand{\B}[2]{{\mathsf B}^{#1}_{#2}}

We can again obtain inductive formulae for determinants by expanding along rows corresponding to the end of a tail. We proceed as follows.

\mdef \label{par:bmn}
For $\l$ any partition, and $\bm{n}=(n_1,n_2,\dots,n_{d_{\l}})$ a $d_\l$-tuple of integers with $n_k \geq l+4$, we introduce the set 
\[
\BBB^{(\l)}_{(l,\bm{n})}:=
   \{ b_i \uu_{jk} \ | \ i= 1,\dots, d_{\l}, \ \ \uu_{jk}\in \BBB_{(l,n_i)}\}.
\]
We may regard $\BBB^{(\l)}_{(l,\bm{n})}$ as a subset of the basis $\BBB^{n}_{(n-2,\l)}$
for $\Delta^{n}_{(p,\l)}$, where $n=\max(\bm{n})$. Indeed 
\[
\BBB^{(\l)}_{(l,\bm{n}=(n,n,...,n))}   \;  = \;    \BBB^{n}_{(n-2,\l)}  
.\]
As such, we may consider the matrix of the form $\langle\_,\_\rangle$ on $ \B{(\lambda)}{(l,\bm{n})} \times \B{(\lambda)}{(l,\bm{n})}$
- denote this matrix by 
\[
\Gamma_{(l,\bm{n})}^{(\l)}   :=  
  \langle\_,\_\rangle|_{  \B{\lambda}{l,\bm{n}} \times \B{\lambda}{l,\bm{n}} }  . 
\] 
This matrix contains the head corresponding to each $b_i\in \mathcal{B}$, but in general can have tails of differing length for each such $b_i$. 

By expanding the determinant along the row at the end of any tail (assuming said tail is long enough), we can verify that the multisequence of determinants $\det(\Gamma_{l,\bm{n}}^{(\l)})$, satisfies the Chebyshev recursion \eqref{eq:Cheby0}, in each component of $\bm{n}$. More precisely, we have proven 

\begin{proposition} \label{prop:detrec} Fix $l$ and let $\lambda\vdash l+2$.
Let $\bm{n}=(n_1,\dots,n_{d_\l})$ be as before and let $e_k\in \Z^{d_{\l}}$ denote the elementary vector. Then we have
\beq
\det(\Gamma_{l,\bm{n}+2e_k}^{(\l)})=\a \det(\Gamma_{l,\bm{n}+e_k}^{(\l)})-\det(\Gamma_{l,\bm{n}}^{(\l)}).
\label{eq:detrec2}
\eq 
\qed\end{proposition}

That is to say, by (\ref{de:Chebyab}), 
\beq   \label{eq:detGamP}
\det(\Gamma_{l,\bm{n}}^{(\l)})=  \PP^{a,b}_{n_k} 
\eq 
for some $a,b$ 
which in general, may depend on the basis element $b_k$, and the $d_\l-1$ integers $\bm{n}_{\setminus k}$ where 
\beq  \label{eq:de-nminusk} 
\bm{n}_{\setminus k} := (n_1,n_2,...,n_{k-1},n_{k+1},...,n_{d_\lambda}).
\eq
Specifically 
\begin{align*}
a &= \det\left(\Gamma_{l,(n_1,...,n_{k-1},l+4,n_{k+1},...,n_{d_\lambda})}^{(\l)}\right),\\
b &= \det\left(\Gamma_{l,(n_1,...,n_{k-1},l+5,n_{k+1},...,n_{d_\lambda})}^{(\l)}\right).\end{align*}
It follows from \eqref{eq:adep} that $a$ and $b$ are monic polynomials and satisfy $\deg(b)=\deg(a)+1$. Therefore, $\PP^{a,b}$ can be written as 
\beq  \label{eq:PcP}
\PP^{a,b}
\; = \;  c \; \PP^{a',b'} 
\eq 
where $a$ and $b$ have the maximal (monic) common factor $c$, and $a'=a/c$ and $b'=b/c$. For $c$ chosen as such, the series $\PP^{a',b'}$ then has the ramping property for $n\geq l+4$ (\textit{cf.} \ref{def:chebycond}).


\medskip \medskip 

The rest of \S\ref{ss:non1dsp} will be devoted to proving the following result, 
which shows that every one-cup determinant is determined by a Chebyshev recurrence from two in low rank (which are in turn amenable to direct computation, as we will illustrate in \S\ref{ss:examples0}):

\begin{theorem} \label{thm:main} 
Fix $l$. For any $\l\vdash l+2$, there exists a monic polynomial $C^{(\l)}$, and a series of polynomials $P_{n}^{(\l)}$ with the ramping property for $n\geq l+4$ (\textit{cf.} \ref{def:chebycond}), such that for $n\geq l+4$ 
(and hence in the $\l$-arm of the rollet graph $\Roll_l$ \textit{cf.} \ref{pa:rollarms})    
\beq 
\Delttt{n}{n-2,\l}(\a) =   C^{(\l)}(\a) (P_{n}^{(\l)}(\a))^{d_{\l}},
\label{eq:detform}
\eq
up to an overall complex scalar, 
\ppmm{where $d_\l =\dim(\SS_\l)$}.
\end{theorem}

Firstly, we will make use of the following Lemma.

\begin{lemma}\label{lem:xi} 
Fix $l$. For each $\l\vdash l+2$, there exists a unique (up to a polynomial factor) 
element $\xi_{\l} \in \Delta^{(l+4)}_{(l+2,\l)}$ with coefficients in $\C[\a]$, such that 
\beq
\uu_{ij}^* \, \xi_{\l} = 0, \text{ for } (i,j)\neq (l+3,l+4), \qquad 
    \uu_{l+3,l+4}^*\, \xi_{\l}  \ppmm{\;=\; D(\xx)} c_\l.  \label{eq:van}
\eq  
\ppmm{for some scalar $D(\xx)\in \C[\a]$.} 
Here $\uu_{ij}\in B_{(l,l+4)}$ as per \eqref{cupbasis}, $(\_)^*$ is as in \ref{de:flip2} and $c_\l \in J_{l}(l+2,l+2)$ is as in \ref{de:cl} (\textit{cf.} also \ref{ss:keybootl1}).
\end{lemma}

\begin{proof} 
Set $\Delta=\Delta^{(l+4)}_{(l+2,\l)}$ for this proof. 
And let $\BB=\{b_1,b_2,\dots, b_{d_\l}\}$ be an orthonormal (with respect to $\langle\_,\_\rangle$) basis for $\SS_\l$, 
\ppmm{and hence for $\Deltaaa{n}{n,\l}$}, 
such that 
$b_1= \nC_\l \in \C \Sym_{l+2}$.  

For any $\xi \in \Delta$, the expression 
$\uu_{ij}^* \, \xi \in \Delta^{(l+2)}_{(l+2,\l)}$. 
The elements $\uu_{ij}^* \xi$ can thus be expressed  
uniquely as a linear combination of the $b_k$. That is,  $\uu_{ij}^* \, \xi=\sum a_k b_k$ with coefficients $a_i$ determined by
\[  
a_k  =  \langle b_k, \uu_{ij}^* \, \xi \rangle  =   \langle \uu_{ij} b_k, \xi \rangle. 
\]
\ppmm{where the first equality uses the form for $\Deltaaa{n}{n,\l}$ and the second then uses the form for $\Deltaaa{n}{n-2,\l}$}.
Expanding $\xi_{}$ in the basis $\BBB^n_{n-2,\l}$ from \eqref{eq:basis}:
$\xi_{}=\sum_{i,j,k} [\xi_{}]_{i,j,k} \, \uu_{i,j} b_k$.  
Thus, \ppmm{comparing with (\ref{de:grammat}),}
equations \eqref{eq:van} 
may be written in component form as
\[ 
\Gamma(\Delta) \cdot [\xi_{\l}] = \; \ppmm{D(\a)} \;  e_{l+3,l+4,1},  \tag{$\star$} \]
where $e_{l+3,l+4,1}$ is the elementary column vector 
(with entries being $0$, except in the position corresponding to the basis element $\uu_{l+3,l+4} b_1$, where it is $1$). 

Note, for example from  \eqref{eq:adep}, 
that the matrix $
\Gamma(\Delta)$ is generically invertible (i.e. considering $\a$ as a formal parameter and working over the field of rational functions $\C(\a)$). Thus the equation $(\star)$ has a unique solution 
for any factor \ppmm{$D(\a)$}, 
with coefficients $[\xi_{\l}]_{i,j,k}\in \C(\a)$. 
Multiplying $\xi_\l$ by a common denominator  
of the coefficients, if necessary, we 
simply vary the factor. Hence we 
obtain the desired $\xi_\l$, obeying $\uu_{l+3,l+4}^* \, \xi_l=D(\a) c_\l$. 
\end{proof}


\mdef 
\ppmm{Passing to the ground field $\C$ by specialising $\xx$ to}
a complex number then, we note that it may be the case that $\xi_\l$ also satisfies $\uu_{l+3,l+4}^* \, \xi_l=0$ (\textit{i.e.} when $\a$ was a root of the denominator $D(\a)$).

\mdef\label{niceelt} 
Let $\xi_{\l} \in \Delta^{(l+4)}_{(l+2,\l)}$ be as per Lemma~\ref{lem:xi} 
\ppmm{ - i.e. in the first one-cup standard module in the $\l$-arm.} 
And let
$\BB=\{b_1,b_2,\dots, b_{d_{\l}}\}$ be an orthonormal basis for $\SS_\l$, with $b_k=x_k \nC_\l$ and $x_k\in \C S_{l+2}$.
For $k =1,2,...,d_\l$ 
define 
\beq 
\xi_{\l}^{(b_k)}:= (x_k \nC_\l \otimes \text{id}_2) \, \xi_\l \in \Delta^{(l+4)}_{(l+2,\l)}, 
\eq
regarding $\nC_\l$ as an element of $\C \Sym_{l+2}\subset J_{l,l+2}$. We now claim the $\xi_{\l}^{(b_k)}$ obey the following: 

\begin{lemma}\label{lem:niceelt} Fix $l$, let $\l \vdash l+2$. Let $\BB$ and $\xi_{\l}^{(b_k)} \in \Delta^{(l+4)}_{(l+2,\l)} $ be as per \ref{niceelt}. Then $\xi_{\l}^{(b_k)}$ has coefficients in $\C[\a]$ and satisfies:
\begin{enumerate}
    \item The coefficient of the basis vector $\uu_{l+3,l+4} \, b_{k'}$ in $\xi_{\l}^{(b_k)}$ is zero unless $k'=k$. Furthermore, the coefficient of $\uu_{l+3,l+4} \, b_{k}$ is independent of $b_k$.
    \item  The form $\langle \xi_{\l}^{(b_k)}, \uu_{i,j} b_{k'}  \rangle $ vanishes unless $(i,j)=(l+3,l+4)$ and $k'=k$. Furthermore, $\langle \xi_{\l}^{(b_k)}, \uu_{l+3,l+4} b_{k}  \rangle $ is independent of $k$.
\end{enumerate}
\end{lemma}
\begin{proof} 
Firstly, since $\xi_\l$ had coefficients in $\C[\a]$ and since $x_i, \nC_\l \in \C \Sym_{l+2}$, it follows that $\xi_\l^{(b_k)}$ has coefficients in $\C[\a]$. 

For any cup $\uu_{i,j}$ with $j<l+4$, the composite $(x_i \nC_\l  \otimes \text{id}_2)\uu_{i,j}$ can be written as a combination of the form $\sum_{i'<j'<l+4} \uu_{i',j'} x_{i',j'}$ with $x_{i',j'}\in \C\Sym_{l+2}$, by ``pulling" the cup up through each permutation. Therefore, the only contribution to basis vectors of the form  $\uu_{l+3,l+4} b_{k''}$ in $\xi_\l^{(b_k)}$ comes from basis vectors of the form $\uu_{l+3,l+4} b_{k'}$ in $\xi_\l$. Computing
\begin{align*}(x_k \nC_\l  \otimes \text{id}_2) \uu_{l+3,l+4} b_{k'}&=\uu_{l+3,l+4}(x_k \nC_\l x_{k'} c_{\l})\\
&=\langle e c_\l, b_{k'}\rangle \  \uu_{l+3,l+4}x_k \nC_{\l} c_{\l}\\
&=\langle e c_{\l}, b_{k'}\rangle \ \uu_{l+3,l+4}b_k,\end{align*}
(using the fact that $\nC_{\l} c_{\l}=c_{\l}$) it follows that the coefficient of $\uu_{l+3,l+4} \, b_{k''}$ (in $\xi_\l^{(b_k)}$) is 0 unless $k''=k$, in which case we have
\[ \left[\xi_\l^{(b_k)}\right]_{l+3,l+4,k}= \sum_{k'} \left[\xi_\l \right]_{l+3,l+4,k'} \langle e c_\l ,b_{k'}\rangle,\]
which is manifestly independent of $k$, proving part 1 of the Lemma.

To prove part 2, we first note that for any $i<j<l+4$, $\uu_{i,j}^* (x_k \nC_\l\otimes \text{id}_2)$ can be written as a combination $\sum_{i'<j'<l+4} x_{i',j'}\uu^*_{i',j'}$ with  $x_{i',j'}\in \C \Sym_{l+2}$ (by ``pulling" a cap down through each permutation). Therefore, by \eqref{eq:van} it follows that for any such $i<j$ we have $\uu_{i,j}^*\xi_{\l}^{(b_k)}=0$ and in particular $\langle \uu_{i,j} b_{k'}, \xi_{\l}^{(b_k)}   \rangle =0$ for any $k'$. When $(i,j)=(l+3,l+4)$, we have
\begin{align*}
    c_{\l} x_{k'}^* \uu_{l+3,l+4}^* \xi_{\l}^{(b_k)}&= 
    c_{\l} x_{k'}^* \uu_{l+3,l+4}^* (x_k \nC_\l \otimes \text{id}_2) \, \xi_\l,\\
    &=(c_\l \,  x_{k'}^* \, x_k \, \nC_\l) \uu_{l+3,l+4}^* \, \xi_\l\\
    &=\langle b_{k'} , b_k\rangle \nC_\l \uu_{l+3,l+4}^* \, \xi_\l=\delta_{k,k'}D(\a) c_\l, 
\end{align*}
where we have used orthogonality of the $b_k$, $\nC_\l c_\l=c_\l$, and $\uu_{l+3,l+4}^* \, \xi_\l= D(\a) c_\l$ (as per \ref{lem:xi}) in the last line. We can thus read off that 
\[
\langle \uu_{l+3,l+4} b_{k'}, \xi_{\l}^{(b_k)}  \rangle =\delta_{k,k'} D(\a),
\]
which completes the proof since $D(\a)\in \C[\a]$ is manifestly independent of $k$. 
\end{proof}

We now apply this Lemma to prove the following:

\begin{lemma}\label{lem:pl} Fix $l$, let $\l \vdash l+2$. Let $\BB$ be an orthonormal basis for $\SS_\l$, let $\bm{n}=(n_1,\dots, n_{d_\l})$ be a $d_\l$-tuple of integers with $n_k\geq l+4$ for all $k$, and let $\Gamma_{(l,\bm{n})}^{(\l)}$ be as per \ref{par:bmn}. Then for each $k=1,\dots, d_\l$, there exists a monic polynomial $C^{(\l,\bm{n}_{\setminus k})}$ and a series of polynomials $P^{(\l)}_{n}$ with the ramping property for $n\geq l+4$ (\textit{cf.} \ref{def:chebycond}), such that 
\[ 
\text{det}(\Gamma_{(l,\bm{n})}^{(\l)}) \;=\; C^{(\l,\bm{n}_{\setminus k})}(\a) \; P^{(\l)}_{n_k}(\a),
\]
where $\bm{n}_{\setminus k}$ is as per \eqref{eq:de-nminusk}. In particular, $P^{(\l)}_{n_k}$ (and hence $C^{(\l,\bm{n}_{\setminus k})}$) is independent of the basis vector $b_k\in \BB$. Furthermore, $P^{(\l)}_{n_k}$ does not depend on the integers $\bm{n}_{\setminus k}$.
\end{lemma}

\begin{proof} 
Monicity of the determinant $\text{det}(\Gamma_{(l,\bm{n})}^{(\l)})$, its degree, and the recursion established in Prop.\ref{prop:detrec}, imply that for each such $k$, we may write   
    \[
    \text{det}(\Gamma_{(l,\bm{n})}^{(\l)})=C^{(b_k,\l,\bm{n}_{\setminus k})}(\a) P^{(b_k,\l, \bm{n}_{\setminus k})}_{n_k}(\a),
    \]
where $C^{(b_k,\l,\bm{n}_{\setminus k})}(\a)$ is a monic polynomial and $P^{(b_k,\l, \bm{n}_{\setminus k})}_{n}(\a)$ is a series of polynomials with the ramping property for $n\geq l+4$. In general, $P^{(b_k, \l,\bm{n}_{\setminus k})}_{n_k}(\a)$ may depend on $b_k$ and on the integers $\bm{n}_{\setminus k}$; the crux of this proof is thus to show otherwise. To do this, it is sufficient to verify that for any $k=1,\dots, d_\l$ and choice of integers $\bm{n}_{\setminus k}$, the sequence has the same ``initial conditions", that is,
\[
P^{(b_k,\bm{n}_{\setminus k})}_{l+4}(\a)=a(\a), \qquad P^{(b_k,\bm{n}_{\setminus k})}_{l+5}(\a)=b(\a),  
\]
for polynomials $a,b$ independent of $b_k$ and $\bm{n}_{\setminus k}$. Furthermore, by \ref{def:chebycond} it suffices to demonstrate the ratio 
\beq
\frac{P^{(b_k,\bm{n}_{\setminus k})}_{l+4}(\a)}{P^{(b_k,\bm{n}_{\setminus k})}_{l+5}(\a)}=\frac{a(\a)}{b(\a)}=\frac{\det(\Gamma^{(\l)}_{(l,(n_1,\dots, n_{k-1},l+4,n_{k+1},\dots, n_{d_\l}))})}{\det(\Gamma^{(\l)}_{(l,(n_1,\dots, n_{k-1},l+5,n_{k+1},\dots, n_{d_\l}))})},\label{eq:ratio}
\eq
is independent of $b_k$ and $\bm{n}_{\setminus k}$.

Suppose, by reordering $\BB$ if necessary, that $k=d_\l$. Now for a $(d_\l-1)$-tuple $\bm{n'}=(n'_1,\dots,n'_{d_\l-1})$, with $n'_{k'}\geq l+4$ for all $k'$, consider the matrix $\Gamma:=\Gamma^{(\l)}_{(l,(\bm{n'},l+5))}$. By the computation \eqref{cheby}, it has the form 
\beq
\Gamma=\Gamma^{(\l)}_{(l,(\bm{n'},l+5))}=
\begin{pNiceArray}{cccc|c}
\Block{4-4}<\Large>{\Gamma^{(\l)}_{(l,(\bm{n'},l+4))}} &&&& 0\\
&&&& \vdots\\
&&&& 0 \\
&&&& 1 \\ \hline
0 & \dots &0 &1 & \a \\
\end{pNiceArray}.\label{eq:matlastrow}
\eq

By Lemma \ref{lem:niceelt} part 2, it follows that the following linear combination of the rows of the Gram-matrix $\Gamma(\Delta^{(l+4)}_{(l+2,\l)} )=\Gamma^{(\l)}_{(l,\bm{l'})}$ (where $\bm{l'}$ is the $d_\l$-tuple with all entries being $l+4$) yields the row vector
\begin{align*}
    \sum_{i,j,k} \left[\xi^{(b_{d_\l})}_{\l}\right]_{i,j,k} R_{\uu_{i,j} b_k}&=  \langle \xi^{(b_{d_\l})}_{\l}, \uu_{l+3,l+4} \, b_{d_{\l}}\rangle \, e_{l+3,l+4, d_{\l}}\\
    &=\begin{pmatrix}
    0 &\dots & 0 & \langle \xi^{(b_{d_\l})}_{\l}, \uu_{l+3,l+4} \, b_{d_{\l}}\rangle
\end{pmatrix}
\end{align*}
where $e_{i,j,k}$ denotes the elementary vector. Furthermore, by \ref{lem:niceelt} part 1 this combination does not include the rows $R_{\uu_{l+3,l+4} b_{k}}$ for $k<d_{\l}$. Since the row $R'_{\uu_{i,j} b_k}$ with $j<l+4$ in the matrix $\Gamma$ can be obtained from the corresponding row $R_{\uu_{i,j} b_k}$ in the Gram-matrix $\Gamma(\Delta^{(l+4)}_{(l+2,\l)})$, by putting a zero in any column labelled by an element of any tail (that is, an element of the form $\uu_{i,j} b_k$ with $j>l+4$), it follows that the same combination of the corresponding rows of the matrix $\Gamma$ yields the row vector
\[ \begin{pmatrix}
    0 &\dots & 0 & \langle \xi^{(b_{d_\l})}_{\l}, \uu_{l+3,l+4} b_{d_{\l}}\rangle & \left[\xi^{(b_{d_\l})}_{\l}\right]_{l+3,l+4,d_{\l}}
\end{pmatrix}.\]
Therefore, after multipling the last row of $\Gamma$ by $\langle \xi^{(b_{d_\l})}_{\l}, \uu_{l+3,l+4} b_{d_{\l}}\rangle$, we can apply a sequence  of row reductions to exhibit the equality
\begin{align*}
    \langle \xi^{(b_{d_\l})}_{\l}, \uu_{l+3,l+4} b_{d_{\l}}\rangle \det(\Gamma) &=
\begin{vNiceArray}{cccc|c}
\Block{4-4}<\Large>{\Gamma^{(\l)}_{(l,(\bm{n'},l+4))}} &&&& 0\\
&&&& \vdots\\
&&&& 0 \\
&&&& 1 \\ \hline
0 & \dots &0 & \langle \xi^{(b_{d_\l})}_{\l}, \uu_{l+3,l+4} b_{d_{\l}}\rangle & \a \, \langle \xi^{(b_{d_\l})}_{\l}, \uu_{l+3,l+4} b_{d_{\l}}\rangle\\
\end{vNiceArray}\\
&=\begin{vNiceArray}{cccc|c}
\Block{4-4}<\Large>{\Gamma^{(\l)}_{(l,(\bm{n'},l+4))}} &&&& 0\\
&&&& \vdots\\
&&&& 0 \\
&&&& 1 \\ \hline
0 & \dots &0 & 0 & \a \, \langle \xi^{(b_{d_\l})}_{\l}, \uu_{l+3,l+4} b_{d_{\l}}\rangle-\left[\xi^{(b_{d_\l})}_{\l}\right]_{l+3,l+4,d_{\l}}\\
\end{vNiceArray}\\
&=\left( \a \, \langle \xi^{(b_{d_\l})}_{\l}, \uu_{l+3,l+4} b_{d_{\l}}\rangle-\left[\xi^{(b_{d_\l})}_{\l}\right]_{l+3,l+4,d_{\l}} \right)\det(\Gamma^{(\l)}_{(l,(\bm{n'},l+4))}).
\end{align*}
Comparing with \eqref{eq:ratio}, we now obtain
\[\frac{P^{(b_{d_{\l}},\bm{n'})}_{l+4}(\a)}{P^{(b_{d_\l},\bm{n'})}_{l+5}(\a)}=\frac{\langle \xi^{(b_{d_\l})}_{\l}, \uu_{l+3,l+4} b_{d_{\l}}\rangle}{\left( \a \, \langle \xi^{(b_{d_\l})}_{\l}, \uu_{l+3,l+4} b_{d_{\l}}\rangle-\left[\xi^{(b_{d_\l})}_{\l}\right]_{l+3,l+4,d_{\l}} \right)},\]
which is manifestly independent of $\bm{n'}$, and independent of $b_{d_{\l}}$ by Lemma \ref{lem:niceelt}. This completes the proof. \end{proof}

We now apply this Lemma to prove:

\begin{proposition}\label{prop:main}
    Fix $l$, let $\l \vdash l+2$. Let $\BB$ be an orthonormal basis for $\SS_\l$, let $\bm{n}=(n_1,\dots, n_{d_\l})$ be a $d_\l$-tuple of integers $n_k\geq l+4$ and let $\Gamma_{(l,\bm{n})}^{(\l)}$ be as per \ref{par:bmn}. There exists a monic polynomial $C^{(\l)}$ and a series of polynomials $P^{(\l)}_{n}$ with the ramping property for $n\geq l+4$ (\textit{c.f} \ref{def:chebycond}), such that 
\[ \text{det}(\Gamma_{(l,\bm{n})}^{(\l)})=C^{(\l)}(\a) \prod_{k=1}^{d_{\l}} P^{(\l)}_{n_k}(\a).\]\end{proposition}
\begin{proof} By Lemma \ref{lem:pl}, it follows that each $P^{(\l)}_{n_k}(\a)$ divides $\text{det}(\Gamma_{(l,\bm{n})}^{(\l)})$. It suffices to prove that $\prod_{k=1}^{d_{\l}} P^{(\l)}_{n_k}(\a)$ divides $\text{det}(\Gamma_{(l,\bm{n})}^{(\l)})$ since then, comparing with the factorisation from Lemma \ref{lem:pl} yields
\[ C^{(\l)}(\a) \prod_{k\neq k' } P^{(\l)}_{n_k}(\a)= C^{(\l,\bm{n}_{\setminus k'})}(\a), \]
for each $k'=1,\dots, d_{\l}$ which shows $C^{(\l)}(\a)$ is independent of all the integers $n_{k'}$. 

Suppose for a contradiction that for some $\bm{n}$, $\prod_{k=1}^{d_{\l}} P^{(\l)}_{n_k}(\a)$ does not divide $\text{det}(\Gamma_{(l,\bm{n})}^{(\l)})$. By reordering the basis $\BB$ is necessary, it follows there is some integer $2\leq d \leq d_{\l}$ such that $\prod_{k=1}^{d-1} P^{(\l)}_{n_k}(\a)$ divides $\text{det}(\Gamma_{(l,\bm{n'})}^{(\l)})$ for any $d_\l$-tuple $\bm{n'}$ which agrees with $\bm{n}$ in the first $(d-1)$-entries, yet $\prod_{k=1}^{d} P^{(\l)}_{n_k}(\a)$ does not divide $\text{det}(\Gamma_{(l,\bm{n})}^{(\l)})$. Since $\prod_{k=1}^{d-1} P^{(\l)}_{n_k}(\a)$ divides $\text{det}(\Gamma_{(l,\bm{n})}^{(\l)})$ by assumption and since $P^{(\l)}_{n_{d}}(\a)$ must divide it, it follows these two factors must ``overlap" non-trivially. That is, we must be able to write
\[ \text{det}(\Gamma_{(l,\bm{n})}^{(\l)})= A(\a) B(\a) C(\a) D(\a),\]
where $D(\a)$ has positive degree and does not divide $A(\a)$, such that $B(\a) D(\a)= P^{(\l)}_{n_{d}}(\a)$ and $C(\a) D(\a)= \prod_{k=1}^{d-1} P^{(\l)}_{n_k}(\a)$. But then lemma \ref{lem:pl} implies that   
\[\text{det}(\Gamma_{(l,\bm{n}+e_{d})}^{(\l)})= A(\a) C(\a)  P^{(\l)}_{n_{d}+1}(\a). \]
Since $C(\a)D(\a)$ must divide $\text{det}(\Gamma_{(l,\bm{n}+e_{d})}^{(\l)})$, yet $D$ does not divide $A$, it follows that $D$ must divide $P^{(\l)}_{n_{d}+1}$. The Chebyshev recursion then implies that $D$ is a non-trivial common factor of the sequence $P^{(\l)}_{n}$ which is a contradiction.\end{proof}

We are now able to prove Theorem \ref{thm:main} as a specialisation of Prop. \ref{prop:main}:

\begin{proof}[Proof of Theorem \ref{thm:main}] For any $n\geq l+4$, taking $\bm{n}$ to be the $d_{\l}$-tuple, $\bm{n}=(n,n,\dots, n)$, yields
\[\Delttt{(n)}{n-2,\l}=\text{det}(\Gamma_{(l,\bm{n})}^{(\l)})=C^{(\l)}(\a) (P^{(\l)}_{n}(\a))^{d_{\l}},\]
as desired. \end{proof}

\subsection{Applications of Theorem~\ref{thm:main}: \ppmm{polynomials,} and examples}   \label{ss:examples0}

In this subsection we apply Theorem \ref{thm:main} to obtain closed forms for gram determinants of the $J_{l,n}$-modules, $\Delt{n}{n-2,\l}$ with $n\geq l+4$, for all cases when $l\leq 2$, by direct computation of low rank (namely $n=l+4,l+5$) cases. 
We also study the cases $\l=(l+2), (1^{l+2}), (l+1,1)\vdash l+2$ and $\l=(l+1,1)^T\vdash l+2$ for $l>0$. We begin by giving examples for $\l$ such that $\SS_\l$ is 1-dimensional.

\mdef \label{pa:1dgdets}
For $l=0$  
direct computation yields
\begin{align*}
\Delttt{(4)}{2,(2)}&=(\alpha -1) \alpha  \left(\alpha ^2+\alpha -4\right),& \Delttt{(5)}{3,(2)}&=(\alpha -1) \left(\alpha ^4+\alpha ^3-5 \alpha ^2-\alpha +2\right),\\
    \Delttt{(4)}{2,(1^2)}&=(\alpha -2) (\alpha -1) (\alpha +1) (\alpha +2),&
   \Delttt{(5)}{3,(1^2)}&=(\alpha -1) (\alpha +2) \left(\alpha ^3-\alpha ^2-3 \alpha +1\right),
\end{align*}  
For $l=1$, we consider the row and column partitions $\l$ here. 
In these cases, direct computation yields
\begin{align*}
    \Delttt{(5)}{3,(3)}&=(\alpha -2)^2 \alpha ^2 (\alpha +1) \left(\alpha ^2+3 \alpha -6\right),& \Delttt{(6)}{4,(3)}&=(\alpha -2)^2 \alpha ^3 \left(\alpha ^3+4 \alpha ^2-4 \alpha -10\right),\\
    \Delttt{(5)}{3,(1^3)}&=(\alpha -3) (\alpha -2)^2 (\alpha +1) (\alpha +2)^3,&
    \Delttt{(6)}{4,(1^3)}&=(\alpha -2)^2 (\alpha +2)^3 \left(\alpha ^3-2 \alpha ^2-4 \alpha +2\right).
\end{align*}
Thus, using the notation of Theorem \ref{thm:main}, 
we have here $\Delttt{(n)}{n-2,\l}=C^{(\l)}(\a) P^{(\l)}_{n}(\a)$ where
\begin{align*}
    C^{(2)}(\a)&=(\a-1), &    C^{(1^2)}(\a)&=(\a-1)(\a+2),\\
    C^{(3)}(\a)&=(\a-2)^2\a^2, &    C^{(1^3)}(\a)&=(\a-2)^2(\a+2)^3,
\end{align*}
and the sequences $P_n^{(\l)}(\a)$ are completely determined by the Chebyshev recursion, and the ``initial conditions":
\begin{align*}
    P_4^{(2)}(\a)&=\a(\a^2+\a-4), &        P_5^{(2)}(\a)&=\a^4+\a^3-5\a^2-\a+2,\\
    P_4^{(1^2)}(\a)&=(\a-2)(\a+1), &        P_5^{(1^2)}(\a)&=\alpha ^3-\alpha ^2-3 \alpha +1,\\
    P_5^{(3)}(\a)&=(\alpha +1) \left(\alpha ^2+3 \alpha -6\right), &        P_6^{(3)}(\a)&=\alpha  \left(\alpha ^3+4 \alpha ^2-4 \alpha -10\right),\\
    P_5^{(1^3)}(\a)&=(\alpha +1) \left(\a-3\right), &        P_6^{(1^3)}(\a)&=  \left(\alpha ^3-2 \alpha ^2-4 \alpha +2\right).
\end{align*}
By making use of \ref{pa:chebslv} we may write each polynomial $P_n^{(\l)}$ explicitly. We address this in this next paragraph, where we treat more cases with a general formula.

\mdef   \label{pa:able} \bajmm{By direct computation of Gram matrices $\Delttt{(n)}{n-2,\l}$ for $\l=(l+2)$ and $(1^{l+2})$, when $-1 \leq l \leq 12$ and $n=l+4,l+5$, we have 
(using the notation of Theorem \ref{thm:main})  }
\begin{align}
    C^{(l'+2)}(\a)&=(\a+(l'-1))^{l+1}(\a-2)^{\frac{1}{2}l'(l'+3)},\label{eq:crow}\\
    C^{(1^{l+2})}(\a)&=(\a-(l+1))^{l+1}(\a+2)^{\frac{1}{2}(l+2)(l+1)},\label{eq:ccol}
\end{align}
and
\begin{align}P^{(l'+2)}_{l'+4+k}&=P_{k+4}^U(\a)+(3 l'+1) P_{k+3}^U(\a)+
\left(2 l'^2-l'-2\right)P_{k+2}^U(\a)\nn\\
&\quad+(l'+1)((2l'-1)P^U_{k+1}(\a)-P^U_{k}(\a)),\label{eq:prow}\\
P^{(1^{l+2})}_{l+4+k}&=P_{k+3}^U(\a)-(l+1)(P_{k+2}^U(\a)+P_{k+1}^U(\a)).\label{eq:pcol}\end{align}
These formulas are verified for $-1\leq l\leq 12$ and $0\leq l'\leq 12$.

\begin{conjecture} Formulae \eqref{eq:crow},\eqref{eq:ccol}, \eqref{eq:prow} and \eqref{eq:pcol} hold for all $l\geq -1$ and $l'\geq 0$. 
\end{conjecture}

\mdef By direct computation, using formulae \eqref{eq:prow} and \eqref{eq:pcol} (including $l,l'>12$) we have:
\begin{align}
    P^{(l'+2)}_{l'+4}&=(\alpha +l') \left(\alpha ^2+(2l'+1)\alpha -2(l'+1)\right), \label{eq:row0}\\
    P^{(1^{l+2})}_{l+4}&=(\alpha +1) (\alpha- (l+2) ). \label{eq:col0}
\end{align}

\mdef \label{pa:21dets}
 As we have seen in \S\ref{ss:sdmods}, for $l=1$ we have gram determinants 
\begin{align*}
    \Delttt{(5)}{3,(2,1) }&= (\alpha -2)^3 \alpha  (\alpha +2) (\alpha +4)(\alpha ^4-7 \alpha ^2+3)^2,\\
    \Delttt{(6)}{4,(2,1) }&= (\alpha -2)^3 \alpha  (\alpha +2) (\alpha +4)(\alpha -1)^2 \alpha^2(\alpha +1)^2 \left(\alpha ^2-7\right)^2.
\end{align*} 
Comparing with \eqref{eq:detform} from the Theorem, \bajmm{for $n\geq 5$,} 
we have
\[
\Delttt{(n)}{n-2,(2,1) }= C^{(2,1)}(\a) (P^{(2,1)}_n(\a))^2, \quad 
\]
Since $P^\l$ is reduced we deduce
\[
C^{(2,1)}(\a)= (\alpha -2)^3 \alpha  (\alpha +2) (\alpha +4),  
\]
as the largest common factor of the two determinants. Hence we deduce
$P^{(2,1)}_n=\PP^{a,b}_n$ with initial conditions (at $n=5=l+4$ and $n=6=l+5$)
\[
a=\alpha ^4-7 \alpha ^2+3, \qquad b=(\alpha -1) \alpha  (\alpha +1) \left(\alpha ^2-7\right).
\]
Using \ref{pa:chebslv} we can determine
\begin{align*}
    P^{(2,1)}_n&=P^U_{n}(\a)-4P^U_{n-2}(\a)-4P^U_{n-4}(\a)-2 P^U_{n-6}(\a).
\end{align*}

\mdef 
Similarly, for $l=2$, direct computation \ppmm{of the low rank gram determinants and applications of the Theorem \ref{thm:main}} yields
\begin{align*}
    C^{(3,1)}(\a)&=(\alpha -3)^5 (\alpha -2)^2 (\alpha -1)^3 (\alpha +1)^6 (\alpha +4) (\alpha +6),\\
    C^{(2,2)}(\a)&=(\alpha -3)^3 (\alpha -1)^6 (\alpha +1)^3 (\alpha +4)^4,\\
    C^{(2,1,1)}(\a)&=(\alpha -3)^6 (\alpha -1)^3 (\alpha +1)^5 (\alpha +2)^3 (\alpha +4)^4,\end{align*}
and
\begin{align*}
    P^{(3,1)}_{n}(\a)&=P^U_{n}(\a)+2 P^U_{n-1}(\a)-9 P^U_{n-2}(\a)\\
    &\quad-4 P^U_{n-5}(\a)- P^U_{n-6}(\a)-2 P^U_{n-7}(\a)-3 P^U_{n-8},\\
    P^{(2,2)}_{n}(\a)&=P^U_{n-2}(\a)-2 P^U_{n-3}(\a)-5 P^U_{n-4}(\a)+6 P^U_{n-5}(\a)-3 P^U_{n-6}(\a), \\
    P^{(2,1,1)}_{n}(\a)&=P^U_{n-1}(\a)-2 P^U_{n-2}(\a)-6 P^U_{n-3}(\a)\\
    &\quad +6 P^U_{n-4}(\a) -2 P^U_{n-5}(\a)+2 P^U_{n-6}(\a)-3 P^U_{n-7}(\a).
\end{align*}

\mdef For $2<l\leq 7$, direct computation of the Gram determinants $\Delttt{(n)}{n-2,\l}$ with $\l=(l+1,1), (l+1,1)^T$ and $n=l+4,l=5$, and application of Theorem \ref{thm:main} yields
\begin{align}
    C^{(4,1)}(\a)&=(\alpha -3)^{16} (\alpha -2)^5 \alpha ^{11} (\alpha +2)^{10} (\alpha +6) (\alpha +8)\nn\\
    C^{(l'+1,1)}(\a)&=(\alpha -3)^{(l'-1)(l+1)(l+3)/3} \, (\alpha -2)^{(l'-1)(l+2)/2} \, \alpha ^{l'('l+3)/2}(\alpha +(l'-3))^{l'(l'+1)/2} \nn\\
    &\quad\times (\alpha +(l'-1))^{(l'+1)(l'+2)/2} (\alpha +2l') (\alpha +2(l'+1)),\label{eq:chook}
\end{align}
where the last formula is valid for $l'\geq 4$, and
\begin{align}
        C^{(l+1,1)^T}(\a)&=(\a+4)^{l(l+1)(l+2)/6}(\a+2)^{l(l+1)/2}(\a+1)^{(l-1)(l+1)(l+3)/3}\nn\\
    &\quad\times(\alpha -(l-1))^{l(l+1)/2} (\alpha -(l+1))^{(l+1)(l+2)/2},\label{eq:chookt}
\end{align}
which holds for $l\geq 1$. The corresponding Chebyshev series are given by
\begin{align}
    P_{l''+4+k}^{(l''+1,1)}&=P^U_{k+6}(\a)+2 (2 l''-3) P^U_{k+5}(\a)+(5 l''^2-21 l''+13)P^U_{k+4}(\a)\nn\\
    &\quad +(l''-2) \left[\left(2 l''^2-17 l''+7\right)P^U_{k+3}(\a)- \left(6 l''^2-19 l''+5\right) P^U_{k+2}(\a)\right]\nn\\
    &\quad +4 \left(l''^3-6 l''^2+8 l''-1\right)P^U_{k+1}(\a) -(l''-3) \left(2 l''^2-4 l''-1\right)P^U_{k}(\a)\nn\\
    &\quad-(3 l''^2-3 l''-4)P^U_{k-1}(\a)-(l''+1)P^U_{k-2}(\a),\label{eq:phook}\end{align}
which holds for $l''\geq2$, and 
\begin{align}
    P_{l+4+k}^{(l+1,1)^T}&=P^U_{k+5}(\a)-2 (l-1) P^U_{k+4}(\a)+l(l-5)P^U_{k+3}(\a)+3l(l-1)P^U_{k+2}(\a)\nn\\
    &\quad 2(l^2-2l-1) P^U_{k+1}(\a)+l(l-1)P^U_{k}(\a)-(l+1)P^U_{k-1}(\a),\label{eq:phookt}
\end{align}
which holds for $l\geq 1$.

 \begin{conjecture}Formulae \eqref{eq:chook},\eqref{eq:chookt},\eqref{eq:phook}, and \eqref{eq:phookt}, hold for all $l\geq 1$, $l'\geq 4$ and $l''\geq 2$.     
 \end{conjecture}
\medskip 

Let us note, therefore, that in all cases directly computed, Theorem \ref{thm:main} can be strengthened as follows.  

\begin{proposition}\label{prop:ints} For $\l \vdash l+2$, such that either, $l\leq 2$; $\l=(l+2), (1^{l+2})$ and $-1\leq l\leq 12$; $\l=(l+1,1), (l+1,1)^T$ and $1\leq l \leq 7$, the polynomials $C^{(\l)}$ and $P_n^{(\l)}$ as per Theorem \ref{thm:main} have integer coefficients. Furthermore, each such $C^{(\l)}$ has only integer roots. \end{proposition}


\begin{conjecture} Proposition \ref{prop:ints} holds for all $l\geq -1$ and $\l\vdash l+2$.
\end{conjecture}

\begin{remark}\label{rem:polys} 
For a given $l$ and $\lambda\vdash l+2$, we note that whilst the polynomials $P_n^{(\l)}$ only appear as a factors in a Gram-determinant for $n\geq l+4$, the same polynomial
series 
with $n<l+4$ are uniquely determined by the Chebyshev recursion. For the cases we have computed, we give $P_{l+3}^{(\l)}$ and $P_{l+2}^{(\l)}$ below:
\begin{align*}
    P_{5}^{(2,2)}(\a)&=(\a-3)(\a+1),  &P_4^{(2,2)}(\a)&=4 (\alpha -2), \\
    P^{(l+2)^T}_{l+3}(\a)&=\alpha -(l+1), &  P^{(l+2)^T}_{l+2}(\a)&=l+2,\\
    P^{(l'+2)}_{l'+3}(\a)&=(\alpha +l'-1) (\alpha +2 (l'+1) ), &  P^{(l'+2)}_{l'+2}(\a)&=(l'+2) (\alpha +2 l'),\\
    P^{(l''+1,1)^T}_{l''+3}(\a)&=(\alpha +2) (\alpha -(l''-1) ) (\alpha -(l''+1) ), &  P^{(l''+1,1)^T}_{l''+2}(\a)&=(l''+2)(\alpha +1) (\alpha -l'' ),\end{align*}
    where these formulae are valid for heights $l\geq -1$, $l'\geq 0$, and $l''\geq 1$. Furthermore, when $l\geq 2$ we have
\begin{align*}
    P^{(l+1,1)}_{l+3}(\a)&=(\alpha -2) (\alpha +l-3) (\alpha +l-1) (\alpha +2 l), \\
     P^{(l+1,1)}_{l+2}(\a)&=(l+2) (\alpha +l-2) \left(\alpha ^2+(2l-3) \alpha -2 l\right).
\end{align*}
 Remarkably then, we observe that for all cases computed thus far, $P_{l+3}^{(\l)}(\a)$ is monic, has only integer roots, and (up to a scalar factor) divides the King polynomial (see \cite[\S 5.1]{DeVisscherMartin2016}).  Furthermore, in all such cases $P_{l+3}^{(\l)}(\a)$ has an overall integer factor of $l+2$, and comparing with \eqref{eq:row0} and \eqref{eq:col0} we observe that, for $l\geq 1$
\begin{align*}
    P^{(l+1,1)^T}_{l+2}(\a)&=(l+2) P^{(l)^T}_{l+2}(\a), & P^{(l+1,1)}_{l+2}&=(l+2)  P^{(l)}_{l+2}(\a). 
\end{align*}
We also note the equalities $P^{(2,2)}_{5}(\a)=P^{(1^3)}_{5}(\a)$ and $P^{(2,2)}_{4}(\a)=4 P^{(1^3)}_{4}(\a)$.
\end{remark}

\section{Roots of the polynomials $P^{(\l)}_n$}\label{s:roots}
In this section, we make some observations about the distribution of the roots of the polynomials $P^{(\l)}_n$ in the cases considered. We begin by noting that all roots of these polynomials must be real; since the form $\langle \_,\_ \rangle$ on $\Delta^{n}_{(n-2,\l)}$ is non-degenerate away from the real line (\textit{cf.} \ref{pa:gramrels}), we obtain:
\begin{corollary} Fix $l$ and let $\l\vdash l+2$. Let $C^{(\l)}\in \C[\a]$ and $P^{(\l)}_n\in \C[\a]$ be as per Theorem \ref{thm:main}, with $n\geq l+4$. Then $C^{(\l)}$ and $P^{(\l)}_n$ have real roots.   
\end{corollary}
In the cases computed thus far, we will use the mean value theorem to provide bounds on the values of roots. It will be helpful first to collect some identities for Chebyshev series with the \ramping\ property, which we do in \S\ref{ss:rampingroots}.

\subsection{Identities for \ramping\ series}\label{ss:rampingroots}
In what follows, let $P$ be a Chebyshev series.

\mdef\label{pa:roots} 
Suppose that $P$ has the \ramping\ property for $n\geq N$ and further suppose that $P_{k}(x_0)=0$ for some $k\in \Z$ and $x_0\in \C$. Since the series has no common factors, we have that $P_{k\pm 1}(x_0)\neq 0$. The Chebyshev recursion then implies that 
\beq
    P_{k+n}(x_0)=P^U_n(x_0)P_{k+1}(x_0),  \label{eq:chebyx0}
\eq
for any $n\in \Z$. Therefore, a series with the \ramping\ property can only have a recurring (with period $N>0$) root at $x_0$, if $x_0$ satisfies $P^U_N(x_0)=0$. 

The following lemma will help us (in \S\ref{ss:roots}) analyse the distribution of the roots of the polynomials $P^{(\l)}_n$ introduced in \S\ref{ss:non1dsp}.

\begin{lemma}\label{lem:roots} 
Let $P_n$ be a series which has the \ramping\ property for $n\geq N$. Let $k, k'$, and $ r$ be positive integers with $r<k+k'$. Then
\[P_{N+k}\left(2 \cos \left(\frac{r \, \pi}{k+k'}\right)\right)= (-1)^{r} P_{N-k'}\left(2 \cos \left(\frac{r \, \pi}{k+k'}\right)\right).  \]
\end{lemma}
\begin{proof} Set $\th=\pi/(k+k')$ for this proof. Using \eqref{eq:trig} we can compute
\
\begin{align*}P^U_{k+k'}\left(2 \cos \left(r \, \th\right)\right)&= \sin\left( r \, \pi \right)/\sin\left( r \, \th \right)=0,\\
    P^U_{k+k'+1}\left(2 \cos \left(r\, \th \right)\right)&= \sin\left( r \, \pi +r \, \th \right)/\sin\left(r\,\th \right)=(-1)^r,
\end{align*}
noting that $\sin\left(r\,\th\right)\neq 0$. Therefore, by \eqref{eq:chebyx0} for any $i\in \Z $ we have
\[
P^U_{k+k'+i}\left(2 \cos \left(r\, \th\right)\right)=(-1)^r P_i^{U}\left(2 \cos (r\, \th) \right).\tag{\Celtcross}\]
By \ref{pa:chebslv}, we then have
\begin{align*}    P_{N+k}(2 \cos(r\, \th) )&=\sum_{i=-(d+1)}^{d+1} a_i P^U_{k+i}(2 \cos(r\, \th) )\\
&=(-1)^r\sum_{i=-(d+1)}^{d+1} a_i  P^U_{-k'+i}(2 \cos(r\, \th) )\\
&= (-1)^r P_{N-k'}(2 \cos(r\, \th) ),
\end{align*}
where we have used (\Celtcross) in the second line with indices shifted $i\mapsto i-k'$.
\end{proof}

\subsection{Roots of the $P^{(\l)}_n$}\label{ss:roots}

\mdef Firstly, we consider the infinite family of series $P^{(l+2)^T}_{n}(\a)$ (i.e. using the formula \eqref{eq:prow} beyond the bound for which it is verified). Parametrising by $n=l+4+k$ with $k\geq 0$, we have that $P^{(l+2)^T}_{n}(\a)$ is of degree $k+2$.  By applying Lemma \ref{lem:roots}, with $k'=2$ we obtain for any $r=1,\dots,k+1$ 
\[P^{(l+2)^T}_{n}( x^{(k)}_{r})= (-1)^r(l+2),\]
where $x^{(k)}_{r}:=2\cos(r \pi/(k+2))$. We can also determine
\begin{align*}P^{(l+2)^T}_{n}( x^{(k)}_{k+2})&=P^{(l+2)^T}_{n}( -2)=(-1)^{k+2} (4 + k + l),\\
P^{(l+2)^T}_{l+4+k}( l+2 )&=-P^{U}_{k}(l+2)<0, \quad (l>-1),\\
\lim_{x\to \infty } P^{(l+2)^T}_{n}( x)&=\lim_{x\to \infty } x^{k+2}=+\infty.
\end{align*}
In the second line, we have noted that $l+2$ is a root of $P_{l+4}^{(l+2)^T}$ and applied \eqref{eq:chebyx0}. Thus by an application of the mean value theorem we have that for $l>-1$, $P^{(l+2)^T}_{l+4+k}(\a)$ has $k+2$ distinct real roots $y^{(k)}_{r}$ with $r=0,1,\dots,k+1$ satisfying $y_0^{(k)}>l+2$ and 
\[-2<y^{(k)}_{k+1}<x^{(k)}_{k+1}<y^{(k)}_{k}<x^{(k)}_{k}< \dots <x^{(k)}_{2}<y^{(k)}_{1}<x^{(k)}_{1}<2.\]

\mdef We now consider the infinite family $P_n^{(l+2)}(\a)$ for $l>-1$. This has degree $k+3$ (with the same parametrisation $n=l+4+k$). Here we compute (for $r=1,\dots, k+1$) that
\begin{align*}
    P_n^{(l+2)}(\a)(x_r^{(k)})&=(-1)^r(l+2)(\a+2 l),\\
    P_n^{(l+2)}(\a)(x_0^{(k)})&=2 (1 + l) (2 + l). 
\end{align*}
For $l>0$, the factor $(\a+2 l)$ is positive on $(-2,2)$. Therefore, in this case, the mean value theorem gives that $P_{l+4+k}^{(l+2)}(\a)$ has at least $k+1$ distinct roots which may be written as $z^{(k)}_{r}$ with $r=0,1,\dots,k,$ satisfying
\[-2<x_{k+1}^{(k)}<z_k^{(k)}<x_{k}^{(k)}< \dots<x_1^{(k)}<z^{(k)}_{1}<x_1^{(k)}<z^{(k)}_{0}<x_0^{(k)}=2.\]
The distribution of remaining roots (if they exist) depends on the values of $k$ and $l$. For instance, when $l>1$ and $n=l+4+k>6$ (all but finitely many cases when $l>1$), $P_{l+4+k}^{(l+2)}(\a)$ has two further roots $z^{(k)}_{k+1}$ and $z^{(k)}_{k+2}$ which satisfy
\[-\infty< z^{(k)}_{k+2}<- 2 (l+1) <- 2 l <z^{(k)}_{k+1}<-(l+1).\]
These bounds can be determined by evaluating $P_{l+4+k}^{(l+2)}$ at the (integral) roots of the $k=-1,-2,$ cases (\textit{c.f} remark \ref{rem:polys}) using \eqref{eq:chebyx0}, and by considering the behaviour as $|x|\to \infty$.

\mdef We can use similar arguments to obtain bounds on roots for the series $P^{(l+1,1)}_{n}$ and $P^{(l+1,1)^T}_{n}$ also. Here $P^{(l+1,1)}_{l+4+k}$ has degree $k+5$ and $P^{(l+1,1)^T}_{l+4+k}$ has degree $k+4$. Noting that $P^{(l'+1,1)}_{l'+3}(\a)$ is negative on $(-2,2)$ when $l'>4$, and $P^{(l+1,1)}_{l+3}(\a)$ is positive on $(-2,2)$ when $l>3$, we apply Lemma \ref{lem:roots} (now with $k'=1$) to obtain 
\[
    \text{sgn}(P^{(l'+1,1)}_{l'+4+k}(X_r^{(k)}))=-(-1)^r=
    -\text{sgn}(P^{(l+1,1)^T}_{l+4+k}(X_r^{(k)})),
\]  
where $X_r^{(k)}:=2\cos(r \pi/(k+1))$ with $r=1,\dots,k$, for $l'$ and $l$ as such. Thus in each case we observe $k-1$ distinct roots, $Y_r^{(k)}\in (-2,2)$ with $r=1,\dots,k-1$, satisfying
\[-2<X_{k}^{(k)}<Y_{k-1}^{(k)}<X_{k-1}^{(k)}< \dots<X_2^{(k)}<Y^{(k)}_{1}<X_1^{(k)}<2.\]
We observe another two roots $X_1^{(k)}<Y_0^{(k)}<X_{0}^{(k)}=2$ and $-2=X_{k+1}^{(k)}<Y_{k}^{(k)}<X_{k}^{(k)}$ in each case by explicitly verifying
\begin{align*}
    -(-1)^0P^{(l'+1,1)}_{l'+4+k}(X_0^{(k)})&>0, &
    -(-1)^{k+1} P^{(l'+1,1)}_{l'+4+k}(X_{k+1}^{(k)})&>0,\\
    (-1)^{0} P^{(l+1,1)^T}_{l+4+k}(X_{0}^{(k)})&>0, &
    (-1)^{k+1} P^{(l+1,1)^T}_{l+4+k}(X_{k+1}^{(k)})&>0,
\end{align*}
for $l,l'$ as before and $k\geq 0$. 

This leaves (at most) four roots of $P^{(l'+1,1)}_{n}$ unaccounted for, and three of $P^{(l+1,1)^T}_{n}$. These remaining roots can be bounded by evaluating at (integral) roots of the $k=-1,-2$ cases (\textit{c.f} Remark \ref{rem:polys}), and by considering the asymptotic behaviour. We find the following: $P^{(l+1,1)^T}_{n}$ (with $l>3$) has three further roots $Y_{k+1}^{(k)},Y_{-1}^{(k)},Y_{-2}^{(k)}$ which satisfy
\[Y_{k+1}^{(k)}<-2<2<l-1<Y_{-1}^{(k)}<l<l+1<Y_{-2}^{(k)},\]
and $P^{(l'+1,1)}_{n}$ (with $l'>4$) has four further roots $Z_{k+3}^{(k)},Z_{k+2}^{(k)},Z_{k+1}^{(k)},Z_{-1}^{(k)}$ which satisfy
\[Z_{k+3}^{(k)}<-2l'< Z_{k+2}^{(k)}<-l'+1<-l'+2<Z_{k+1}^{(k)}<-l'+3\leq -2<2<Z_{-1}^{(k)}.\]
The roots of the series $P^{(2,2)}_{6+k}(\a)$ can be analysed similarly, but we omit this for brevity. Based on our observations here, we conjecture:

\begin{conjecture} Fix $l$ and let $\l \vdash l+2$. For each $k\geq -1$, the polynomial $P^{(\l)}_{l+4+k}$ is square free. Furthermore, for large enough $k$, $P^{(\l)}_{l+4+k}$ has a fixed number of roots lying outside the interval $[-2,2]$, and the set
\[\bigcup_{k=0}^{\infty } \{\a \in [-2,2] \ | \ P^{(\l)}_{l+4+k}(\a)=0  \},\]
    is dense \ppmm{in $[-2,2]$}.
\end{conjecture}

\section{Gram-determinant-decorated Rollet Graphs}   \label{ss:gdetdecrollet}

Theorem~\ref{thm:main} is the main technical result in the computation of the `bootstrap' data 
for the ToR method for representation theory of KY algebras. 
In \S\ref{ss:repthy} we will see how to apply this directly in representation theory. 
But it is also interesting from a general representation theory point of view to consider KY gram matrices beyond the bootstrap. In the $l=-1$ case there is a very neat pattern, related to the corresponding classical problem for symmetric groups over fields of finite characteristic. 
Here we will evidence that this pattern generalises to other $l$ values. 

In \S\ref{ss:drgl-1} and \S\ref{ss:gdTL} we review the $l=-1$ case. 
From this we see that it is natural to express $\Delttt{n}{p,\l}$ in the form of the corresponding marginal vertex function $\V^n_{(p,\l)}$; and that this `correction' factor is related to ratios of the appropriate Chebyshev polynomials raised to a suitable power. 
In \S\ref{ss:e-next!} we  write this in a way that 
formally generalises to all $l$. 
In \S\ref{ss:gdetl0}-\ref{ss:gdetl2} we compare this to the results of direct computation.

\newcommand{\ournotmain}{arm}

\subsection{Statement of \ournotmain\ property and conjecture \ref{conj:onm0}}
\label{ss:e-next!}


\newcommand{\CC}{{\mathcal C}}  

\mdef   \label{de:arm-property}
Armed with the polynomials $P_n^{(\l)}$, determined in \S\ref{ss:non1dsp}, 
\ppmm{and motivated by the proof of the $l=-1$ formula \eqref{eq:det-recurs} (and finite rank calculations),}
we can define 
\beq  \label{eq:defCC} 
\CC^{\ppmm{n}}_{(p,\lambda)} 
= \prod_{{(p+1,\mu )\; adjacent \; \atop to \; (p,\lambda) \ in \ \Roll_l}} 
     \left( \frac{P_n^{(\mu)}}{P_{n-1}^{(\mu)}} \right)^{dim (\Deltaa^{n-1}_{(p+1,\mu )}) } 
\eq 
- at least for values of $(p,\lambda)$ where the graphical condition makes sense. 

The utility of $\CC^n_{(p,\l)}$ comes from the range of values of the index for which we have 
an identity $\V^n_{(p,\lambda)} = \CC^n_{(p,\lambda)}$ (cf. \ref{eq:MVF1}). 
We call this {\em the \ournotmain\ property}. 
See in particular Conjecture \ref{conj:onm0}. 

\medskip 

A key theorem (extended by a conjecture) we consider in this paper relates the marginal vertex function 
\eqref{eq:MVF1}, 
on the arms of our Rollet graphs 
\ref{de:Rollet}, 
to a family of Chebyshev polynomials defined for each integer partition. 
In other words it gives a way, in principle, to compute almost all KY gram determinants. 


  \begin{theorem}[\bajmm{The arm property}]\label{thm:V=C}
Sufficient conditions on $l,m,$ and $p$, such that 
for each  $\l \vdash l+2$, 
%
the arm $J_{l,p+2m}$ standard modules $\Deltaaa{p+2m}{(p,\l)}{}$ with  $p\geq l+2$ and $m\geq 1$, 
\ppmm{obey the \ournotmain\ property --- that is,   
the gram determinants,} 
satisfy 
\beq   \label{eq:V=C} 
\V^{n=p+2m}_{(p,\l)} = \CC^{n=p+2m}_{(p,\l)}
\eq 
as per (\ref{eq:MVF1}) and (\ref{eq:defCC})
--- are as follows. For $m=1$ the identity holds for all $l$ and $p\geq l+3$.
For $l=-1$ the identity holds for all $n=p+2m$. 
For $l=0$ the identity holds when $m=2$ and $p\leq 8$, and when $m=3$ and $p\leq 5$. For $l=1$ the identity when $m=2$ and $p\leq 7$.
Since we are in the arm, 
equation (\ref{eq:V=C}) is the same as 
to say
\[  
\Delttt{(n)}{p,\l}   
\;\; = \; \left(\frac{P^{(\l)}_{n}}{P^{(\l)}_{n-1}}\right)^{d^{{n-1}}_{(p+1,\l)}} 
\prod_{v \text{ adjacent}\atop \text{to } (p,\l) \text{ in } \Roll_{l}} \!\! \Delttt{(n-1)}{v}    
\;\; ,
\]
where 
$$
d^{{n}}_{(p,\l)}  \; :=  \; \text{dim}(\Deltaaa{n}{(p,\l)})    . 
$$
  \end{theorem}

\begin{proof} For $m=1$, and $p\geq l+3$, \eqref{eq:V=C} becomes
\[\frac{\Delttt{(p+2)}{p,\l}}{\Delttt{(p+1)}{p-1,\l}}=\left(\frac{P^{(\l)}_{n}}{P^{(\l)}_{n-1}}\right)^{d^{p+1}_{(p+1,\l)}}, \]
which follows from Theorem \ref{thm:main} since $d^{p+1}_{(p+1,\l)}=d_\l$.

For the case $l=-1$ note from \eqref{eq:det-recurs} that we have 
the \ournotmain\ property on the nose for all $n$. For the other cases the identity holds by direct calculation over the indicated ranges. 
\end{proof}

\begin{conjecture}   \label{conj:onm0} Formula \eqref{eq:V=C} holds for all $l\geq -1$, $p\geq l+2$ and $m\geq 1$.
\end{conjecture}

\ignore{\begin{conjecture}   \label{conj:onm01}
Fix integers $l\geq -1$ and $n \geq l+2$. For each integer partition $\l \vdash l+2$, there exists 
\ppm{[-well we have already had a stab at saying what they are. if we retract that, this is non-empty but weaker. And the $\CC$ formulation would need adjustment.]}
a family of monic, integral, polynomials $P^{(\l)}_k$ for $k\geq l+2$ satisfying the Chebyshev recursion {as in (\ref{eq:Cheby0})}  
such that, the gram determinants of the $J_{l,n}$ standard modules, $\Delta^{(n)}_{(p,\l)}$ with  $p\geq l+2$, satisfy 
\[
\V^{n}_{(p,\l)} = \CC^{n}_{(p,\l)}
\]
as in (\ref{eq:MVF1}) and (\ref{eq:defCC}), 
that is to say
\[
\frac{|\Delta^{(n)}_{(p,\l)}|}{\prod_{v \text{ adjacent}\atop \text{to } (p,\l) \text{ in } \Roll_{l}} |\Delta^{(n-1)}_{v}|}
= \left(\frac{P^{(\l)}_{p+1}}{P^{(\l)}_{p}}\right)^{d^{\ppm{n-1}}_{(p+1,\l)}},
\]
\ppm{[-so far I don't see this directly agreeing with above! - just something for me to check. I guess the point is the $n\geq l+2$ assumption, but so far this doesn't seem like the right way of putting it to me.]}\bajm{[Actually the $n\geq l+2$ assumption is no assumption at all, the $p\geq l+2$ is the point, this ensures we are on the tails of the rollet which is the only place the conjecture we originally had holds]}
where 
$d^{\ppm{n-1}}_{(p+1,\l)}=\text{dim}(\Delta^{(n-1)}_{(p+1,\l)})$.
\ppm{[-RHS should depend on same things as LHS (unless really indep.).]}

\end{conjecture}}

\subsection{On proof of proven part}   \label{ss:gdetdecrolletl-1}


\subsubsection{Decorated Rollet Graph, for $l=-1$}  \label{ss:drgl-1}

See Fig.\ref{fig:gramdets0312} for the Rollet graph $\Roll_{-1}$ augmented with its fibres of standard gram determinants.
And 
(for the impatient reader, jump ahead to)
see Fig.\ref{fig:gramdets03} for a schematic for the representation theory interpretation. 

\begin{figure} 
\input{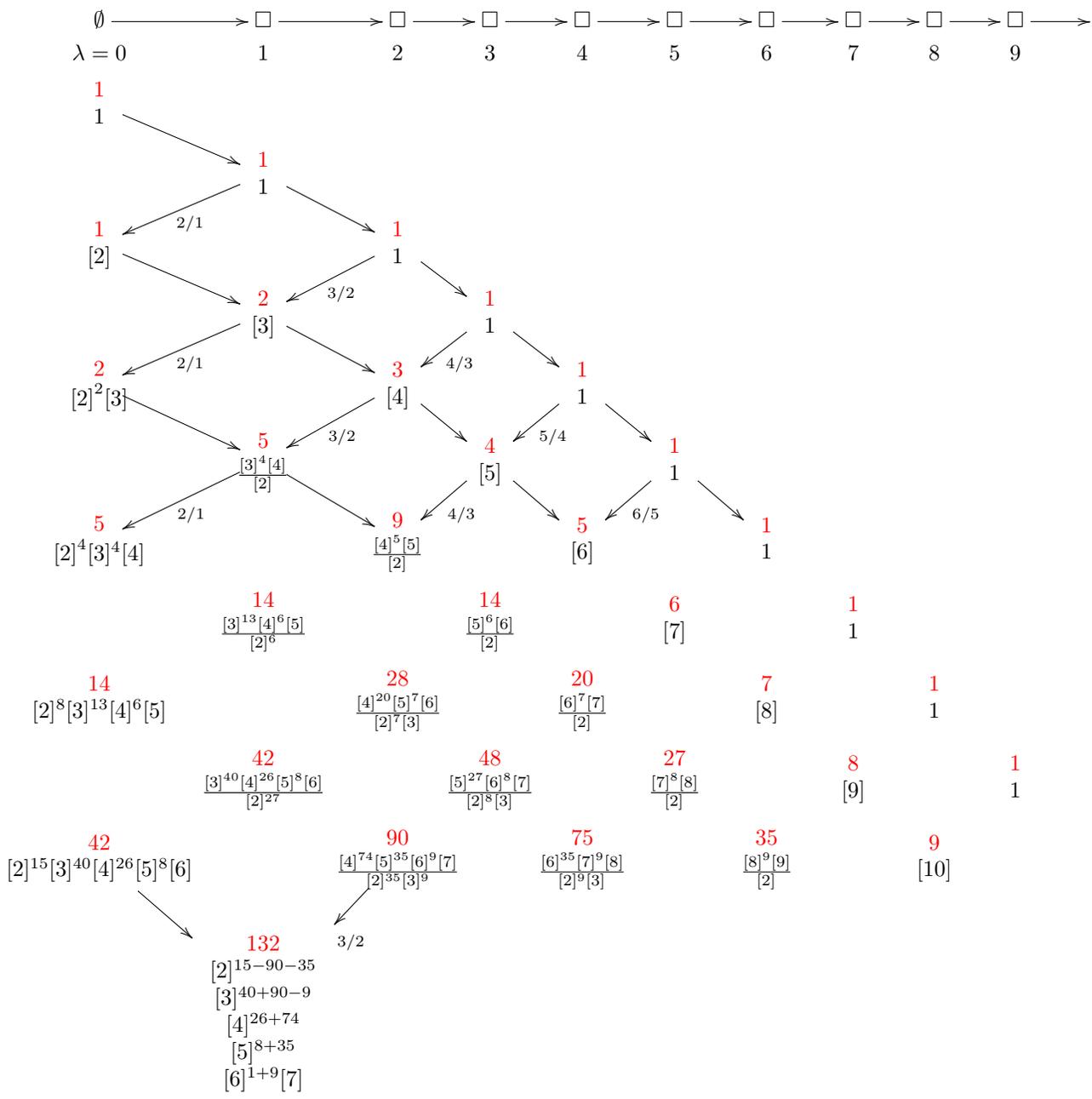}
\caption{The $l=-1$ Rollet graph (top) augmented with many standard gram determinants (parametrised as in Fig \ref{fig:gramdets03}) --- organised with increasing rank $n$ going downwards, each under their fixed standard module label. Edges amongst these fibres of nodes indicate schematically how they are computed from \cite{Martin91}. \label{fig:gramdets0312}}
\end{figure}

\ignore{{
\begin{figure} 
\input{tex/gramdets03}
\caption{The $l=-1$ Rollet graph augmented with many standard gram determinants --- 
and curved arrows indicating standard module morphisms in case $[5]=0$. 
\label{fig:gramdets03}}
\end{figure}
}}

Figure~\ref{fig:gramdets0312} is already quite dense. Let us use this case to illustrate how the information can be presented more concisely using the marginal vertex function.

\mdef \label{th:not-main--1}
For neatness set $d_i = \dim(\Delta_{(i,(1))}^{(n-1)})$. 
Then  $\RollV_{-1}$ is given by  
\[
\begin{tikzcd}[row sep=tiny,column sep=tiny]
    \a^{d_1}\ar[r,dash]&\left(\frac{(\alpha -1) (\alpha +1)}{\alpha }\right)^{d_2} \ar[r,dash] &
    \left(\frac{\alpha  \left(\alpha ^2-2\right)}{(\alpha -1) (\alpha +1)}\right)^{d_3}\ar[r,dash] &
    \left(\frac{\left(\alpha ^2-\alpha -1\right) \left(\alpha ^2+\alpha -1\right)}{\alpha  \left(\alpha^2 -2\right)}\right)^{d_4} \ar[r,dash] &
    \left(\frac{  
    \alpha  (\alpha^2 -1) 
    \left(\alpha ^2-3\right)}{\left(\alpha ^2-\alpha -1\right) \left(\alpha ^2+\alpha -1\right)}\right)^{d_5} \ar[r,dash] &
    \dots
\end{tikzcd}
\]
- the idea is that this captures each fibre complete:
each $d_i$ requires an $n$ as input, so the expression gives the fibre as $n$ varies. 
That is, to recover the fibre innards one plugs in $n$ and $p$ to work out $m$.  

A proof can be found in, for example \cite[Ch.10]{paulsonlinelecturenotes}, 
and we include the relevant extract here in \S\ref{ss:gdTL}.  
In this setting (classical) Chebyshev polynomials arise as part of the geometry of Young diagrams (hook lengths and so on - the machinery of the quantisation of Young's normal forms), or equivalently of walks on the Bratelli diagram. 

(While there are certainly other possible methods, and it is far from clear if the strategy there offers the best chance of a lifting to other $l$,
this method is certainly suggestive of a generalisation, as we will see.)

\mdef  
For example the fixed $m = (n-i)/2$ slice with $m=1$ is: 
\ignore{{
Let us decorate a rollet graph 
\ppm{[it feels like more words are needed here!? for the moment I am confused by what is here. doesn't seem quite right to me. Is this intended to be $\RollV_{-1}$, or $\Roll^{(m)}_{-1}$ for some $m$, or something else?]}
as per the $l=-1$ case
}}
\[
\begin{tikzcd}[row sep=tiny,column sep=tiny]
    \a^2\ar[r,dash]&\frac{(\alpha -1)^3 (\alpha +1)^3}{\alpha ^3}\ar[r,dash]&\frac{\alpha ^4 \left(\alpha ^2-2\right)^4}{(\alpha -1)^4 (\alpha +1)^4}\ar[r,dash]&\frac{\left(\alpha ^2-\alpha -1\right)^5 \left(\alpha ^2+\alpha -1\right)^5}{\alpha ^5 \left(\alpha ^2-2\right)^5}\ar[r,dash]&\frac{(\alpha -1)^6 \alpha ^6 (\alpha +1)^6 \left(\alpha ^2-3\right)^6}{\left(\alpha ^2-\alpha -1\right)^6 \left(\alpha ^2+\alpha -1\right)^6}\ar[r,dash]&\dots
\end{tikzcd}
\]

\subsection{Examples} 

\subsubsection{Decorated Rollet Graphs for $l=0$}   \label{ss:gdetl0}

Here we provide some examples of decorated Rollet graphs for $l=0$ to verify Theorem \ref{thm:V=C}. Firstly, recall the undecorated Rollet graph for $l=0$ is as follows:
\[
\begin{tikzcd}[row sep=tiny,column sep=tiny]&&(2)\ar[r,dash]&(2)\ar[r,dash]&(2)\ar[r,dash]&\dots\\
    \emptyset\ar[r,dash]&(1)\ar[ur,dash]\ar[dr,dash]\\
    &&(1^2)\ar[r,dash]&(1^2)\ar[r,dash]&(1^2)\ar[r,dash]&\dots
\end{tikzcd}
\]

\mdef   \label{pa:l0OEIS}
We will be using the standard module dimensions again. 
The table of standard module dimensions for $l=0$ starts as in Fig.\ref{fig:l=0Brat}. 

\begin{figure} 
\[
\xymatrix@R=-2pt@C=7pt{
1 \ar@[red][ddddrr] \\ \mbox{ } \vspace{.1in} \; \\ \\ \\ 
&& 1 \ar[dddll] \ar[ddrr] \ar[ddddr] \\ \\ 
 &&&& 1 \ar@[red][ddddddll] \ar@[red][dddddrr] \\
1 \ar@[red][dddddrr] \\
 &&& 1 \ar@[red][ddddl] \ar@[red][dddddrr] \\ \\ \\ 
 &&&&&&1 \ar[ddddrr] \ar[ddddll] \\
 &&3 \ar[ddddll] \ar[dddrr] \ar[dddddr] \\ 
 &&&&&1 \ar[ddddrr] \ar[ddddll] \\ \\ 
 &&&&4 \ar@[red][dddddll] \ar@[red][ddddrr] &&&& 1 \ar@[red][ddddrr] \ar@[red][ddddll] \\
 3 \ar@[red][ddddrr] \\
 &&& 4 \ar@[red][dddl] \ar@[red][ddddrr] &&&& 1 \ar@[red][ddddrr] \ar@[red][ddddll] \\ 
 \\ 
 &&&&&&5 \ar[ddddrr] \ar[ddddll] &&&& 1 \ar[ddddrr] \ar[ddddll] \\
&& 11 \ar[ddddll] \ar[dddrr] \ar[dddddr] \\  
&&&&& 5 \ar[ddddrr] \ar[ddddll] &&&& 1 \ar[ddddrr] \ar[ddddll] \\ 
\\ 
&&&& 16 \ar@[red][dddddll] \ar@[red][ddddrr] &&&& 6 \ar@[red][ddddrr] \ar@[red][ddddll] &&&& 1 \ar@[red][ddddrr] \ar@[red][ddddll] \\
11 \ar@[red][ddddrr] \\
&&& 16 \ar@[red][dddl] \ar@[red][ddddrr] &&&& 6 \ar@[red][ddddrr] \ar@[red][ddddll] &&&& 1 \ar@[red][ddddrr] \ar@[red][ddddll] \\ 
\\ 
 &&&&&& 22 \ar[ddddrr] \ar[ddddll] &&&& 7 \ar[ddddrr] \ar[ddddll] &&&& 1 \ar[ddddll] \ar[ddddrr] \\
&& 43 \ar[ddddll] \ar[dddrr] \ar[dddddr] \\  
&&&&&  22 \ar[ddddrr] \ar[ddddll] &&&& 7 \ar[ddddrr] \ar[ddddll] &&&& 1 \ar[ddddll] \ar[ddddrr] \\ 
\\ 
&&&& 65 &&&& 29 &&&& 8 &&&& 1 \\
43 \\
&&& 65 &&&& 29 &&&& 8 &&&& 1 \\
}
\]
\caption{ Full Bratelli diagram for $l=0$. Edges from an even $n$ to an odd $n$ layer are coloured red, just to aid visibility of layers. \label{fig:l=0Brat} }
\end{figure}
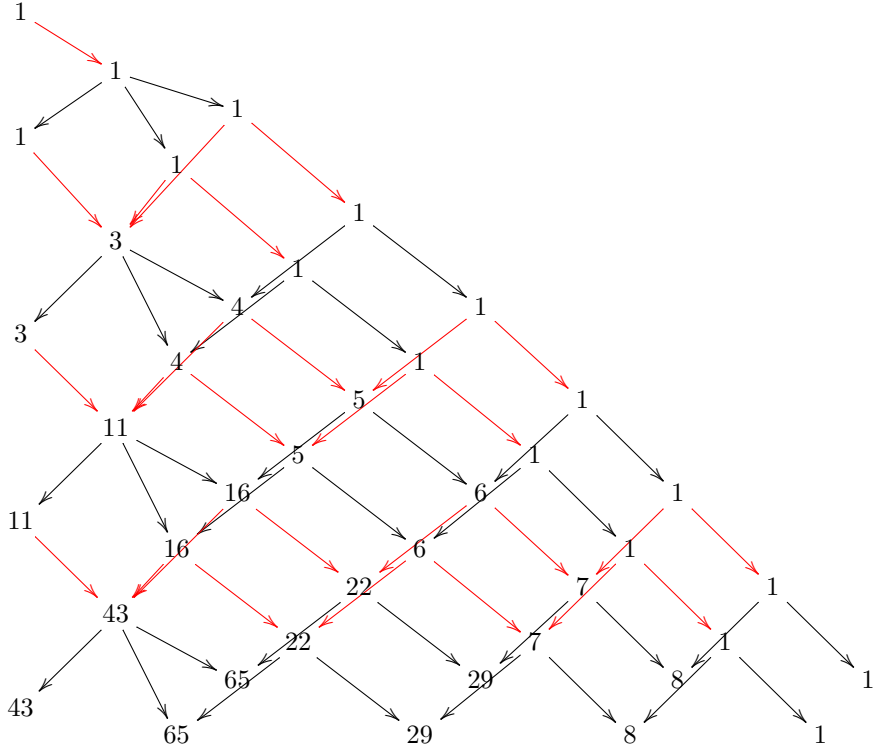

\begin{figure}
\input{tex/Bratelli-l0-gramdets}
    \caption{Full Bratelli for $l=0$ giving gram determinants (with module dimensions below in red). 
    The top row indicates the conjectured ($n$-independent up to dimensions) edge factors. }
    \label{fig:Bratelli-l0-gramdets}
\end{figure}

\mdef 
The full Bratelli diagram for $l=0$ giving gram determinants (with dimensions below in red) 
starts as in Fig.\ref{fig:Bratelli-l0-gramdets}. 
    The very top row in the figure indicates the conjectured ($n$-independent up to dimensions giving exponents) edge factors. 

\begin{proposition}\label{pr:1cupgram}
The 1-cup gram-det-decorated graph $\Roll^1_0$ (as in \ref{de:Roll-etc}) is   
\[  \hspace{-1.5cm} 
\begin{tikzcd}[row sep=tiny,column sep=tiny]
&&  \overbrace{\alpha (\alpha-1)(\alpha^2 +x-4)}^{\Delttt{4}{2,(2)} \;=\; C^{(2)}P^{(2)}_4} \ar[r,dash] &
  \overbrace{(\alpha-1)(\alpha^4 +\alpha^3 -5x^2 -x+2)}^{C^{(2)}P^{(2)}_{5}} \ar[r,dash] &
  C^{(2)}P^{(2)}_{n=6} \ar[r,dash] & \dots
  \\
    \alpha \ar[r,dash] &
    \overbrace{
    (\alpha-1)^2 (\alpha+2)   }^{\Delttt{3}{1,(1)}  }   \ar[ur,dash]\ar[dr,dash]
    \\
    &&\underbrace{(\alpha^2 -1)(\alpha^2 -4)}_{C^{(1^2)}P^{(1^2)}_4} \ar[r,dash] &
    \underbrace{(\alpha-1)(\alpha+2)(\alpha^3 -\alpha^2 -3\alpha+1)}_{C^{(1^2)}P^{(1^2)}_5} \ar[r,dash] &
    C^{(1^2)}P^{(1^2)}_{n=6} \ar[r,dash] & \dots
\end{tikzcd}
\]
\end{proposition}
\begin{proof}
    The $p=0,1$ vertices are by direct calculation. Also the first two arm vertices in each arm. And thereafter we may use 
\ref{eq:Cheby0} and in particular \ref{thm:main} and so on as explained above, 
noting that $\dim(\Deltaaa{n}{(n,\l)}) = 1$ here. 
\end{proof}



\begin{proposition}\label{pr:l0}
(I) The graph $\Roll^{(1)}_0 $ 
(in which the $(p,\l)$ vertex is $\V^{p+2}_{(p,\l)}$)
is given by      
\[ \hspace{-1.5cm} 
\begin{tikzcd}[row sep=0em,column sep=0.5em]
&&
\overbrace{
\tfrac{\alpha^{} \left(\alpha^2+\alpha -4\right)^{}}{(\alpha -1)^{} (\alpha +2)^{}} }^{ \frac{P^{(2)}_{4}}{P^{(2)}_{3}} }  \ar[r,dash]
  &
  \overbrace{
  \tfrac{\left(\alpha^4+\alpha^3-5 \alpha^2-\alpha +2\right)^{}}{\alpha^{} \left(\alpha ^2+\alpha -4\right)^{}}}^{ \V^{5}_{(3,(2))} = \frac{P^{(2)}_{5}}{P^{(2)}_{4}}} \ar[r,dash]
  &
  \overbrace{
  \tfrac{(\alpha -1)^{} \alpha^{} \left(\alpha ^3+2 \alpha ^2-4 \alpha -6\right)^{}}{\left(\alpha ^4+\alpha ^3-5 \alpha ^2-\alpha +2\right)^{}} }^{\frac{P^{(2)}_{6}}{P^{(2)}_{5}} \ppm{}}   \ar[r,dash]&
  \dots
\\
\alpha^{}\ar[r,dash] &
  \overbrace{
 (\alpha-1)^{} \tfrac{(\alpha-1)^{} (\alpha+2)^{}}{\alpha^{}}    }^{ \propto \frac{P^{(2)}_{3}}{P^{(2)}_{2}} \frac{P^{(1^2)}_{3}}{P^{(1^2)}_{2}}  }    \ar[ur,dash]\ar[dr,dash]
    \\
    &&
    \underbrace{
    \tfrac{(\alpha -2)^{} (\alpha +1)^{}}{(\alpha -1)^{}}  }_{ \frac{P^{(1^2)}_{4}}{P^{(1^2)}_{3}}  }  \ar[r,dash]
      &
      \underbrace{
    \tfrac{\left(\alpha ^3-\alpha ^2-3 \alpha +1\right)^{}}{(\alpha -2)^{} (\alpha +1)^{}} }_{\frac{P^{(1^2)}_{5}}{P^{(1^2)}_{4}}} \ar[r,dash]
    &
    \underbrace{
    \tfrac{(\alpha -1)^{} \left(\alpha ^3-4 \alpha -2\right)^{}}{\left(\alpha ^3-\alpha ^2-3 \alpha +1\right)^{}} }_{\frac{P^{(1^2)}_{\ppmm{6}}}{P^{(1^2)}_{5}}} \ar[r,dash]&\dots
\end{tikzcd}
\]
%
(II) The graph $\Roll^{(2)}_0 $ \ppmm{(i.e. for $m=2$,  
so $n=4$ for vertex $\emptyset$; 
$n=5$ for vertex $(1)$; 
$n=6$ for $(2,(2))$ and $(2,(1^2))$; 
$n=7$ for $(3,(2))$ and $(3,(1^2))$; 
and so on)} 
starts (up to $p=8$):  
\[  \hspace{-1.5cm} 
\begin{tikzcd}[row sep=0em,column sep=0.5em]
&&\frac{\alpha ^5 \left(\alpha ^2+\alpha -4\right)^5}{(\alpha -1)^5 (\alpha +2)^5}\ar[r,dash]
  &
  \overbrace{
\tfrac{\left(\alpha ^4+\alpha ^3-5 \alpha ^2-\alpha +2\right)^6}{\alpha ^6 \left(\alpha ^2+\alpha -4\right)^6} 
 }^{\V^7_{(3,(2))} = \left( \frac{ P^{(2)}_{5}}{P^{(2)}_{4}} \right)^{d^6_{(4,(2))}} }  \ar[r,dash]
  &
  \frac{(\alpha -1)^7 \alpha ^7 \left(\alpha ^3+2 \alpha ^2-4 \alpha -6\right)^7}{\left(\alpha ^4+\alpha ^3-5 \alpha ^2-\alpha +2\right)^7}\ar[r,dash]
  & \dots
  \\
    \alpha^3\ar[r,dash]&\frac{(\alpha -1)^8 (\alpha +2)^4}{\alpha ^4}\ar[ur,dash]\ar[dr,dash]
    \\
    &&
    \underbrace{
    \tfrac{(\alpha -2)^5 (\alpha +1)^5}{(\alpha -1)^5}   }_{ \left( \frac{ P^{(1^2)}_{4}}{P^{(1^2)}_{3}} \right)^{d^5_{(3,(1^2))}} }   \ar[r,dash]
      &
      \frac{\left(\alpha ^3-\alpha ^2-3 \alpha +1\right)^6}{(\alpha -2)^6 (\alpha +1)^6}\ar[r,dash]
      &
      \frac{(\alpha -1)^7 \left(\alpha ^3-4 \alpha -2\right)^7}{\left(\alpha ^3-\alpha ^2-3 \alpha +1\right)^7}\ar[r,dash]
      & \dots
\end{tikzcd}
\]
(III)
The graph $\Roll^{(3)}_0$ starts (up to $p=5$):
\[\begin{tikzcd}[row sep=0em,column sep=0.5em]&&\frac{\alpha ^{22} \left(\alpha ^2+\alpha -4\right)^{22}}{(\alpha -1)^{22} (\alpha +2)^{22}}\ar[r,dash]&\frac{\left(\alpha ^4+\alpha ^3-5 \alpha ^2-\alpha +2\right)^{29}}{\alpha ^{29} \left(\alpha ^2+\alpha -4\right)^{29}}\ar[r,dash]&\frac{(\alpha -1)^{37} \alpha ^{37} \left(\alpha ^3+2 \alpha ^2-4 \alpha -6\right)^{37}}{\left(\alpha ^4+\alpha ^3-5 \alpha ^2-\alpha +2\right)^{37}}\ar[r,dash]&\dots\\
    \alpha^{11}\ar[r,dash]&\frac{(\alpha -1)^{32} (\alpha +2)^{16}}{\alpha ^{16}}\ar[ur,dash]\ar[dr,dash]\\
    &&\frac{(\alpha -2)^{22} (\alpha +1)^{22}}{(\alpha -1)^{22}}\ar[r,dash]&\frac{\left(\alpha ^3-\alpha ^2-3 \alpha +1\right)^{29}}{(\alpha -2)^{29} (\alpha +1)^{29}}\ar[r,dash]&\frac{(\alpha -1)^{37} \left(\alpha ^3-4 \alpha -2\right)^{37}}{\left(\alpha ^3-\alpha ^2-3 \alpha +1\right)^{37}}\ar[r,dash]&\dots
\end{tikzcd}\]
\end{proposition}
\noindent 
{\em Proof}. 
(I) follows from \ref{pr:1cupgram}. 
\ppmm{(II,III) So far the Prop. is to be understood only to apply for the cases explicitly shown. Thus proof is by direct calculation.  \qed}

\ignore{\mdef
\bajm{[delete from here?]} 
\ppm{
Observe from Fig.\ref{fig:l=0Brat}, 
and cf. (\ref{eq:l=0+}), 
that the exponents follow (essentially) the exact same rule as in the $l=-1$ case, in (\ref{th:not-main--1}). 
The fine tuning is as follows. Firstly note that vertex $(1)$ is, in a sense, a new kind of case - since trivalent - and has non-trivial `edge factors' from two edges ... 
}

From this we have a conjecture for all $n$. 

\mdef 
The next steps (which can be done in either order!) will be:
\\
$\bullet\;$ Prove it.
\\
$\bullet\;$ Interrogate for the structure of standard modules - see \S\ref{ss:repthy}. 
\\
Then:
\\
$\bullet\;$  Lift to $l=1$ etc.!

\pagebreak

\mdef \bajm{[keep or not?]}
Let us have a look at $n=7$ row of determinants:
\[ \hspace{-2cm} 
\begin{tikzcd}
&
\baa{
(\alpha -2) (\alpha -1)^{10} \\ (\alpha +1) (\alpha +2) \\\left(\alpha ^3+2 \alpha ^2-4 \alpha -6\right) \\\left(\alpha ^4+\alpha ^3-5 \alpha ^2-\alpha +2\right)^7  \hspace{-1cm} \\ \ppm{(3,(2))}\\\bajm{22}
}    \hspace{-1cm} 
&
\baa{(\alpha -1) \\
\left(\alpha ^6+\alpha ^5-7 \alpha ^4-3 \alpha ^3+11 \alpha ^2+\alpha -2\right)\\\ppm{(5,(2))}\\\bajm{7}
}
&
\baa{1\\\ppm{(7,(2))}\\\bajm{1}
}
\\
\baa{
(\alpha -2)^8 (\alpha -1)^{60} (\alpha +1)^8\\ (\alpha +2)^{30} \left(\alpha ^2+\alpha -4\right)^8 \\\left(\alpha ^3-\alpha ^2-3 \alpha +1\right) \\\left(\alpha ^4+\alpha ^3-5 \alpha ^2-\alpha +2\right)\\\ppm{(1,(1))}\\\bajm{43}
}  \hspace{-1cm} 
\\
&
\baa{(\alpha -1)^{10} (\alpha +2)^8 \\ \left(\alpha ^2+\alpha -4\right)\\ \left(\alpha ^3-4 \alpha -2\right) \\ \left(\alpha ^3-\alpha ^2-3 \alpha +1\right)^7 \\ \ppm{(3,(1^2))}\\\bajm{22}
}  \hspace{-1cm} 
&
\baa{(\alpha -1) (\alpha +1) (\alpha +2)\\
\left(\alpha ^4-2 \alpha ^3-3 \alpha ^2+6 \alpha -1\right)\\\ppm{(5,(1^2))}\\\bajm{7}
}
&
\baa{1\\\ppm{(7,(1^2))}\\\bajm{1}
}
\end{tikzcd}\]
How to account for the factor $\left(\alpha ^2+\alpha -4\right)^8$ in $\Delttt{7}{1,(1)}$? 
The answer is now in Fig.\ref{fig:standardmapshenri}. 

\medskip }

\subsubsection{Decorated Rollet Graph, for $l=1$} \label{ss:DRG1}
Here we provide some examples of decorated Rollet graphs for $l=1$ to verify Theorem \ref{thm:V=C}. Firstly, recall the undecorated Rollet graph for $l=1$ is as follows:
\[
\begin{tikzcd}[column sep=0.6em,row sep=0em]
&&&(3)\ar[dash,r]&(3)\ar[dash,r]&(3)\ar[dash,r]&{}\\
&&(2)\arrow[dash,ur]\arrow[dash,dr]\\
\emptyset\arrow[dash,r] & (1)\arrow[dash,ur]\arrow[dash,dr] &&(2,1)\ar[dash,r]&(2,1)\ar[dash,r]&(2,1)\ar[dash,r]&{}\\
&&(1^2)\arrow[dash,ur]\arrow[dash,dr]\\
&&&(1^3)\ar[dash,r]&(1^3)\ar[dash,r]&(1^3)\ar[dash,r]&{}
\end{tikzcd}
\]
\ignore{Each node, then, labels an infinite family of $J_{1,n}$-modules for $n=p,p+2,p+4,\dots$. 
In the case of $n=p$, the corresponding module can simply be regarded as the 
\ppmm{image of the}
Specht module $\mathcal{S}^{\lambda}$, with any non-fully propagating diagram from $J_{l,n}$ acting as 0. 
For $n=p+2k$, we sometimes refer to the corresponding module as a $k$-cup module since a set of basis elements for the $J_{l,n}$-module can be described by pairs consisting of basis element for $\mathcal{S}^{\lambda}$, together with a half-diagram of height $\leq l$ with $k$-cups (\textit{cf.} \cite[Theorem something]{KadarYu})
\ppm{[-fix me!]}. 
We include the case $k=0$ in this terminology.

Now let us decorate this Rollet graph with the gram determinants of the $k$-cup modules (for $k>0$, since the $k=0$ case is not interesting \ppm{[LOL! you mean its a simple rep?]} over $\C$). 

\mdef 
\ppm{[explain a bit more!? this chunk seems to have been pulled from nowhere? linear combinations just correspond to tweaking the initial conditions, right?... Can this, then, be dragged to (or explained in terms of) the 2nd Cheby section, \S\ref{s:roots}?]}
To give some of the determinants here 
we use the following notation:
\begin{align*}
    F^{(3)}_{n}(\a)&=U_{n-2}(\a/2)+4U_{n-3}(\a/2)-U_{n-4}(\a/2)-2U_{n-5}(\a/2)-2U_{n-6}(\a/2)\\
    F^{(2,1)}_{n}(\a)&=U_{n-1}(\a/2)-4U_{n-3}(\a/2)-4U_{n-5}(\a/2)-2U_{n-7}(\a/2)\\
    F^{(1^3)}_{n}(\a)&=U_{n-3}(\a/2)-2U_{n-4}(\a/2)-2U_{n-5}(\a/2)
\end{align*}
\ppm{where $U_n(x)$ is as defined in \ref{where??} above.}

These polynomials (seem to  \ppm{[-need to deal with this somehow. ideally facts kept well away from speculations - in different sections.]}) satisfy the following:
\begin{itemize}
    \item $F^{(3)}_{n}(\a)$ is monic of degree $n-2$. It has $n-3$ roots of the form $z_k^{(n)}=2\cos (\th_k)$ for $\frac{(k-1)\pi }{n-3}<\th_k<\frac{k\pi }{n-3}$, for $k=1,\dots, n-3$, and one root of the form $z_0^{(n)}=-2 \cosh(\th_0)$ for $\th_0\in \R_{>0}$. Furthermore, $\lim_{n\to \infty} z_0^{(n)}$ seems to exist
    $\lim_{n\to \infty} z_0^{(n)}\approx -4.39356$.
    \item $F^{(1^3)}_{n}(\a)$ is monic of degree $n-3$. It has $n-4$ roots of the form $y_k^{(n)}=2\cos (\phi_k)$ for $\frac{k\pi }{n-3}<\phi_k<\frac{(k+1)\pi }{n-3}$, for $k=1,\dots, n-4$, and one root of the form $y_0^{(n)}=2 \cosh(\phi_0)$ for $\phi_0\in \R_{>0}$. Furthermore, $\lim_{n\to \infty} y_0^{(n)}$ seems to exist and $\lim_{n\to \infty} y_0^{(n)}\approx 3.09808$.
    \item $F^{(2,1)}_{n}(\a)$ is monic of degree $n-1$. It has $n-3$ roots of the form $x_k^{(n)}=2\cos (\psi_k)$ for $\frac{(2k-1)\pi }{2(n-2)}<\psi_k<\frac{(2k+1)\pi }{2(n-2)}$, for $k=1,\dots, n-3$, and two roots of the form $x_{0,\pm}^{(n)}=\pm 2 \cosh(\psi_0)$ for $\psi_0\in \R_{>0}$. Furthermore, $\lim_{n\to \infty} x_{0,+}^{(n)}$ seems to exist and $\lim_{n\to \infty} x_{0,+}^{(n)}\approx 2.66529$.
\end{itemize}
With this in mind, we proceed. 

\mdef 
The Rollet graph \ppm{[which one? gram? yes]} is given on the next page. 
We fill this graph with as much as we know for now - the $1$-cup gram det and the 2-cup one above this:
\ppm{[darn it, I hoped we'd dealt with the landscape problem. not sure yet what's achievable here. LOL, I guess this one is just even wider than before.]}}

\begin{landscape}
\begin{figure}
    \label{fig:placeholder}\[
\begin{tikzcd}[column sep=0.6em,row sep=0em]
&&&\begin{smallmatrix} \a (\a-2)^{8} (\a+2)^4(\a+4)\\
\times (\a^3+4\a^2-4\a-10){(\a^4-7\a^2+3)^2}\\
\times \left[(\a-2)^2\a^2(\a+1)(\a^2+3 \a-6)\right]^9
\\[0.5em] \hline  
    (\a-2)^2\a^2(\a+1)(\a^2+3 \a-6)
\end{smallmatrix}\ar[dash,r]&\dots\ar[dash,r]&
\begin{smallmatrix}(\a+4)(\a+2)^4{(\a^4-7\a^2+3)^2}\\\times (\a-2)^8 \left[P^{(3)}_{p+3}(\a)\right]\\  \times\left[\a^2(\a-2)^2 \left[P^{(3)}_{p+2}(\a)\right]\right]^{p+6} \\[0.5em] \hline 
    \a^2(\a-2)^2 \left[P^{(3)}_{p+2}(\a)\right]
\end{smallmatrix}\ar[dash,r]&{}\\
&&\begin{smallmatrix}
    (\a-2)^{8} (\a+1)(\a+4) (\a+2)^{4}\\
    \times(\a^2+3\a-6){(\a^4-7\a^2+3)^2}\\
    \times \left[\a^3 (\a-2)^2 (\a+4)\right]^8\\[0.5em] \hline 
    \a^3(\a-2)^2(\a+4)
\end{smallmatrix}\arrow[dash,ur]\arrow[dash,dr]\\
\begin{smallmatrix}\a^3 \,(\a-1)^2  (\a+2)\\[0.5em] \hline 
\a\end{smallmatrix} \arrow[dash,r] & \begin{smallmatrix}(\a+4)(\a-2)^5(\a+2)^3\\\times \left[ (\a-1)^{2} (\a+2) \right]^7  \\[0.5em] \hline 
    (\a-1)^2(\a+2)
\end{smallmatrix}\arrow[dash,ur]\arrow[dash,dr] &&\begin{smallmatrix}(\alpha -3) (\alpha -2)^{42} (\alpha -1)^2  \alpha ^{12}\\ 
\times(\alpha +1)^4 (\alpha +2)^{17}(\alpha +4)^{12} \left(\alpha ^2-7\right)^2\\
\times\left(\alpha ^2+3 \alpha -6\right) \left(\alpha ^4-7 \alpha ^2+3\right)^{20}\\[0.5em] \hline
    (\alpha -2)^3 \alpha  (\alpha +2) (\alpha +4) {\left(\alpha ^4-7 \alpha ^2+3\right)^2}
\end{smallmatrix}\ar[dash,r]&\dots\ar[dash,r]&\begin{smallmatrix}
\frac{1}{\a^6}
    \left[\Delttt{(p+2)}{p,(2,1)}(\a)\right]^{p+6}
    \Delttt{(p+3)}{p+1,(2,1)}(\a)\Delttt{(4)}{2,(1^2)}(\a)\\
    \times \Delttt{(4)}{2,(2)}(\a)\Delttt{(5)}{3,(1^3))}(\a)\Delttt{(5)}{3,(3)}(\a)\Delttt{(5)}{3,(2,1)}(\a)\\[0.5em] \hline
    (\a-2)^3\a (\a+2)(\a+4) \left[P^{(2,1)}_{p+2}(\a)\right]^2
\end{smallmatrix}\ar[dash,r]&{}\\
&&\begin{smallmatrix}(\a-2)^{7}(\a+2)^{4}(\a+4)^2(\a+1)\\
\times\a (\a-3){(\a^4-7\a^2+3)^2}\\
\times \left[(\a-2)^3(\a+2)^3\right]^8\\[0.5em] \hline
    (\a-2)^3(\a+2)^3
\end{smallmatrix}\arrow[dash,ur]\arrow[dash,dr]\\
&&&\begin{smallmatrix}\a (\a-2)^{7} (\a+2)^{4}  (\a+4)^2 \\
\times(\a^3-2\a^2-4\a+2){(\a^4-7\a^2+3)^2}\\
\times \left[(\a-2)^2(\a-3 )(\a+1)(\a+2)^3\right]^9 \\[0.5em] \hline
    (\a-2)^2(\a-3 )(\a+1)(\a+2)^3
\end{smallmatrix}\ar[dash,r]&\dots\ar[dash,r]&
    \begin{smallmatrix} \a (\a+4)^2  {(\a^4-7\a^2+3)^2}\\
    \times (\a-2)^7(\a+2)^4 \left[P^{(1^3)}_{p+3}(\a)\right] \\
    \times \left[(\a-2)^2(\a+2 )^3 \left[P^{(1^3)}_{p+2}(\a)\right]\right]^{p+6}\\[0.5em] \hline
        (\a-2)^2(\a+2 )^3\left[ P^{(1^3)}_{p+2}(\a)\right]
    \end{smallmatrix}\ar[dash,r]&{}
\end{tikzcd}\]    
    \caption{Gram decorated rollet graph for $l=1$. We place two non-trivial determinants on each vertex (i.e. the cases $n=p+2,p+4$ on vertex $(p,\l)$). This general form of two cup determinants in the $\l$-arms for the row and column partitions is verified up to $p=7$.}
\end{figure}
\end{landscape}

\ignore{\bajm{[delete from here ***]}
Of course we know from the structure theorem for the Brauer algebra \cite{Martin09} that all the roots in the head part of the graph are integral  in these low ranks 
(indeed we know the gram determinants in the same way) until $n$ is large enough to have standard modules in the \arm s of the graph. 
Thus in the 1-cup Rollet section all the roots are integral until we reach the \hip: 
\ppm{[Caption this figure?:]}
\[
\begin{tikzcd}[column sep=0.6em,row sep=0em]
&&&\begin{smallmatrix} 
    (2)^2(0)^2(-1)\bajm{\dots}
\end{smallmatrix}\ar[dash,r]&\dots\ar[dash,r]&\begin{smallmatrix}
    (0)^2(2)^2 \left[F^{(3)}_{p+2}(\a)\right]
\end{smallmatrix}\ar[dash,r]&{}\\
&&\begin{smallmatrix}
    (0)^3(2)^2(-4)
\end{smallmatrix}\arrow[dash,ur]\arrow[dash,dr]\\
\begin{smallmatrix}(0)^1\end{smallmatrix} \arrow[dash,r] & \begin{smallmatrix}    (1)^2(-2)
\end{smallmatrix}\arrow[dash,ur]\arrow[dash,dr] &&\begin{smallmatrix}
    (2)^3 (0)  (-2) (-4) \bajm{\dots}
\end{smallmatrix}\ar[dash,r]&\dots\ar[dash,r]&\begin{smallmatrix}
    (-2)^3 (0) (-2)(-4) \left[F^{(2,1)}_{p+2}(\a)\right]^2
\end{smallmatrix}\ar[dash,r]&{}\\
&&\begin{smallmatrix}    (2)^3(-2)^3
\end{smallmatrix}\arrow[dash,ur]\arrow[dash,dr]\\
&&&\begin{smallmatrix}    (2)^2(3 )(-1)(-2)^3
\end{smallmatrix}\ar[dash,r]&\dots\ar[dash,r]&
    \begin{smallmatrix}         (2)^2(-2 )^3\left[ F^{(1^3)}_{p+2}(\a)\right]
    \end{smallmatrix}\ar[dash,r]&{}
\end{tikzcd}\] 

Let $\left|\Deltaaa{n}{(p,\l)}\right|$ denote the gram-det of the $J_{1,n}$-module labelled by the pair $(p,\l)$. We seem to have \ppm{[-time to adjust this phrasing.]} observed the following factorisation rules for $2$-cup determinants in terms of $1$-cup determinants:
\begin{align*}
 \left|\Delta_{(p,(1^3))}^{(p+4)}\right|&=\frac{1}{\a^3}
    \left|\Delta_{(p,(1^3))}^{(p+2)}\right|^{p+6}
    \left|\Delta_{(p+1,(1^3))}^{(p+3)}\right|
    \left|\Delta_{(2,(2))}^{(4)}\right|
    \left|\Delta_{(3,(2,1))}^{(5)}\right|\\
 \left|\Delta_{(p,(3))}^{(p+4)}\right|&=\frac{1}{\a^3}
    \left|\Delta_{(p,(3))}^{(p+2)}\right|^{p+6}
    \left|\Delta_{(p+1,(3))}^{(p+3)}\right|
    \left|\Delta_{(2,(1^2))}^{(4)}\right|
    \left|\Delta_{(3,(2,1))}^{(5)}\right|\\
 \left|\Delta_{(p,(2,1))}^{(p+4)}\right|&=\frac{1}{\a^6}
    \left|\Delta_{(p,(2,1))}^{(p+2)}\right|^{p+6}
    \left|\Delta_{(p+1,(2,1))}^{(p+3)}\right|
    \left|\Delta_{(2,(1^2))}^{(4)}\right|\left|\Delta_{(2,(2))}^{(4)}\right|\\
    &\quad \times\left|\Delta_{(3,(1^3))}^{(5)}\right|\left|\Delta_{(3,(3))}^{(5)}\right|   
    \left|\Delta_{(3,(2,1))}^{(5)}\right|\\
\end{align*}

\bajm{[to here ***]}}

\mdef \label{pa:l1}
We now include 
(by direct computation)
the \ppmm{section $\Roll^{(2)}_{1}$ through the }
decorated Rollet graph $\RollV_1$ as per the $l=0$ case in \ref{pr:l0} II. This computation completes the cases remaining in Theorem \ref{thm:V=C} (the $p=7$ case is omitted since it did not fit in the margins).
\[
\hspace{-4.51cm} 
\begin{tikzcd}[column sep=0.6em,row sep=0em]
&&&\frac{(\alpha +1)^8 \left(\alpha ^2+3 \alpha -6\right)^8}{\alpha ^8 (\alpha +4)^8}\ar[dash,r]
   &
   \overbrace{
   \tfrac{\alpha ^9 \left(\alpha ^3+4 \alpha ^2-4 \alpha -10\right)^9}{(\alpha +1)^9 \left(\alpha ^2+3 \alpha -6\right)^9}   }^{ \left( \frac{ P^{(3)}_{\ppmm{6}}}{P^{(3)}_{5}} \right)^9}   \ar[dash,r]
   &
   \overbrace{
   \tfrac{\left(\alpha ^5+4 \alpha ^4-5 \alpha ^3-14 \alpha ^2+3 \alpha +6\right)^{10}}{\alpha ^{10} \left(\alpha ^3+4 \alpha ^2-4 \alpha -10\right)^{10}}   }^{  \left( \frac{ P^{(3)}_{\ppmm{7}}}{P^{(3)}_{6}} \right)^{10}  }   \ar[dash,r]
   &\dots
   \\
&&\frac{(\alpha -2)^{14} \alpha ^{14}}{(\alpha -1)^{14}} \frac{\a^7(\alpha +4)^7}{(\alpha +2)^7}\arrow[dash,ur]\arrow[dash,dr]
\\
\alpha ^3\arrow[dash,r] & (\alpha -1)^{6}\frac{(\alpha -1)^{6} (\alpha +2)^6}{\alpha ^6}\arrow[dash,ur]\arrow[dash,dr] 
  &&
  \underbrace{
  \tfrac{\left(\alpha ^4-7 \alpha ^2+3\right)^{16}}{(\alpha -2)^{16} \alpha ^{16} (\alpha +2)^{16}} }_{  \left( \frac{ P^{(21)}_{\ppmm{5}}}{P^{(21)}_{4}} \right)^{16}     } 
    \ar[dash,r]
  &
  \frac{(\alpha -1)^{18} \alpha ^{18} (\alpha +1)^{18} \left(\alpha ^2-7\right)^{18}}{\left(\alpha ^4-7 \alpha ^2+3\right)^{18}}\ar[dash,r]&\frac{\left(\alpha ^6-9 \alpha ^4+14 \alpha ^2-3\right)^{20}}{(\alpha -1)^{20} \alpha ^{20} (\alpha +1)^{20} \left(\alpha ^2-7\right)^{20}}\ar[dash,r]&\dots
\\
&&(\alpha -2)^{7}\frac{(\alpha -2)^{14} (\alpha +2)^{14}}{(\alpha -1)^{14}}\arrow[dash,ur]\arrow[dash,dr]
\\
&&&
  \underbrace{
  \tfrac{(\alpha -3)^8 (\alpha +1)^8}{(\alpha -2)^8} }_{ \left( \frac{ P^{(1^3)}_{\ppmm{5}}}{P^{(1^3)}_{4}} \right)^{8}    }     \ar[dash,r]
  &
  \frac{\left(\alpha ^3-2 \alpha ^2-4 \alpha +2\right)^9}{(\alpha -3)^9 (\alpha +1)^9}\ar[dash,r]&\frac{\left(\alpha ^4-2 \alpha ^3-5 \alpha ^2+4 \alpha +3\right)^{10}}{\left(\alpha ^3-2 \alpha ^2-4 \alpha +2\right)^{10}}\ar[dash,r]&\dots
\end{tikzcd}
\]    
Where the $P_{n}^{(\l)}$ are as computed in \S\ref{ss:examples0} (in particular, \textit{cf.} \ref{pa:1dgdets} and \ref{pa:21dets}).

Observe that
\[
\CC^{6}_{(2,(2))} = 
\left( \frac{ P^{(21)}_{\ppmm{4}}}{P^{(21)}_{3}} \right)^{14} 
\left( \frac{ P^{(3)}_{\ppmm{4}}}{P^{(3)}_{3}} \right)^{7} 
=  \left( \frac{\xx(\xx^2 -4)}{3(\xx^2 -1)}  \right)^{14} 
\left( \frac{\xx (\xx+4)}{3(\xx+2)}  \right)^{7} 
\;\propto \; (\xx+2)^{14} \; \V^6_{(2,(2))}
\]
\[
\CC^{6}_{(2,(1^2))} =
\left( \frac{ P^{(21)}_{\ppmm{4}}}{P^{(21)}_{3}} \right)^{14} 
\left( \frac{ P^{(1^3)}_{\ppmm{4}}}{P^{(1^3)}_{3}} \right)^{7} 
=  \left( \frac{\xx(\xx^2 -4)}{3(\xx^2 -1)}  \right)^{14} 
\left( \frac{ (\xx-2)}{3}  \right)^{7} 
\;\propto  \; \xx^{14} \; \V^6_{(2,(1^2))}
\]
so the \ournotmain\ conjecture narrowly fails here to extend into the head.

\subsubsection{Decorated Rollet Graph, for $l=2$}   \label{ss:gdetl2}

We now continue the same process for $l=2$.
The (undecorated) Rollet graph in integer-partition (as opposed to Young diagram) notation is:
\[
\begin{tikzcd}[column sep=0.6em,row sep=0em]
&&&&(4)\ar[dash,r]&(4)\ar[dash,r]&(4)\ar[dash,r]&\dots\\
&&&(3)\ar[dash,ur]\ar[dash,dr]\\
&&(2)\arrow[dash,ur]\arrow[dash,dr]&& (3,1)\ar[dash,r]&(3,1)\ar[dash,r]&(3,1)\ar[dash,r]&\dots\\
\emptyset\arrow[dash,r] & (1)\arrow[dash,ur]\arrow[dash,dr] &&(2,1)\ar[dash,ur]\ar[dash,dr]\ar[dash,r]&(2,2)\ar[dash,r]&(2,2)\ar[dash,r]&(2,2)\ar[dash,r]&\dots\\
&&(1^2)\arrow[dash,ur]\arrow[dash,dr]&&(2,1^2)\ar[dash,r]&(2,1^2)\ar[dash,r]&(2,1^2)\ar[dash,r]&\dots\\
&&&(1^3)\ar[dash,dr]\ar[dash,ur]\\
&&&&(1^4)\ar[dash,r]&(1^4)\ar[dash,r]&(1^4)\ar[dash,r]&\dots
\end{tikzcd}
\]
In Figure \ref{fig:l2} we give the section $\Roll_2^{(2)}$ in all cases so far determined by direct computation. 
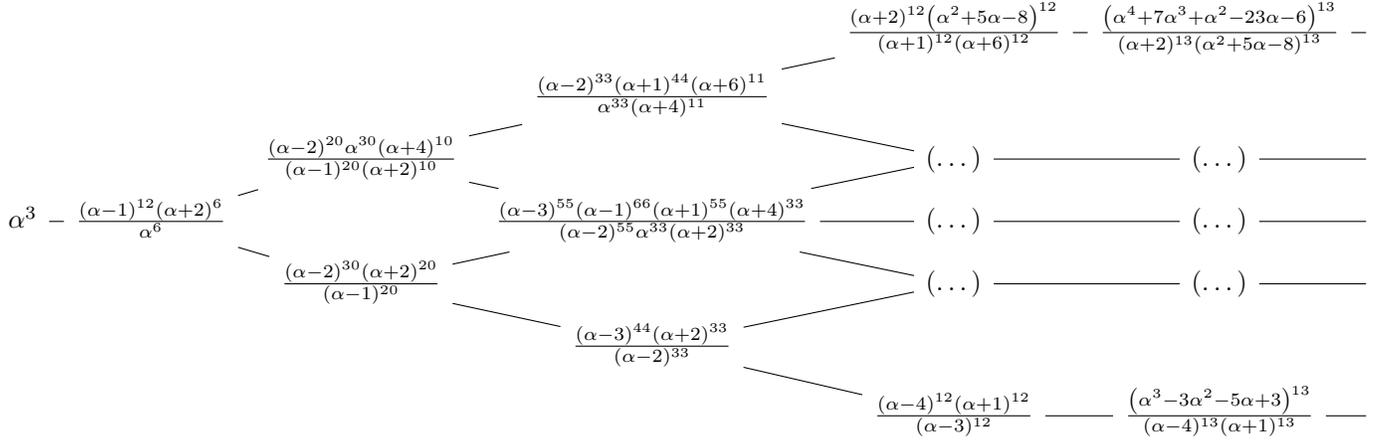
\begin{figure}
\[
\hspace{-3.51cm} \begin{tikzcd}[column sep=0.6em,row sep=0em]
&&&&\frac{(\alpha +2)^{12} \left(\alpha ^2+5 \alpha -8\right)^{12}}{(\alpha +1)^{12} (\alpha +6)^{12}}\ar[dash,r]&\frac{\left(\alpha ^4+7 \alpha ^3+\alpha ^2-23 \alpha -6\right)^{13}}{(\alpha +2)^{13} \left(\alpha ^2+5 \alpha -8\right)^{13}}\ar[dash,r]&{}\\
&&&\frac{(\alpha -2)^{33} (\alpha +1)^{44} (\alpha +6)^{11}}{\alpha ^{33} (\alpha +4)^{11}}\ar[dash,ur]\ar[dash,dr]\\
&&\frac{(\alpha -2)^{20} \alpha ^{30} (\alpha +4)^{10}}{(\alpha -1)^{20} (\alpha +2)^{10}}\arrow[dash,ur]\arrow[dash,dr]&& (\dots)\ar[dash,r]&(\dots)\ar[dash,r]&{}\\
\alpha ^3 \arrow[dash,r] & \frac{(\alpha -1)^{12} (\alpha +2)^6}{\alpha ^6} \arrow[dash,ur]\arrow[dash,dr] &&\frac{ (\alpha -3)^{55} (\alpha -1)^{66} (\alpha +1)^{55} (\alpha +4)^{33}}{ (\alpha -2)^{55} \alpha ^{33} (\alpha +2)^{33}}\ar[dash,ur]\ar[dash,dr]\ar[dash,r]&(\dots)\ar[dash,r]&(\dots)\ar[dash,r]&{}\\
&&\frac{(\alpha -2)^{30} (\alpha +2)^{20}}{(\alpha -1)^{20}}\arrow[dash,ur]\arrow[dash,dr]&&(\dots)\ar[dash,r]&(\dots)\ar[dash,r]&{}\\
&&&\frac{(\alpha -3)^{44} (\alpha +2)^{33}}{(\alpha -2)^{33}}\ar[dash,dr]\ar[dash,ur]\\
&&&&\frac{(\alpha -4)^{12} (\alpha +1)^{12}}{(\alpha -3)^{12}}\ar[dash,r]&\frac{\left(\alpha ^3-3 \alpha ^2-5 \alpha +3\right)^{13}}{(\alpha -4)^{13} (\alpha +1)^{13}}\ar[dash,r]&{}
\end{tikzcd}
\]
\caption{The section, $\Roll_2^{(2)}$, of the decorated Rollet graph, $\RollV_2$, in all cases so far computed. Any entry marked $(\dots)$ is so far not determined by direct computation.
\label{fig:l2}}
\end{figure}
 
Observe that
\begin{align*}
\CC^{7}_{(3,(3))} = 
\left( \frac{ P^{(31)}_{\ppmm{5}}}{P^{(31)}_{4}} \right)^{33} 
\left( \frac{ P^{(4)}_{\ppmm{5}}}{P^{(4)}_{4}} \right)^{11} 
&=  \left( \frac{(\alpha -2) (\alpha -1) (\alpha +1) (\alpha +4)}{4 \alpha  \left(\alpha ^2+\alpha -4\right)}  \right)^{33} 
\left( \frac{(\alpha +1) (\alpha +6)}{4 (\alpha +4)}\right)^{11} 
\\
&\propto \; \frac{\left(\alpha ^2+\alpha -4\right)^{33}}{(\alpha -1)^{33} (\alpha +4)^{33}} \; \V^7_{(3,(3))}
\end{align*}
\begin{align*}
\CC^{7}_{(3,(2,1))} = 
\left( \frac{ P^{(31)}_{\ppmm{5}}}{P^{(31)}_{3}} \right)^{33} 
\left( \frac{ P^{(4)}_{\ppmm{5}}}{P^{(4)}_{4}} \right)^{11} 
&= \left( \frac{(\alpha -2) (\alpha -1) (\alpha +1) (\alpha +4)}{4 \alpha  \left(\alpha ^2+\alpha -4\right)}  \right)^{33}\\
&\quad \times \left(\frac{(\alpha -3) (\alpha -1) (\alpha +2)}{4 (\alpha -2) (\alpha +1)}\right)^{33}\left(\frac{(\alpha -3) (\alpha +1)}{4 (\alpha -2)}\right)^{22}\\ 
\\
&\propto \; \frac{ (\alpha +1)^{33} \left(\alpha ^2+\alpha -4\right)^{33}}{(\alpha -2)^{33} (\alpha +2)^{66}} \; \V^7_{(3,(2,1))}
\end{align*}
\begin{align*}
\CC^{7}_{(3,(1^3))} = 
\left( \frac{ P^{(31)^T}_{\ppmm{5}}}{P^{(31)^T}_{4}} \right)^{33} 
\left( \frac{ P^{(4)^T}_{\ppmm{5}}}{P^{(4)^T}_{4}} \right)^{11} 
&=  \left( \frac{(\alpha -3) (\alpha -1) (\alpha +2)}{4 (\alpha -2) (\alpha +1)}  \right)^{33} 
\left( \frac{\alpha -3}{4}\right)^{11} 
\\
&\propto \; \frac{(\a+1)^{33}}{(\a-1)^{33}} \; \V^7_{(3,(1^3))}
\end{align*}
which again suggests the \ournotmain\ conjecture can be modified suitable to extend into the head.

\newpage

\newcommand{\rb}[1]{\raisebox{.685cm}{$\; #1 \;$}}

\section{Representation theory}  \label{ss:repthy}

By construction, if the gram matrix of a standard module is singular - i.e. of submaximal rank - then the module map to the corresponding co-standard is of submaximal rank and the standard module has a submodule, and hence there is a morphism in from some other standard module.
By quasiheredity (or globalisation) \cite{KadarYu} the possible maps in are restricted.
In particular in the one-cup cases the possibilities are very restricted, and we 
can identify the module mapping in. 
More generally this identification requires some detective work. 
There is the general {\em tower of recollement} scheme, or we can investigate examples piecemeal. Both approaches are illuminating, and we will work through some cases here. 

\subsection{Generalities}

For each $l$, we can in principle apply the tower of recollement method. 
For $l\geq 0$ 
the tower of recollement method requires an inital `bootstrap' calculation, 
here of morphisms into `1-cup' modules - necessarily from no-cup modules, by globalisation. 


\subsubsection{Bootstrap for $l=0$}   \label{ss:bootcase0}

Set $l=0$. 
For definiteness consider $\Delta^{n}_{(n-2,(2))}$ for now, so there is a $(2)$ projector attached to the first two propagating legs of each basis diagram.

\mdef   \label{exa:l0p6la2-}
Example. 
Explicit  basis for  
$\Deltaa^{l=0,n=6}_{(4,(2))}$ is:  
\medskip 

\includegraphics[width=10.9cm]{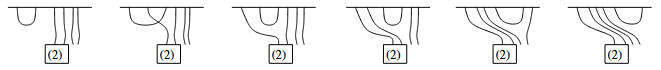}

\mdef  
(We will omit the projector from diagrams for now. And drop the example to $n=5$ for brevity.) 
A 1-d submodule must be spanned by an element of form:
\[ 
\rb{\;a\;} \begin{tikzpicture}[scale=0.68]
     \draw (0,0) -- (2.43,0) ;
     \draw (1.5,0) -- (1.5,-1);
     \draw (1.85,0) -- (1.85,-1);
     \draw (2.2385,0) -- (2.2385,-1);
     \draw (.5,0) .. controls (.5,-.5) and (1.0,-.5) .. (1,0) ;
 \end{tikzpicture}
 \rb{+\;b\;} \begin{tikzpicture}[scale=0.68]
     \draw (0,0) -- (2.43,0) ;
     \draw (1,0) -- (1,-1);
     \draw (1.85,0) -- (1.85,-1);
     \draw (2.2385,0) -- (2.2385,-1);
     \draw (.5,0) .. controls (.5,-.5) and (1.5,-.5) .. (1.5,0) ;
 \end{tikzpicture}
\rb{+\;c\;} \begin{tikzpicture}[scale=0.68]
     \draw (0,0) -- (2.43,0) ;
     \draw (.5,0) -- (.5,-1);
     \draw (1.85,0) -- (1.85,-1);
     \draw (2.2385,0) -- (2.2385,-1);
     \draw (1.,0) .. controls (1.,-.5) and (1.39,-.5) .. (1.39,0) ;
      \end{tikzpicture}
 \rb{+\;d_1\;} \begin{tikzpicture}[scale=0.68]
     \draw (0,0) -- (2.43,0) ;
     \draw (.5,0) -- (.5,-1);
     \draw (.85,0) -- (.85,-1);
     \draw (2.2385,0) -- (2.2385,-1);
     \draw (1.5,0) .. controls (1.5,-.5) and (1.89,-.5) .. (1.89,0) ;
 \end{tikzpicture}
  \rb{+\; d_2 \;...\;}
\]
(this being a general element of the space with basis as in (\ref{exa:l0p6la2-})).
We are looking for 1d submodules where all algebra defining-basis elements with any cup act as 0. 
Thus for example 
(action of elements with cup and cap in same position, say, taken in the natural order 
- noting that the position of the cup plays no role in determining the coefficients):
\[
\begin{bmatrix}
   x&1& 1 \\ 1&x&1 &1 \\ 1&1&x&1 \\ &1 &1&x&1 \\ &&& 1&x&1 \\ &&&& 1&x&1 \\  &&&&& 1&x&1 \\ &&&&&& 1&x&1 \\ &&&&&&&1&x
\end{bmatrix}
\begin{bmatrix}
    a \\ b \\ c \\ d_1 \\ d_2 \\ \vdots \\ \vdots \\ \vdots
\end{bmatrix}
=0
\]
(other entries 0; we write $x$ for the $\xx$ parameter for now). 
Observe that with this choice the matrix coincides with the gram matrix \cite{ShBr}. 
Evidently we need submaximal rank here, so this coincides, so far, 
with the Chebychev problem we already have (see \S\ref{ss:cheby0}). 
But now let us also require that the first elementary transposition 
$\sigma_1$ acts like +1 (we  consider -1 later, in \ref{case-1}). 

\mdef  
Aside: Observe that all these matrices have manifestly submaximal rank when $x=1$. 
And indeed rank 6 case has another (fairly manifest) linear dependence when $x=1$ (consider row-2 + row-6 vs row-1 + row-5). Is there a pattern? 
(An affine alcove-like effect would manifest as some kind of re-occurrence of special values, such as $x=1$ when $n$ cong. 0 mod. 3.) 
Yes.
... Rank 9 has row-1 + row-5 + row-8 = row-2 + row-6 + row-9. 
And the pattern will now be clear.

\mdef 
Requiring $\sigma_1$ to act like +1 on our subspace 
(requiring a $\Deltaaa{n}{(n,(2))}$ submodule) 
sets $b=c$. 
If $a=0$ then $b+c=0$ so $b=c=0$ and there is no nonzero solution, so here we can take $a=1$. Thus $x+2b=0$ and so $b=c=-x/2$. Thus 
$d_1=-(1-(1+x)x/2) = \frac{x^2 +x-2}{2} =\frac{(x-1)(x+2)}{2}$; then 
$d_2=-(xd_1 +2b) = -x(d_1 -1) =\frac{-x(x^2 +x-4)}{2}$; then 
$d_3=-(xd_2+d_1) = \frac{x^2(x^2+x-4) - (x^2+x-2)}{2} = \frac{x^4 +x^3-5x^2-x+2}{2}$ 
(confer \ref{ss:cheby0})
and so on until $d_{j} = -(xd_{j-1} +d_{j-2})$; 
and finally $xd_{j}=-d_{j-1}$, which is 
$d_{j+1} = -(xd_{j} +d_{j-1})$ provided that $d_{j+1} =0$. 
This last condition gives the $x$ values for which there is a map at each $n$.

\mdef  \label{case-1}
In case $\sigma_1$ acts like -1 
 then also $a=0$ and $b=-c\neq 0$ and $d_i =0$. Thus $xb-b = b-xb=0$ and $x=1$. 
Confer Fig.\ref{fig:Bratelli-l0-gramdets}. 

\mdef For completeness we can consider maps to $\Delta^{n}_{(n-2,(1^2))}$. 
The matrix part is 
\beq \label{eq:l0ev12}
\begin{bmatrix}
   x&1& 1 \\ 1&x&1 &-1 \\ 1&1&x&1 \\ &-1 &1&x&1 \\ &&& 1&x&1 \\ &&&& 1&x&1 \\  &&&&& 1&x&1 \\ &&&&&& 1&x&1 \\ &&&&&&&1&x
\end{bmatrix}
\begin{bmatrix}
    a \\ b \\ c \\ d_1 \\ d_2 \\ \vdots \\ \vdots \\ \vdots
\end{bmatrix}
=0
\eq 
For maps from $(n,(2))$ we require $\sigma_1$ to act as +1. 
Here $b=c$ and $d_i = 0$. 
If $a=0$ then $b+c=0$ from (\ref{eq:l0ev12}) and hence $b=c=0$, so we can assume 
$a=1$ so $x+2b=0$ and $b=-x/2$. 
The second row of  (\ref{eq:l0ev12}) gives 
$1-(x+1)x/2 =d_1 = 0 = -\frac{x^2 +x-2}{2} =-\frac{(x-1)(x+2)}{2}$. 
The rest are automatic. 

Requiring $\sigma_1$ to act as -1 we get $a=0$ and $b=-c$. 
We can put $b=1$.
The second row gives $x-1 = d_1$. The third is automatic. 
The fourth gives $-1-1+x(x-1) =-d_2 = x^2 -x -2 = (x-2)(x+1)$. 
The $d_i$s then satisfy the Chebyshev recurrence until the last step, and we get a solution provided that
$x$ is such that the last $d_i$ is 0.

\subsubsection{Bootstrap for \ignore{$l=1$ and}general $l$}   \label{ss:keybootl1}
Fix $l$, $\l\vdash l+2$ and $n\geq l+4$. In this subsection we prove that when $\a_0$ is a root of the polynomial $P_{n}^{(\l)}$, the $J_{l,n}(\a_0)$ standard module $\Delta^{n}_{(n-2,\l)}$ has a submodule isomorphic to $\Delta^{n-2}_{(n-2,\l)}$. We proceed as follows.

\ignore{Here the simplest indicative example is for $\Deltaaa{5}{3,(21)}$ given by \ref{exa:keyn5la21}. 
Recall a basis is:
\[
\hspace{-1cm} 
\onecuppppp{6}{1}{2}{3}{4}{5}  \; 
\onecuppppp{6}{1}{3}{2}{4}{5}  \;
\onecuppppp{6}{1}{4}{2}{3}{5}  \;
\onecuppppp{6}{2}{4}{1}{3}{5}  \;
\onecuppppp{6}{2}{3}{1}{4}{5}{}{}  \;
\onecuppppp{6}{3}{4}{1}{2}{5}{}{}  \;
\onecuppppp{6}{4}{5}{1}{2}{3}{}{}  
\]
%
\[
\hspace{1cm} 
\onecupppppx{6}{1}{2}{3}{4}{5}  \;
\onecupppppx{6}{1}{3}{2}{4}{5}    \;
\onecupppppx{6}{1}{4}{2}{3}{5}  \;
\onecupppppx{6}{2}{4}{1}{3}{5}  \;
\onecupppppx{6}{2}{3}{1}{4}{5}{}{}   \;  
\onecupppppx{6}{3}{4}{1}{2}{5}{}{}  \;
\onecupppppx{6}{4}{5}{1}{2}{3}{}{}  
\]
and so the action of one-cup (one-cap) elements on basis elements is given by, for example: 
\medskip 


\hspace{-2.8231cm} 
\flip{\onecupppppo{6}{1}{2}{3}{4}{5}} \hspace{.0031cm}
\flip{\onecupppppo{6}{1}{2}{3}{4}{5}} \hspace{.0041cm}
\flip{\onecupppppo{6}{1}{2}{3}{4}{5}} \hspace{-.004cm}
\flip{\onecupppppo{6}{1}{2}{3}{4}{5}} \hspace{.001cm} 
\flip{\onecupppppo{6}{1}{2}{3}{4}{5}} \hspace{.001cm}
\flip{\onecupppppo{6}{1}{2}{3}{4}{5}} \hspace{.001cm}
\flip{\onecupppppo{6}{1}{2}{3}{4}{5}}

\hspace{-2.791cm} 
\onecuppppp{6}{1}{2}{3}{4}{5}  \; 
\onecuppppp{6}{1}{3}{2}{4}{5}  \;
\onecuppppp{6}{1}{4}{2}{3}{5}  \;
\onecuppppp{6}{2}{4}{1}{3}{5}  \;
\onecuppppp{6}{2}{3}{1}{4}{5}{}{}  \;
\onecuppppp{6}{3}{4}{1}{2}{5}{}{}  \;
\onecuppppp{6}{4}{5}{1}{2}{3}{}{}  

\vspace{.321cm}

\hspace{-2.8231cm} 
\flip{\onecupppppo{6}{1}{2}{3}{4}{5}} \hspace{.0031cm}
\flip{\onecupppppo{6}{1}{2}{3}{4}{5}} \hspace{.0041cm}
\flip{\onecupppppo{6}{1}{2}{3}{4}{5}} \hspace{-.004cm}
\flip{\onecupppppo{6}{1}{2}{3}{4}{5}} \hspace{.001cm} 
\flip{\onecupppppo{6}{1}{2}{3}{4}{5}} \hspace{.001cm}
\flip{\onecupppppo{6}{1}{2}{3}{4}{5}} \hspace{.001cm}
\flip{\onecupppppo{6}{1}{2}{3}{4}{5}}

\hspace{-2.81cm} 
\onecupppppx{6}{1}{2}{3}{4}{5}  \;
\onecupppppx{6}{1}{3}{2}{4}{5}    \;
\onecupppppx{6}{1}{4}{2}{3}{5}  \;
\onecupppppx{6}{2}{4}{1}{3}{5}  \;
\onecupppppx{6}{2}{3}{1}{4}{5}{}{}   \;  
\onecupppppx{6}{3}{4}{1}{2}{5}{}{}  \;
\onecupppppx{6}{4}{5}{1}{2}{3}{}{}  


\medskip 

\noindent 
Observe here that one-cup elements can be written in the form $\uu_{kl} \uu_{ij}^*$; 
and the $\uu_{kl}$ plays essentially no role in the linear algebra of their action, so 
we omit it and our example above has $(i,j)=(1,2)$. 
\\ 
Observe also that 
the remainder of the expression is then encoded as an element of the Specht module 
- provided that we take account of the quotient that is in play. 
Specifically here then, writing $\{ b_1, b_2 ,..., b_{d_\l} \}$ as usual for our Specht module basis (ported into the standard module as per the construction):  
the first term is $\xx b_1$; 
then next is simply $b_1$; 
the next two are both $\sigma_1 b_1 = b_1$ in our chosen basis; 
the next is $b_1$; and 
the next two are zero. 
On the next row: the first is $\xx b_2$; 
the next is $b_2$; 
the next two are $\sigma_1 b_2$ (which can be rewritten as a linear combination); 
the next is $b_2$; and the final two are zero. 
That is, for a general element of the module, with coefficients $\kappa_{hij}$ say for basis elements $\uu_{ij} b_h$, we have  
\[
\uu_{kl} \uu_{12}^* \sum_{hij} \kappa_{hij} \uu_{ij} b_h \; 
        = \; \uu_{kl} ((\xx\kappa_{112} +\kappa_{113}+\kappa_{114} \sigma_1+\kappa_{124}\sigma_1+\kappa_{123}+0+ 0) b_1  + (\xx\kappa_{212} +\kappa_{213}+...+ ...) b_2 ) 
\]
Observe that we can write this as 
\[ ...\]
...

\vspace{.53cm} 


\hspace{-2.8231cm} 
\flip{\onecupppppo{6}{4}{5}{1}{2}{3}} \hspace{.0031cm}
\flip{\onecupppppo{6}{4}{5}{1}{2}{3}} \hspace{.0041cm}
\flip{\onecupppppo{6}{4}{5}{1}{2}{3}} \hspace{-.004cm}
\flip{\onecupppppo{6}{4}{5}{1}{2}{3}} \hspace{.001cm} 
\flip{\onecupppppo{6}{4}{5}{1}{2}{3}} \hspace{.001cm}
\flip{\onecupppppo{6}{4}{5}{1}{2}{3}} \hspace{.001cm}
\flip{\onecupppppo{6}{4}{5}{1}{2}{3}}

\hspace{-2.791cm} 
\onecuppppp{6}{1}{2}{3}{4}{5}  \; 
\onecuppppp{6}{1}{3}{2}{4}{5}  \;
\onecuppppp{6}{1}{4}{2}{3}{5}  \;
\onecuppppp{6}{2}{4}{1}{3}{5}  \;
\onecuppppp{6}{2}{3}{1}{4}{5}{}{}  \;
\onecuppppp{6}{3}{4}{1}{2}{5}{}{}  \;
\onecuppppp{6}{4}{5}{1}{2}{3}{}{}  

\vspace{.321cm}

\hspace{-2.8231cm} 
\flip{\onecupppppo{6}{1}{2}{3}{4}{5}} \hspace{.0031cm}
\flip{\onecupppppo{6}{1}{2}{3}{4}{5}} \hspace{.0041cm}
\flip{\onecupppppo{6}{1}{2}{3}{4}{5}} \hspace{-.004cm}
\flip{\onecupppppo{6}{1}{2}{3}{4}{5}} \hspace{.001cm} 
\flip{\onecupppppo{6}{1}{2}{3}{4}{5}} \hspace{.001cm}
\flip{\onecupppppo{6}{1}{2}{3}{4}{5}} \hspace{.001cm}
\flip{\onecupppppo{6}{1}{2}{3}{4}{5}}

\hspace{-2.81cm} 
\onecupppppx{6}{1}{2}{3}{4}{5}  \;
\onecupppppx{6}{1}{3}{2}{4}{5}    \;
\onecupppppx{6}{1}{4}{2}{3}{5}  \;
\onecupppppx{6}{2}{4}{1}{3}{5}  \;
\onecupppppx{6}{2}{3}{1}{4}{5}{}{}   \;  
\onecupppppx{6}{3}{4}{1}{2}{5}{}{}  \;
\onecupppppx{6}{4}{5}{1}{2}{3}{}{}  

\medskip

(compare also with \ref{lem:xi}). 
Thus the first row of... a condition is:
\[
(\xx, 1,1,1,1,0,0, \;\; 0,0,-,-,0,0,0) \; (a_{121},a_{131},a_{141},a_{241},a_{231},...)^t  \;= 0
\]


\bigskip 


\bajm{[I think I have solved the general bootstrap!!!!!!!!! See the following.... Loading ....! ]}\bajm{[I am not claiming the arguments below are complete its just a scaffold for me but I'm confident it works]} \ppm{[ :-) ] [certainly worth polishing!]}\bajm{[I realise general is not quite true, but I have shown at least that the roots of $P_n^{(\l)}(\a)$ are the conditions for a $\Delta^{(n)}_{(n,\l)}$ submodule.]}}

\begin{lemma}\label{le:xi2}
Here we take the ground ring to be $\C[\xx]$. 
    Fix $l$ and let $\l\vdash l+2$. For any $n\geq l+4$, there exists a unique (up to an overall scalar $\in \C[\a]$) element $\xi^{(n)}_\l\in \Delta^{(n)}_{(n-2,\l)}$ with coefficients in $\C[\a]$ 
    such that the following hold:
    \begin{enumerate}
        \item $\uu^*_{i,j} \xi^{(n)}_{\l}=0$ for all $i<j<n$ with $\uu_{i,j}\in \BBB_{(l,n)}$.
        \item $\uu^*_{n-1,n} \xi^{(n)}_{\l}=D_n(\a) c_\l $ for some $D_n(\a)\in \C[\a]$.  
    \end{enumerate}
\end{lemma}
\begin{proof}
    Existence and uniqueness follow by solving a (full rank) linear problem which is analogous to the proof of Lemma \ref{lem:xi}. 
\end{proof}
\begin{lemma}\label{lem:cxi} Fix $l$ and let $\l\vdash l+2$. For any $n\geq l+4$, let $\xi^{(n)}_\l\in \Delta^{(n)}_{(n-2,\l)}$ be as per Lemma \ref{le:xi2} (with a fixed normalisation, \textit{i.e.} choice of $D_n(\a)$). Then, with $\a$ generic (i.e. working over $\C[\a]$) we have $c_\l  \xi^{(n)}_\l=\xi^{(n)}_\l$, with $c_\l\in J_{l,n}$ as per \ref{de:cl}.   
\end{lemma}
\begin{proof} 
It will be clear from the proof of \ref{lem:niceelt}, and idempotency of $c_\l$, that $c_\l  \xi^{(n)}_\l$ also satisfies conditions 1. and 2. from the previous lemma (with the same factor $D_n(\a)$) thus the desired equality follows by uniqueness of $\xi_\l$.
\end{proof}

\mdef\label{pa:spechbas} For the remainder of this section, fix a basis $\BB=\{b_1,\dots, b_{d_\l}\}$ for $\SS_\l$, which is orthonormal wrt $\langle \_,\_ \rangle$ and for which $b_1=c_\l$, and $b_k=x_k c_\l$ for some $x_k\in \C \Sigma_{l+2}$. It follows from Lemma \ref{lem:cxi} that the coefficient of $u_{n-1,n}b_k$ in $\xi^{(n)}_{\l}$ is only non-zero when $k=1$, since 
\[ c_\l u_{n-1,n}b_k  =u_{n-1,n}c_\l x_k c_\l =\langle b_1, b_k \rangle u_{n-1,n}  c_\l,\]
and since $c_\l u_{ij}$ for $j<n$ can be written as a linear combination $\sum_{i'<j'<n} u_{i'j'} \sigma_{i'j'}$ for $\sigma_{i'j'}\in \C\Sigma_{l+2}$, by ``pulling" cups up through each permutation in $c_\l$.

\mdef \label{de:xinorm} Thus far, the element $\xi^{(n)}_\l\in \Delta^{(n)}_{(n-2,\l)}$ is only defined up to a polynomial factor. Let us fix a choice of ``normalisation" for $\xi^{(n)}_\l\in \Delta^{(n)}_{(n-2,\l)}$. Firstly, we will assume that there is no common polynomial factor of all coefficients in $\xi^{(n)}_\l$. Thus for all specialisations $\a\in \C$, we have that $\xi^{(n)}_\l\neq 0$. This determines $\xi^{(n)}_\l$ (and hence $D_n(\a)$) up to a complex scalar. By multiplying $\xi^{(n)}_\l$ by a complex scalar, we may assume that $D_n(\a)$ is monic. In what follows we will always take $\xi^{(n)}_\l$ to be normalised as such.

\begin{lemma} Fix $l$ and let $\l\vdash l+2$. For any $n > l+4$, the (normalised) $\xi^{(n)}_\l$ can be obtained inductively by
\begin{equation} \label{eq:ind2} \xi^{(n)}_\l = D_{n-1}(\a) \uu_{n-1,n}c_{\l}-\xi_{\l}^{(n-1)}\otimes id_1. \end{equation}
\end{lemma}

\begin{proof} Fix a basis $\BB=\{b_1,\dots, b_{d_\l}\}$ for $\SS_\l$ as per \ref{pa:spechbas}. Let $\xi_{\l}^{(n-1)}$ be as per \ref{de:xinorm} with $\uu_{n-2,n-1}^*\xi_{\l}^{(n-1)}=D_{n-1}(\a)$. Then define
\[\xi:= D_{n-1}(\a) \uu_{n-1,n}c_{\l}-\xi_{\l}^{(n-1)}\otimes id_1.\]
Since $\xi_{\l}^{(n-1)}\otimes id_1$ has no components of the form $\uu_{n-1,n} b_k$ it follows that the coefficients in $\xi$ have no common non-trivial factor, since those in $\xi_{\l}^{(n-1)}\otimes id_1$ didn't.

Now observe for any $i<j<n-1$, for $u_{ij}\in \BBB_{(l,n)}$ we have
\[ \uu_{ij}^*\xi=D_{n-1}(\a) \uu_{ij}^*\uu_{n-1,n}c_{\l}-(\uu_{ij}^*\xi_{\l}^{(n-1)})\otimes id_1=0.\]
using the property 1 from from Lemma \ref{le:xi2}. Using property 2. we have 
\begin{align*}
    \uu_{n-2,n-1}^*\xi&=D_{n-1}(\a) \uu_{n-2,n-1}^*\uu_{n-1,n}c_{\l}-(\uu_{n-2,n-1}^*\xi_{\l}^{(n-1)})\otimes id_1\\
    &=D_{n-1}(\a)c_\l-D_{n-1}(\a)c_\l=0.
\end{align*}
Now note that 
\[\uu_{n-1,n}^* (\xi_{\l}^{(n-1)}\otimes id_1)=[\xi_{\l}^{(n-1)}]_{n-2,n-1,1 }\uu_{n-1,n}^*\uu_{n-2,n-1} c_\l=[\xi_{\l}^{(n-1)}]_{n-2,n-1,1 } c_\l,\]
since $\uu_{n-1,n}^*\uu_{i,j}$ has $n-4$ propagating lines for $j<n-1$, and the only non-zero components of the form $\uu_{n-2,n-1}b_{k}$ in $\xi_{\l}^{(n-1)}$ have $k=1$. Therefore, we compute that 
\[\uu_{n-2,n-1}^*\xi =(\a D_{n-1}(\a)-[\xi_{\l}^{(n-1)}]_{n-2,n-1,1 })c_\l,\]
so $\xi$ satisfies 1. and 2. from Lemma \ref{le:xi2}, with $D_n(\a)=(\a D_{n-1}(\a)-[\xi_{\l}^{(n-1)}]_{n-2,n-1,1 })$. Therefore, by uniqueness $\xi= a \xi^{(n)}_\l$ for some factor $a\in \C$. We will now verify that $a=1$, that is $D_n(\a)$ is monic.

Firstly, we treat the case $n=l+5$. For $\xi^{(l+5)}_\l$ as defined by the recursion \eqref{eq:ind2} we have that $D_{l+5}(\a)=\a D_{l+4}(\a) +[\xi_\l]_{l+3,l+4,1}$, where $\xi_\l^{(l+4)}=\xi_\l$ is as per Lemma \ref{lem:xi}. In the proof of Lemma \ref{lem:pl}, we saw that 
\[
\frac{D_{l+5}(\a)}{D_{l+4}(\a)}=\frac{\a D_{l+4}(\a)+[\xi_\l]_{l+3,l+4,1}}{D_{l+4}(\a)}=\frac{P_{l+5}^{(\l)}(\a)}{P_{l+4}^{(\l)}(\a)}.
\]
Since the $P_{n}^{(\l)}(\a)$ are monic for $n\geq l+4$, it follows that if $D_{l+4}(\a)$ is taken to be monic, then $D_{l+5}(\a)$ will be too. Thus $\xi^{(l+5)}_\l$ as defined by \eqref{eq:ind2} is normalised as per \ref{de:xinorm}. 

Now for $n>l+5$, by applying \eqref{eq:ind2} twice we find
\[D_{n}(\a)=\a D_{n-1}(\a)+[\xi_{\l}^{(n-1)}]_{n-2,n-1,1 }=\a D_{n-1}(\a)+D_{n-2}(\a),\]
which implies $D_{n}$ is monic, since $D_{l+4}$ and $D_{l+5}$ were, and satisfy $\deg(D_{l+5})=\deg(D_{l+4})+1$, in addition.\end{proof}

\begin{corollary}\label{co:PdivD} Fix $l$ and let $\l\vdash l+2$. For any $n > l+4$, the (normalised) element $\xi^{(n)}_\l\in \Delta^{(n)}_{(n-2,\l)}$ satisfies $u_{n-1,n}^* \xi^{(n)}_\l=D_n(\a) c_\l$ where $P_n^{(\l)}(\a)$ divides $D_n(\a)$.
\end{corollary} 

\begin{proof}
Since for $n>l+5$, the $D_n(\a)$ obey the Chebyshev recursion \eqref{eq:Cheby0}, the claim now follows from the previous proof, where we had
\[\frac{D_{l+5}(\a)}{D_{l+4}(\a)}=\frac{P_{l+5}^{(\l)}(\a)}{P_{l+4}^{(\l)}(\a)},\]
which implies $D_{l+4}(\a)=f(\a) P^{(\l)}_{l+4}(\a)$ and $D_{l+5}(\a)=f(\a) P^{(\l)}_{l+5}(\a)$ for some $f(\a)\in \C[\a]$, since the series $P^{(\l)}_{n}$ has no common non-trivial factor.
\end{proof}

\begin{theorem} Fix $l$ and let $\l\vdash l+2$. Let $n\geq l+4$, and suppose $\a_0\in \C$ is such that $P_n^{(\l)}(\a_0)=0$. 
Then 
\ppmm{when the algebra parameter $\xx=\xx_0$ we have that} 
$J_{l,n} \xi^{(n)}_\l$ is a submodule of $\Delta^{(n)}_{(n-2,\l)}$ isomorphic to $\Delta^{(n)}_{(n,\l)}$. 
\end{theorem} 
\begin{proof} 
Firstly, we note that $\xi^{(n)}_\l$ is non-zero after specialisation at $\a=\a_0$ and satisfies $c_\l \xi^{(n)}_\l=\xi^{(n)}_\l$, and $u_{n-1,n}^* \xi^{(n)}_\l=D(\a_0) c_\l=0$ by Corollary \ref{co:PdivD}. In particular, combining with condition 1. from Lemma \ref{le:xi2} we have $J_{l,n}^{n-2} \xi^{(n)}_{\l}=0$ (where $J_{l,n}^{p}$ is as per \ref{de:Jln}). Since $J_{l,n}=\C \Sigma_{l+2}\oplus J_{l,n}^{n-2}$ we compute (as $\C$ vector spaces):
\begin{align*} J_{l,n}\xi^{(n)}_\l&=\C \Sigma_{l+2} \xi^{(n)}_\l= \C \Sigma_{l+2} c_\l \xi^{(n)}_\l,   
\end{align*}
    It follows that the $b_i \xi^{(n)}_\l= x_i \xi^{(n)}_\l $ span $J_{l,n}\xi^{(n)}_\l$, where $\BB=\{b_1,\dots, b_{d_{\l}}\}$ is a basis for $\C \Sigma_{l+2} c_\l \simeq \SS_\l$ as per \ref{pa:spechbas}. To see they are linearly independent we argue as follows: suppose that $\sum_k a_k b_k \xi^{(n)}_\l=0$ for some coefficients $a_k\in \C$. Multiplying by $c _\l x_{i}^*$, we obtain then
    \[
    0=\sum_k a_k c _\l x_{i}^* x_k c_\l \xi^{(n)}_\l=\sum_k a_k \langle b_k, b_i\rangle c_\l \xi^{(n)}_\l= a_i c_\l \xi^{(n)}_\l =a_i \, \xi^{(n)}_\l,
    \]
    which implies $a_i=0$ for all $i$. It therefore follows that $x c_\l \mapsto x c_\l \xi_\l^{(n)} $ for $x\in J_{l,n}$ is a \bajmm{monomorphism} $\Delta^{(n)}_{(n,\l)}\to \Delta^{(n)}_{(n-2,\l)} $ 
    of $J_{l,n}$ modules. 
    \end{proof}

\subsection{Case $l=-1$}

As ever, the first case to consider is $l=-1$.  See Fig.\ref{fig:gramdets03}. 

\medskip

Rollet diagram with all $\Delta$ maps, for a particular choice of $\xx$ 
(here $\xx$ a root of $P^{(1)}_4$):   

\[
\hspace{-2.1cm} 
\xymatrix{
   \;\bullet_0 \ar@{-}[r] &  
   \;\bullet_1 \ar@{-}[r] &  
   \;\bullet_2 \ar@{-}[r] &  
   \;\bullet_3 \ar@{-}[r] &  
   \;\bullet_4 \ar@{-}[r] &  
   \;\bullet_5 \ar@{-}[r] \ar@[red]@/_1pc/[ll] &  
   \;\bullet_6 \ar@{-}[r] \ar@[red]@/_1.3pc/[llll] &  
   \;\bullet_7 \ar@{-}[r] \ar@[red]@/_1.653pc/[llllll] &  
   \;\bullet_8 \ar@{-}[r] \ar@[red]@/_1.9653pc/[llllllll] &  
   \;\bullet_9 \ar@{-}[r] &  
   \;\bullet_{} \ar@{-}[r] \ar@[red]@/_1pc/[ll] &  
   }
\]
Note in particular the alcove geometry, with blocks being orbits of the affine reflection group action; and the separation of reflection hyperplanes being determined by $\xx$.
It will be very interesting to look for analogues of this geometry for higher $l$! 

\begin{figure} 
\input{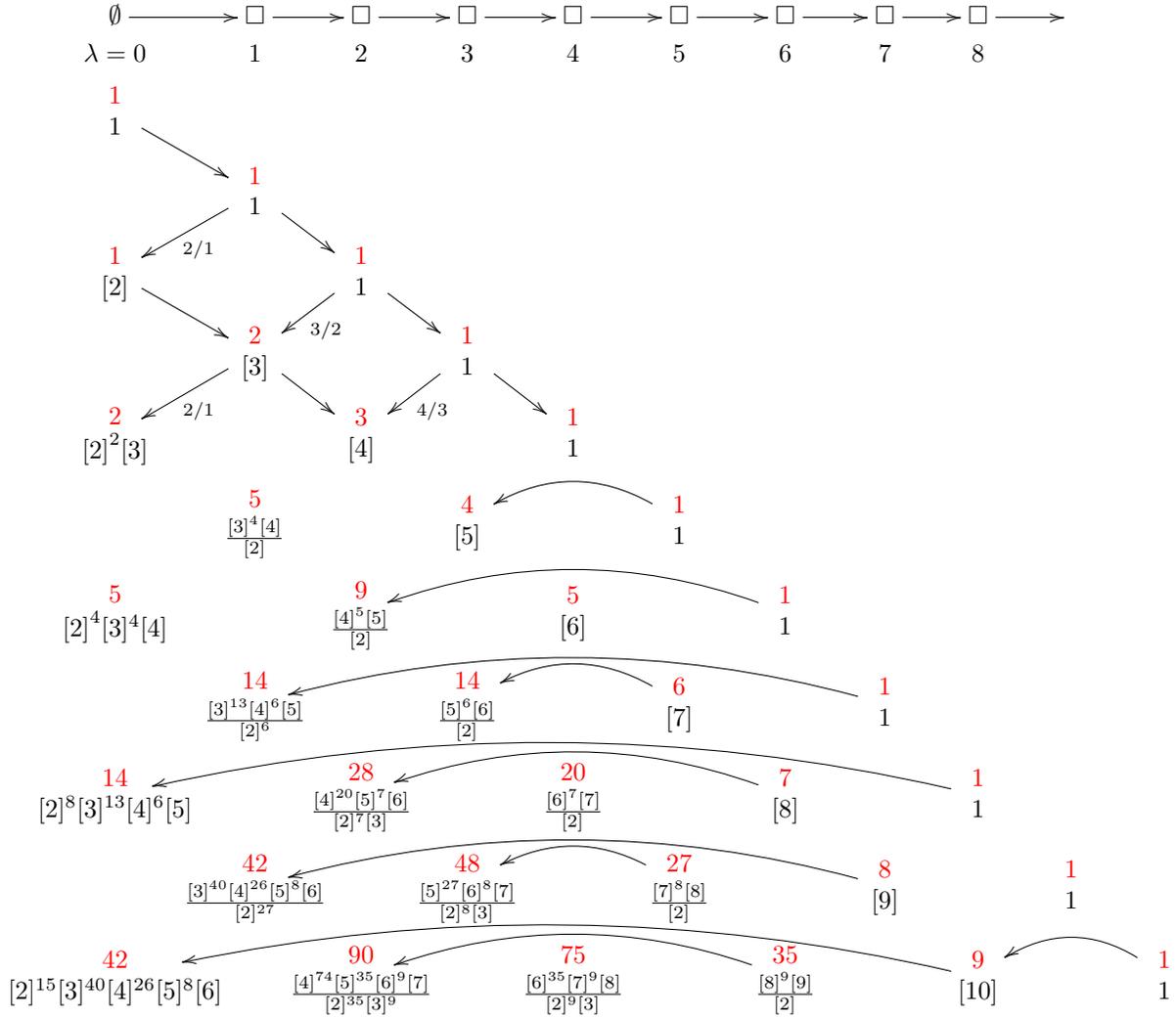}
\caption{The $l=-1$ Rollet graph augmented with many standard gram determinants (parametrised as $\a=[2]$ using the quantum number notation $[n]:=(q^n-q^{-n})/(q-q^{-1})$) --- 
and curved arrows indicating standard module morphisms in case $[5]=0$. 
\label{fig:gramdets03}}
\end{figure}

\newpage

\subsection{Representation theory - case $l=0$}

Here we investigate the maps between standard modules at special values of $\alpha$. 
Recall that for $l=-1$ we presented corresponding data in Fig.\ref{fig:gramdets03}.
Perhaps the first thing to do here is to find an embedding of the Rollet graph at $l=0$ in the plane (of the page) so that data can be presented in as uncluttered a way as possible. One possibility is:
\newcommand{\ver}[1]{{\!\!\!\!\!\!\begin{array}{c} { } \\ \bullet \\ \!\!\!\! \scriptsize{#1} \!\!\!\! \end{array}\!\!\!\!\!\!}}
\[
\hspace{-2.1cm} 
\xymatrix{
   &&&&&& \;\bullet_0 \ar@{-}[d] \\
\ver{(7,\!(2))} \ar@{-}[r] & 
\ver{(6,\!(2))} \ar@{-}[r] & 
\ver{(5,\!(2))} \ar@{-}[r] & 
\ver{(4,\!(2))} \ar@{-}[r] & 
\ver{(3,\!(2))} \ar@{-}[r] & 
\ver{(2,\!(2))} \ar@{-}[r] & 
\bullet_1 \ar@{-}[r] & 
\ver{(2,\!(1^2))} \ar@{-}[r] & 
\ver{(3,\!(1^2))} \ar@{-}[r] & 
\ver{(4,\!(1^2))} \ar@{-}[r] & 
\bullet \ar@{-}[r] & 
\bullet \ar@{-}[r] & 
\ver{(7,\!(1^2))} & 
}
\]

\subsubsection{Subcase $\alpha =1$}  \label{ss:alph1}

Observe from Fig.\ref{fig:Bratelli-l0-gramdets} (say) that 
the first singular gram matrix here is for  $\Delta^{3}_{(3,(1))}$. 
The only possible maps in are from the one-dimensional modules 
$\Delta^{3}_{(3,(3))}$ and/or $\Delta^{3}_{(3,(1^3))}$. 
Thus we are looking for subspaces of the form:

\newcommand{\IU}{
\raisebox{-.32cm}{
\begin{tikzpicture}
     \draw (0,0) -- (2,0) ;
     \draw (.5,0) -- (.5,-1);
     \draw (1,0) .. controls (1,-.5) and (1.5,-.5) .. (1.5,0) ;
     \draw (1,-1) .. controls (1,-.5) and (1.5,-.5) .. (1.5,-1) ;
 \draw (0,-1) -- (2,-1);
 \end{tikzpicture}
}}
\newcommand{\UI}{
\raisebox{-.32cm}{
\begin{tikzpicture}
     \draw (0,0) -- (2,0) ;
     \draw (1.5,0) -- (1.5,-1);
     \draw (.5,0) .. controls (.5,-.5) and (1,-.5) .. (1,0) ;
     \draw (.5,-1) .. controls (.5,-.5) and (1,-.5) .. (1,-1) ;
 \draw (0,-1) -- (2,-1);
 \end{tikzpicture}
}}
\newcommand{\SigI}{
\raisebox{-.32cm}{
\begin{tikzpicture}
     \draw (0,0) -- (2,0) ;
     \draw (1.5,0) -- (1.5,-1);
     \draw (.5,0) -- (1,-1) ;
     \draw (.5,-1) -- (1,0) ;
 \draw (0,-1) -- (2,-1);
 \end{tikzpicture}
}}

\[ 
\rb{c\;} \begin{tikzpicture}
     \draw (0,0) -- (2,0) ;
     \draw (.5,0) -- (.5,-1);
     \draw (1,0) .. controls (1,-.5) and (1.5,-.5) .. (1.5,0) ;
 \end{tikzpicture}
 \rb{+\;b\;} \begin{tikzpicture}
     \draw (0,0) -- (2,0) ;
     \draw (1,0) -- (1,-1);
     \draw (.5,0) .. controls (.5,-.5) and (1.5,-.5) .. (1.5,0) ;
 \end{tikzpicture}
 \rb{+\;a\;} \begin{tikzpicture}
     \draw (0,0) -- (2,0) ;
     \draw (1.5,0) -- (1.5,-1);
     \draw (.5,0) .. controls (.5,-.5) and (1.0,-.5) .. (1,0) ;
 \end{tikzpicture}
\]
\medskip \medskip

\noindent 
The action of 
\[
\UI \;\;\;\mbox{ and } \;\;\;  \IU 
\]
must be by 0, so in general 
$\xx a+b+c=0$ and $a+b+\xx c=0$. 
The action of the first elementary transposition
\[
\SigI  
\]
must be by 1 or -1, yielding $b=c$ and $a+2b=0$ in the + case; and $a=0$ and $b=-c$ in the - case, when $\xx=1$. Hence 1d solution spaces. 
(Note that here there is no second elementary transposition! If there were, the system would change radically.)
Thus we see that both one-dimensional $\Delta$-modules are submodules. 
By globalisation this means that we have 
\[
\Delta^{n}_{(3,(2))} \oplus \Delta^{n}_{(3,1^2)} \; \rightarrow\; \Delta^{n}_{(1,(1))}
\]
for all odd $n$, with both summands having nonzero image. 
We draw this as
\[
\hspace{-2.1cm} 
\xymatrix{
   &&&&&& \;\bullet_0 \ar@{-}[d] \\
\ver{(7,\!(2))} \ar@{-}[r] & 
\ver{(6,\!(2))} \ar@{-}[r] & 
\ver{(5,\!(2))} \ar@{-}[r] & 
\ver{(4,\!(2))} \ar@{-}[r]  & 
\ver{(3,\!(2))} \ar@{-}[r] \ar@[red]@/^1.3pc/[rr] & 
\ver{(2,\!(2))} \ar@{-}[r] & 
\bullet_1 \ar@{-}[r] & 
\ver{(2,\!(1^2))} \ar@{-}[r] & 
\ver{(3,\!(1^2))} \ar@{-}[r] \ar@[red]@/_1.3pc/[ll] & 
\ver{(4,\!(1^2))} \ar@{-}[r] & 
\bullet \ar@{-}[r] & 
\bullet \ar@{-}[r] & 
\ver{(7,\!(1^2))} & 
}
\] 

Looking at $n=4$ we see that by Frobenius reciprocity, 
with $A = \Delta^{n-1}_{(3,\lambda)}$ ($\lambda = (2), (1^2)$) and 
$B=\Delta^{n}_{(0,0)}$ in (\ref{eq:FR}), 
we have 
\[
\Delta^{n}_{(4,(2))} \oplus \Delta^{n}_{(4,1^2)} \; \rightarrow\; \Delta^{n}_{(0,(0))}
\]
On the other hand with $A=\Delta^{n-1}_{(3,(2))}$ and $B=\Delta^{n}_{(2,(1^2))}$ we have
\[
\hom(\Delta^{n}_{(4,(2))} \oplus \Delta^{n}_{(2,(2))} , \Delta^{n}_{(2,(1^2))} ) 
 \cong  
\hom(\Delta^{n-1}_{(3,(2))} , \Delta^{n-1}_{(3,(1^2))} + \Delta^{n-1}_{(1,(1))}) 
\]
The sum on the right is non-split, but both factors have $ \Delta^{n-1}_{(3,(2))} $ in the socle, so  the dimension on the right is at least 1. 
And hence also on the left - and by globalisation there is no map from 
$ \Delta^{n}_{(2,(2))}$ to $\Delta^{n}_{(2,(1^2))} $. So we deduce:
\[
\Delta^{n}_{(4,(2))}  \rightarrow   \Delta^{n}_{(2,(1^2))}
\]
(It is worth thinking about other choices for $A$ and $B$!... .) 
And similarly for ... by symmetry. 
Altogether that is:  
\[
\hspace{-2.1cm} 
\xymatrix{
   &&&&&& \;\bullet_0 \ar@{-}[d] \\
\ver{(7,\!(2))} \ar@{-}[r] & 
\ver{(6,\!(2))} \ar@{-}[r] & 
\ver{(5,\!(2))} \ar@{-}[r] & 
\ver{(4,\!(2))} \ar@{-}[r] \ar@[red]@/^1.3pc/[rrru] \ar@[red]@/^1.9863pc/[rrrr] & 
\ver{(3,\!(2))} \ar@{-}[r] \ar@[red]@/^1.3pc/[rr] & 
\ver{(2,\!(2))} \ar@{-}[r] & 
\bullet_1 \ar@{-}[r] & 
\ver{(2,\!(1^2))} \ar@{-}[r] & 
\ver{(3,\!(1^2))} \ar@{-}[r] \ar@[red]@/_1.3pc/[ll] & 
\ver{(4,\!(1^2))} \ar@{-}[r] \ar@[red]@/_1.3pc/[lllu] \ar@[red]@/_1.93pc/[llll] & 
\bullet \ar@{-}[r] & 
\bullet \ar@{-}[r] & 
\ver{(7,\!(1^2))} & 
}
\] 

Notice that this is consistent with the gram matrix results as in Fig.\ref{fig:Bratelli-l0-gramdets}. 
In particular there is not a morphism from $(4,(2)) \rightarrow (2,(2))$ 
(in the obvious shorthand); nor the symmetry image version. 

Passing on to $n=5$ we see, taking $A=(4,(2))$ and $B=(3,(1^2))$ that 
\[
\hom(\Delta^{n}_{(5,(2))} \oplus \Delta^{n}_{(3,(2))} , \Delta^{n}_{(3,(1^2))} ) 
 \cong  
\hom(\Delta^{n-1}_{(4,(2))} , \Delta^{n-1}_{(2,(1^2))} + \Delta^{n-1}_{(4,(1^2))}) 
\]
Obviously there is a map $(4,(2))\rightarrow (2,(1^2))$, so dim-hom $\geq 1$ on the right unless the $(4,(1^2))$ is glued under - but this would imply a map that we 
noted above  
is not there. Thus dim-hom on the left is nonzero. This forces 
$(5,(2))\rightarrow (3,(1^2))$; and its symmetry image. 
Altogether: 
\[
\hspace{-2.1cm} 
\xymatrix{
   &&&&&& \;\bullet_0 \ar@{-}[d] \\
\ver{(7,\!(2))} \ar@{-}[r] & 
\ver{(6,\!(2))} \ar@{-}[r] & 
\ver{(5,\!(2))} \ar@{-}[r] \ar@[green]@/^2.963pc/[rrrrrr] & 
\ver{(4,\!(2))} \ar@{-}[r] \ar@[red]@/^1.3pc/[rrru] \ar@[red]@/^1.9863pc/[rrrr] & 
\ver{(3,\!(2))} \ar@{-}[r] \ar@[red]@/^1.3pc/[rr] & 
\ver{(2,\!(2))} \ar@{-}[r] & 
\bullet_1 \ar@{-}[r] & 
\ver{(2,\!(1^2))} \ar@{-}[r] & 
\ver{(3,\!(1^2))} \ar@{-}[r] \ar@[red]@/_1.3pc/[ll] & 
\ver{(4,\!(1^2))} \ar@{-}[r] \ar@[red]@/_1.3pc/[lllu] \ar@[red]@/_1.93pc/[llll] & 
\bullet \ar@{-}[r] \ar@[green]@/_3.093pc/[llllll] & 
\bullet \ar@{-}[r] & 
\ver{(7,\!(1^2))} & 
}
\] 
Here there are the possibilities of composite maps:
\[
(5,(2)) \rightarrow (3,(1^2)) \rightarrow (1,(1)) 
\]
and symmetry image. Recall that in the $l=-1$ case all composite maps are zero 
(see \cite[Ch.6]{Martin91}).  
However here we deduce from the gram determinant that 
$ (3,(1^2)) \hookrightarrow (1,(1)) $, and hence the composite is non-zero. 

\subsubsection{Subcase  $(\alpha^2 + \alpha -4)=0$}

\mdef  
Let us have a look at $\alpha$ such that $(\alpha^2 + \alpha -4)=0$. 
Observe from Fig.\ref{fig:Bratelli-l0-gramdets} (say) that 
the first singular gram matrix is for  $\Delta^{4}_{(2,(2))}$. 
By globalisation the only possible maps in are from 
$\Delta^{4}_{(4,(2))}$ or from $\Delta^{4}_{(4,(1^2))}$. 
It is easy to check 
(see e.g. \S\ref{ss:bootcase0})
that there is no  $\Delta^4_{(4,(1^2))}$ subspace
(unless $\alpha=1$), 
so the subspace is $\Delta^4_{(4,(2))}$ here. 

By globalisation this means that we have module morphisms 
\beq \label{eq:morphismto22} 
\Delta^n_{(4,(2))}  \rightarrow \Delta^n_{(2,(2))}
\eq 
for all even $n\geq 4$.
Thus:
\[
\hspace{-2.1cm} 
\xymatrix{
   &&&&&& \;\bullet_0 \ar@{-}[d] \\
\ver{(7,\!(2))} \ar@{-}[r] & 
\ver{(6,\!(2))} \ar@{-}[r] & 
\ver{(5,\!(2))} \ar@{-}[r] & 
\ver{(4,\!(2))} \ar@{-}[r] \ar@[red]@/^1.3pc/[rr] & 
\ver{(3,\!(2))} \ar@{-}[r] & 
\ver{(2,\!(2))} \ar@{-}[r] & 
\bullet_1 \ar@{-}[r] & 
\ver{(2,\!(1^2))} \ar@{-}[r] & 
\ver{(3,\!(1^2))} \ar@{-}[r] & 
\ver{(4,\!(1^2))} \ar@{-}[r] & 
\bullet \ar@{-}[r] & 
\bullet \ar@{-}[r] & 
\ver{(7,\!(1^2))} & 
}
\]

The next singularity is for $\Delta^{5}_{(1,(1))}$. 
The map in here is from 
$\Delta^{5}_{(5,(2))}$ or from $\Delta^{5}_{(5,(1^2))}$. 
By Frobenius reciprocity 
- see (\ref{exa:henri1}) below for details - 
the $\lambda\in\{ (2),(1^2)\}$ will be the same as for the first map above. 
Thus we have 
\beq \label{eq:morphismto11} 
\Delta^n_{(5,\lambda)}  \rightarrow \Delta^n_{(1,(1))}
\eq 
with $\lambda=(2)$, 
for all odd $n \geq 5$. 
Thus: 
\[
\hspace{-2.1cm} 
\xymatrix{
   &&&&&& \;\bullet_0 \ar@{-}[d] \\
\ver{(7,\!(2))} \ar@{-}[r] & 
\ver{(6,\!(2))} \ar@{-}[r] & 
\ver{(5,\!(2))} \ar@{-}[r]  \ar@[red]@/^1.7pc/[rrrr] & 
\ver{(4,\!(2))} \ar@{-}[r] \ar@[red]@/^1.3pc/[rr] & 
\ver{(3,\!(2))} \ar@{-}[r] & 
\ver{(2,\!(2))} \ar@{-}[r] & 
\bullet_1 \ar@{-}[r] & 
\ver{(2,\!(1^2))} \ar@{-}[r] & 
\ver{(3,\!(1^2))} \ar@{-}[r] & 
\ver{(4,\!(1^2))} \ar@{-}[r] & 
\bullet \ar@{-}[r] & 
\bullet \ar@{-}[r] & 
\ver{(7,\!(1^2))} & 
}
\]

At $n=6$ we already have a map from  (\ref{eq:morphismto22}). 
Note that this explains the $(\alpha^2+\alpha-4)^6$ in the gram-det of 
$ \Delta^6_{(2,(2))}  $, since the map is   
from a 6-dimensional simple module. 
In addition we have two further singularities here, for 
$ \Delta^6_{(0,(0))}  $    and   $ \Delta^6_{(2,(1^2))}  $. 
By Frobenius reciprocity it seems that both should be from 
$ \Delta^6_{(6,\lambda=(2))}  $. 
Altogether we add to the figure as shown in red in Fig.\ref{fig:standardmapshenri}. 

\begin{figure}
\[
\hspace{-2.1cm} 
\xymatrix{
   &&&&&& \;\bullet_0 \ar@{-}[d] \\
\ver{(7,\!(2))} \ar@{-}[r] \ar@[green]@/^2.643pc/[rrrrrr] \ar@[green]@/^3.643pc/[rrrrrrrr] & 
\ver{(6,\!(2))} \ar@{-}[r] \ar@[red]@/^1.7pc/[rrrrru] \ar@[red]@/^2.1pc/[rrrrrr] & 
\ver{(5,\!(2))} \ar@{-}[r]  \ar@[red]@/^1.7pc/[rrrr] & 
\ver{(4,\!(2))} \ar@{-}[r] \ar@[red]@/^1.3pc/[rr] & 
\ver{(3,\!(2))} \ar@{-}[r] & 
\ver{(2,\!(2))} \ar@{-}[r] & 
\bullet_1 \ar@{-}[r] & 
\ver{(2,\!(1^2))} \ar@{-}[r] & 
\ver{(3,\!(1^2))} \ar@{-}[r] & 
\ver{(4,\!(1^2))} \ar@{-}[r] & 
\bullet \ar@{-}[r] & 
\bullet \ar@{-}[r] & 
\ver{(7,\!(1^2))} & 
}
\]
    \caption{ Standard module morphisms for $l=0$ for $(\alpha^2+\alpha-4)=0$. \label{fig:standardmapshenri} }
\end{figure}
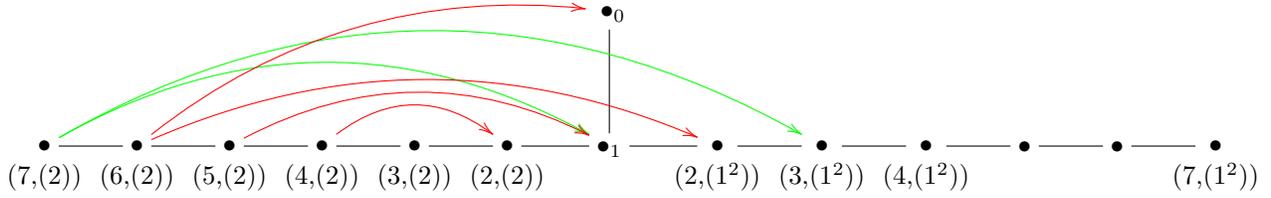

At $n=7$ we already have a map from (\ref{eq:morphismto11}). 
Note that this explains a factor $(\alpha^2+\alpha-4)^7$ in the gram-det of $\Delta^7_{(1,(1))}$. 
But in fact the exponent is 8. 
On the other hand we can consider if Frobenius reciprocity gives us a map in 
also from $ \Delta^7_{(7,\lambda)}$.  
In that case we would have two socle factors. 
The other singular case is $(3,(1^2))$, but for this just consider Frobenius reciprocity applied to 
$   \Delta^6_{(6,\lambda)} \rightarrow   \Delta^6_{(2,(1^2))}  $ 
from above. 
See the red and green arrows in Fig.\ref{fig:standardmapshenri}. 

\mdef  \label{pa:FR}   {\bf{Frobenius reciprocity}}. 
Here we briefly review and apply the {\em tower of recollement} method of \cite{ToR}. 
Recall that given an inclusion of algebras $ X \subset Y$ and an $X$-module $A$ and
a $Y$-module $B$, 
then Frobenius reciprocity is the isomorphism of hom-spaces
\beq \label{eq:FR} 
\hom_Y(Ind\;A,B) \cong \hom_X(A,Res\;B)
\eq 
where $Ind-$ and $Res-$ refer to the induction and restriction along $ X \subset Y$. 

\mdef   \label{exa:henri1}
For example consider $J_{0,n-1} \subset J_{0,n}$ with $(\alpha^2+\alpha-4)=0$. 
From (\ref{eq:morphismto22}) above we have the map  
$ \Delta^n_{(4,(2))}  \rightarrow \Delta^n_{(2,(2))} $ ($n$ even). 
Now note that when $n=5$ we have 
\[
Res \; \Delta^n_{(1,(1))} =  \Delta^{n-1}_{(2,(2))} \oplus \Delta^{n-1}_{(2,(1^2))} \oplus \Delta^{n-1}_{(0,(0))}
\]
(it is a general result for all $n$ odd that the factors on the right are filtration factors, but here they are in different blocks)
so we see that 
$$
\hom(\Delta^{n-1}_{(4,(2))}, Res \; \Delta^n_{(1,(1))}) \neq 0   .
$$ 
By reciprocity therefore 
$$
\hom_Y(Ind\; \Delta^{n-1}_{(4,(2))},  \Delta^{n}_{(1,(1))}) \neq 0   .
$$ 
We have 
\[
Ind\; \Delta^{n-1}_{(4,(2))} = \Delta^n_{(5,(2))} \oplus \Delta^n_{(3,(2))}
\]
- certainly a direct sum when $n=5$. There is no map 
$ \Delta^n_{(3,(2))} \rightarrow \Delta^n_{(1,(1))} $ 
- else it would have an image at lower $n$ by localisation. 
Thus we must have a map $ \Delta^n_{(5,(2))} \rightarrow \Delta^n_{(1,(1))} $. 

\medskip 

The above example  confirms completely what we already part-deduced from the gram-dets in our argument for $n=5$. 
Now let us try the same method for the $n=6$ cases. 
Consider
\[Res(\Delta^6_{(0,(0))})= \Delta^5_{(1,(1))}  \qquad  Res(\Delta^6_{(2,(1^2))})= \Delta^5_{(1,(1))}\oplus \Delta^5_{(3,(1^2))}\]
Therefore, we have 
\[  \hom(\Delta^5_{(5,(2))}, Res\Delta^6_{(0,(0))})\neq 0\]
hence
\[\hom(Ind\Delta^5_{(5,(2))}, \Delta^6_{(0,(0))})\neq 0\]
and $Ind\Delta^5_{(5,(2))}=\Delta^6_{(4,(2))} \oplus \Delta^6_{(6,(2))}$, we would have seen maps $\Delta^6_{(4,(2))}\to \Delta^6_{(0,(0))}$ at a lower stage so we must have $\Delta^6_{(6,(2))}\to \Delta^6_{(0,(0))}$ (and likewise $\Delta^6_{(6,(2))}\to \Delta^6_{(2,(1^2))}$ ).

\appendix
\newcommand{\texdd}{tex2}
\newcommand{\xfigdd}{tex2/xfig-bratelli5}
\newcommand{\xfigGST}{tex2/xfig-gramdetbyGST}
\vspace{1cm} \newpage 
\section*{Appendix}
The next bit is an extract from \cite{paulsonlinelecturenotes}, slightly adapted to our purpose. 
One should keep in mind a few local notations as follows. \\ 
Temperley--Lieb is our case $l=-1$ (with our loop parameter $\xx$ written as $\delta$). \\ 
$M_n(\lambda)$ is the gram matrix (in the integral/diagram basis), with 
$\lambda$ simply the number of propagating lines - 
$\lambda =p$, so that $\lambda=0$ below is our $(0,\emptyset)$ and $\lambda=p>0$ is $(p,(1))$. \\
The algebra element $U_i$ here is $U_i = \uu_{i,i+1} \uu_{i,i+1}^*$. \\
$E_m^{(r)} = 1_r \otimes \nE_{m-1}$ (the image of the symmetriser, cf. $\nE_n$ from (\ref{pa:Specht})). 


\section{Determinant computation: Generic structure theorem method}  \label{ss:gdTL}

Our next method for computing the
Temperley--Lieb standard 
gram determinant
$\det(M_n(\lambda ))$
uses the generic structure theorem \cite[Th.1]{Martin91}. 
(Let us reiterate that there is no need, in representation theory, 
to compute any but the
elementary $\lambda=n-2$ case. We continue merely as an exercise in
service of later generalisations.)

The theorem gives a basis of the TL algebra for generic parameter
over a suitable field in the form of a complete
orthonormal set for the multimatrix structure.
In this case we can extract a basis for each cell module such 
that, by orthonormality, the gram matrix is the unit matrix. 
Unlike the gram matrix over the defining `diagram' basis, this is of no
intrinsic use:  
manifestly since it contains no information; but implicitly
since we cannot specialise the 
parameter as we can if we work over the `integral' (polynomial) ring.
The trick is to keep track of the renormalisations of elements 
in passing from the integral basis, so that we can reconstruct
the determinant over the integral basis.  

The basis elements we need are defined as follows
(we extract directly from \cite[\S6.4]{Martin91}).

To start with, 
fix a standard/cell module $\Delta_n(\lambda)$ and consider the walk 
enumeration of the diagram basis. 
This consists of 
walks 
$s=(s_0,s_1, s_2, ..., s_{n})$ 
of length $n$ on the positive part of the integral line
that start at 1 and end at $\lambda+1$.
The `lowest' walk can be depicted:
\[
e_m \; = \; 12121212123...m \; = \; 
\includegraphics[width=4cm]{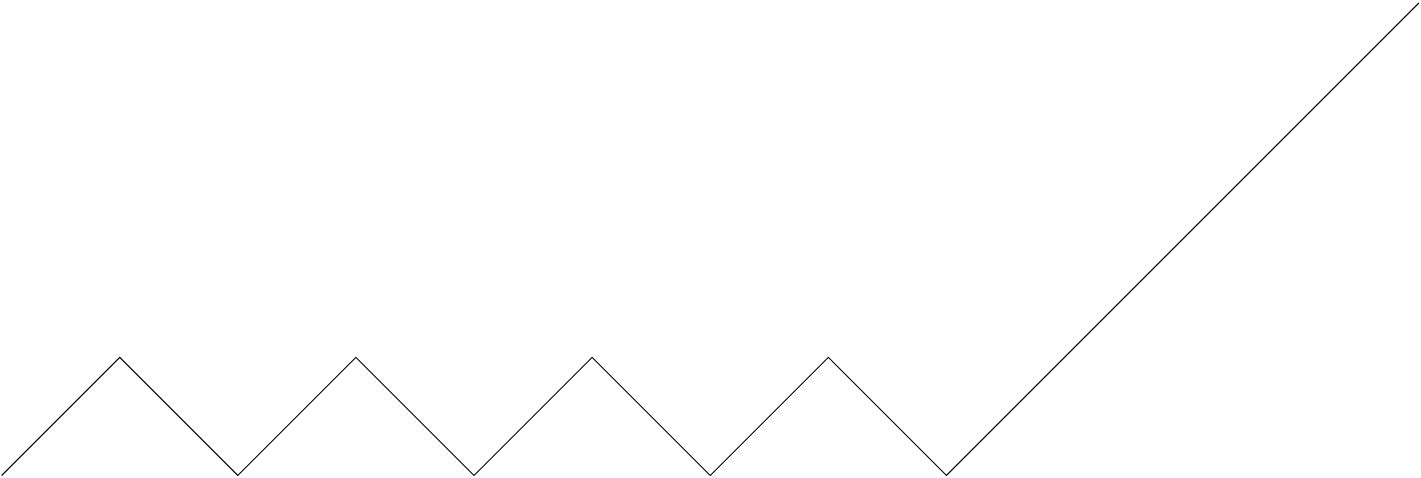}
\]
where $m=\lambda+1$, for example. 
We will recall the correspondence between diagrams and walks shortly. 

Algebra basis elements correspond to pairs of these walks to the same endpoint, 
for all possible endpoints. 
Thus
with the endpoints running over the  $ \lambda \mapsto m=\lambda+1$ shifted possible 
cell module indices $\lambda$. Recall the $\lambda$s are  $n,n-2,...,0/1$, 
so we have  these shifted to $n+1,n-1,...,1$ or $...,2$. 

The map from pairs of walks to algebra elements can be given iteratively. 
Firstly the lowest walk (pair) for each $\lambda$ is mapped to 
an element by, for example with $m=\lambda+1$: 
\[
(e_m, e_m) = (121212123...m,121212123...m) = U_1 U_3 U_5 E_{m}^{(6)}
\]
(see \cite[\S6.4]{Martin91} for notation) (note normalisation as
idempotent needed). 
Then if $s_i =g-1$ is a minimum of
sequence $s$, and $s^i$ denotes $s$ with $s_i$ replaced by $g+1$:
\[
(s^i,t) = \; \sqrt{k_g k_{g+1}} \left( 1-\frac{U_i}{k_g} \right) (s,t)
\]
where $k_g = \frac{[g-1]}{[g]}$. 
And similarly on the right. 

The generic structure theorem \cite[\S6.4]{Martin91} ensures that these elements are an orthonormal basis of the TL algebra. 

The integral/diagram basis can be considered to start with 
$(e_m,e_m)$ (up to quotients and id normalisation - again see \cite{Martin91}). 
The next diagram basis element, corresponding to the walk 
$123212123...m$, is
$U_2 (e_m,e_m)$; 
and more generally $(s^i,t)_{diagram} = U_i (s,t)_{diagram}$. 
Thus we see that the new orthonormal basis element $(s^i,t)$ in this
case is got by adding a scalar multiple of the previous  
(irrelevant for the determinant);
and rescaling   $U_i (s,t)$  
by 
\[
- \sqrt{\frac{[h-1]}{[h+1]}} \frac{[h]}{[h-1]}  \; =  \; - \sqrt{\frac{[h]^2}{[h+1][h-1]}}
\]
($h=g$).
This applies both to the bra and the ket, and both rescale the determinant. 
For a general walk there is a factor like this for each diamond added
to get to it from the $121212123...m$ walk.
Thus altogether we have the following. 

\begin{lem}
Schematically, 
\[
\det(M_n(\lambda)) = \prod_s \prod_d {\frac{[h]^2}{[h+1][h-1]}}
\]
where the products are over the set of walks and the diamonds in each
walk. 
\qed
\end{lem}

Example: The factors for each diamond are (in abridged notation):
\[
\includegraphics[width=4.47cm]{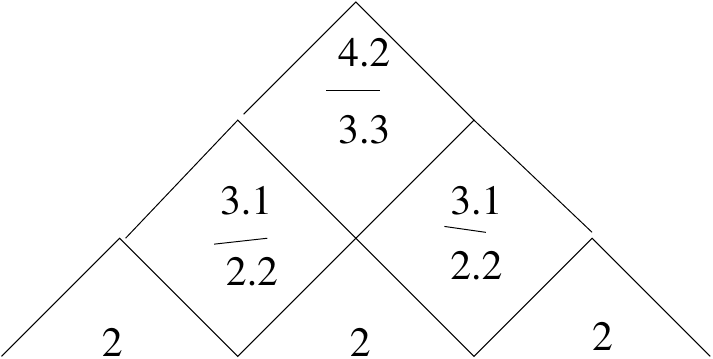}
\]
Thus in particular
\[
\includegraphics[width=4.47cm]{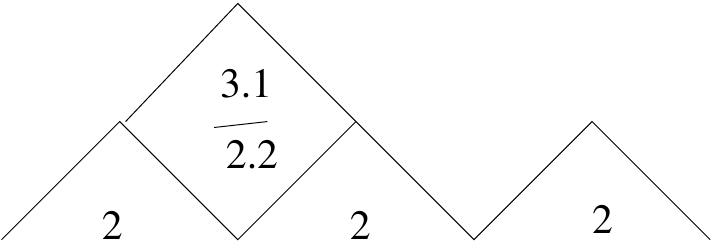}
\]
gives a factor $\frac{[3][2]}{1}$. Altogether we have
\[
\det( M_6(\lambda=0 )) = \frac{[2]^3}{1}
    \frac{[3][2]}{1} \frac{[3][2]}{1} \frac{[3]^2}{[2]^{}} \frac{[4]}{1} 
 = [2]^4 [3]^4 [4]
 = \delta^5 (\delta^2 -1)^4 (\delta^2 -2)
\]

Example II:
\[
\includegraphics[width=5.37cm]{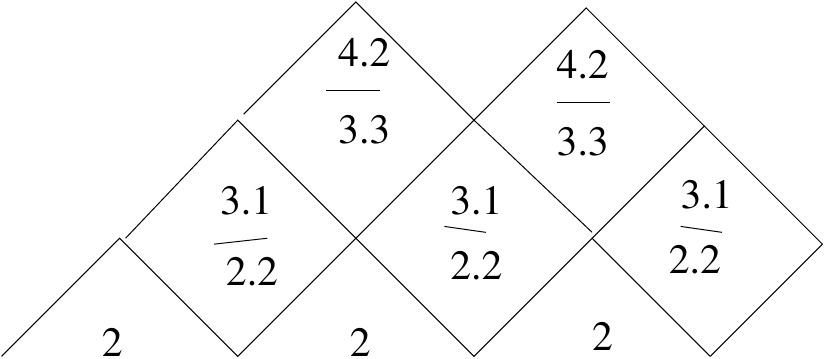}
\]
gives a factor $\frac{[4]^2 }{[2][3]}$. Altogether this det comes out as
\[
\det( M_7(\lambda=1 )) = \frac{[2]^{3\times 14}}{1}
    \left( \frac{[3]   }{[2]^2} \right)^{9+10+9} 
    \left( \frac{[4][2]}{[3]^2} \right)^{4+4}
    \left( \frac{[5][3]}{[4]^{2}} \right)^{1} 
 = [2]^{-6} [3]^{13} [4]^6 [5]
\]

Observe that the determinants can be computed recursively.
The walks that pass through $(\lambda-1,n-1)$ do not pick up another diamond in the final step, so contribute a factor $ det(M_{n-1}(\lambda-1)) $;
and those that pass $(\lambda+1,n-1)$ all pick up the same string of diamonds,
compared to their contribution to $ det(M_{n-1}(\lambda+1)) $. 
For example:
\[ { } \hspace{-.32in} 
\includegraphics[width=5.4cm]{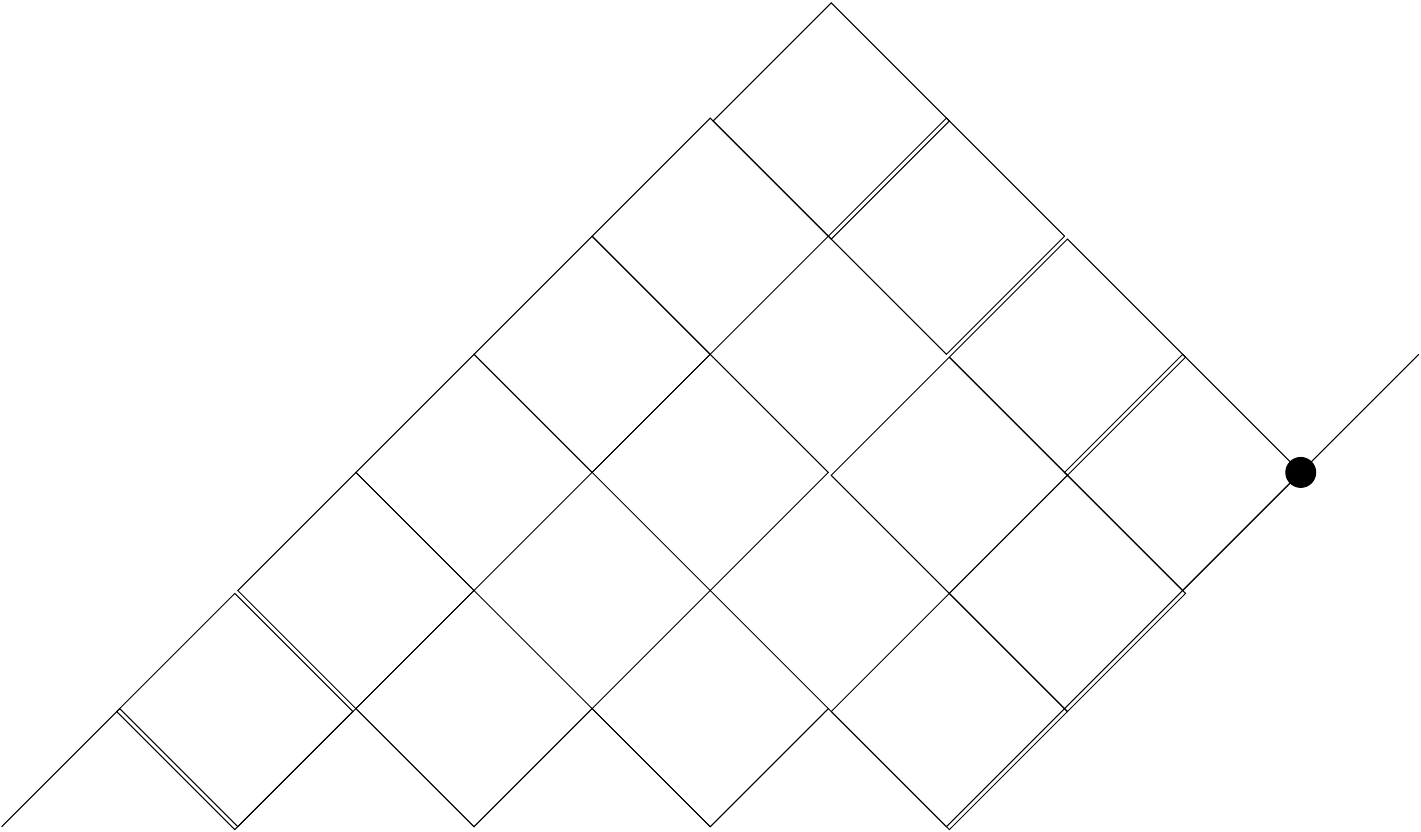}
\qquad
\includegraphics[width=5.4cm]{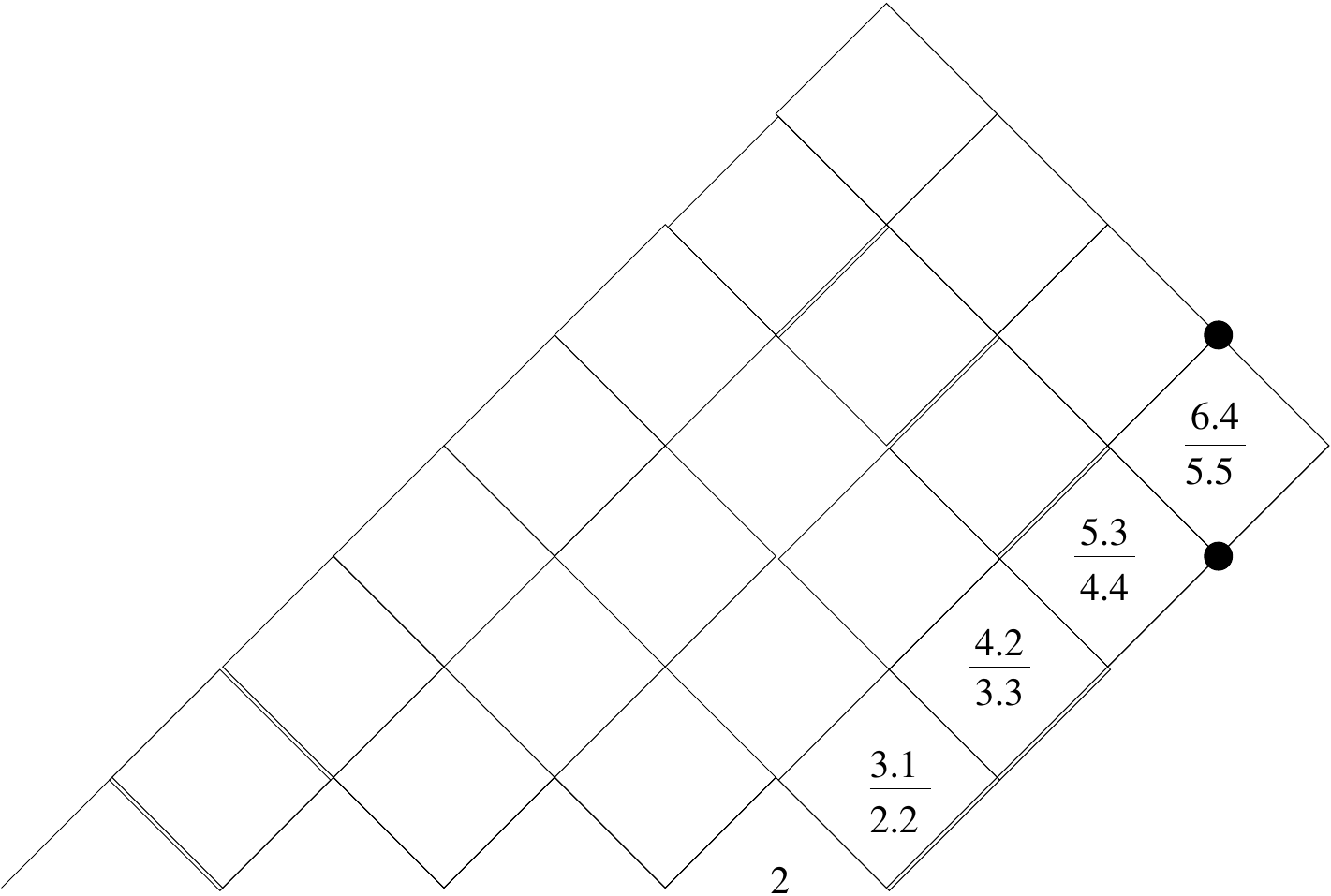}
\qquad
\includegraphics[width=5.04cm]{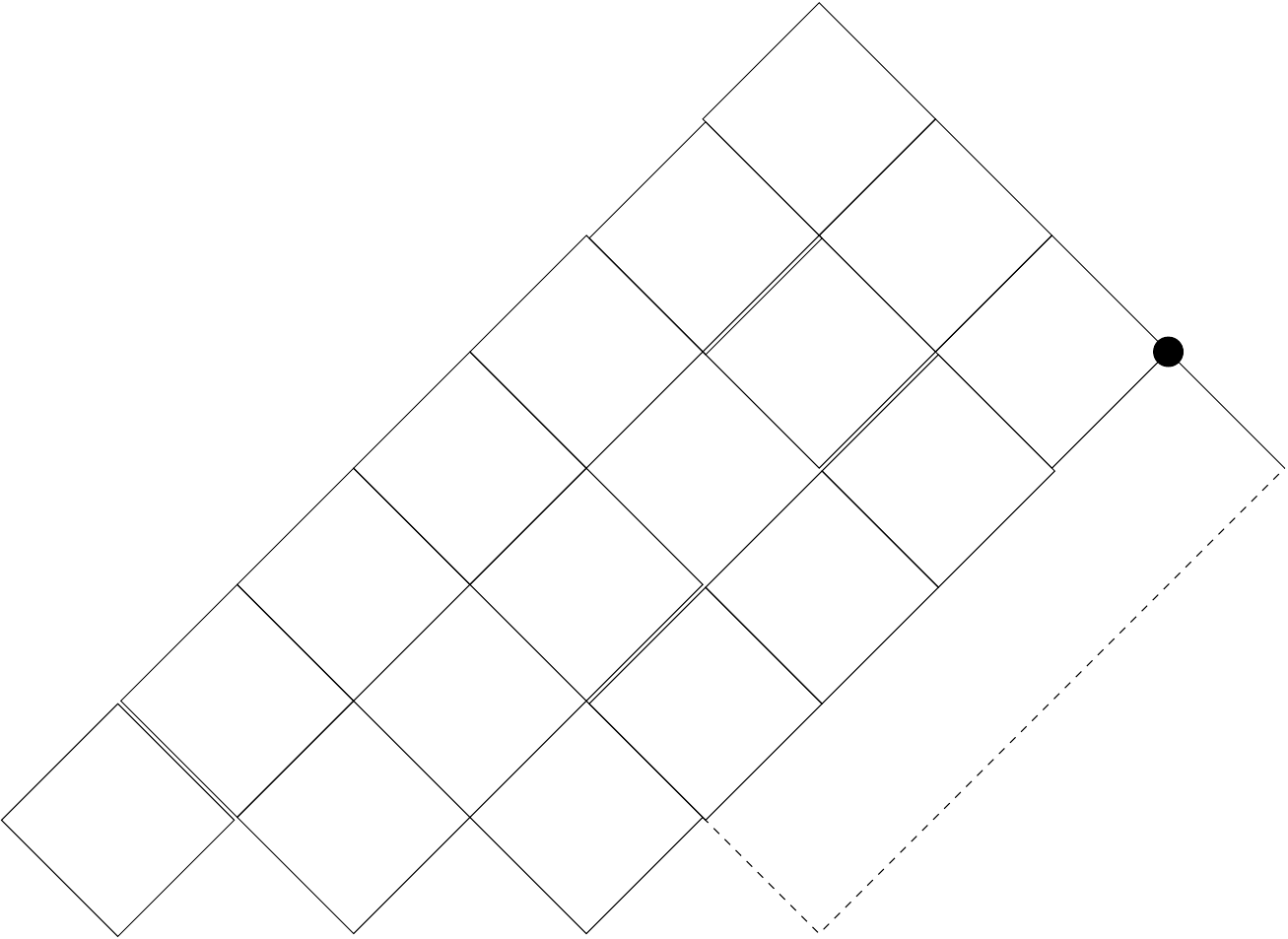}
\]
We have
\beq \label{eq:det-recurs}
det(M_n(\lambda)) \;\; = \;\;\;\;
   det(M_{n-1}(\lambda-1)) \;\; . \;\;\;
    F_n(\lambda+1)^{dim(M_{n-1}(\lambda+1))} . det(M_{n-1}(\lambda+1))
\eq 
where the diamond-string factor is given by 
$$
F_n(\lambda) \;=\; [2] \frac{[3].1}{[2].[2]} \frac{[4].[2]}{[3].[3]} \frac{[5].[3]}{[4].[4]}
                     ... \frac{[\lambda+1].[\lambda-1]}{[\lambda].[\lambda]}
             \;\; = \; \frac{[\lambda+1]}{[\lambda]}
$$
             - note that it does not actually depend on $n$.

Example:
\[
det(M_5(1)) \; =\;\; det(M_4(0))\;.F_5(2)^{dim(M_4(2))}.\; det(M_4(2))
            \; = \; [2]^2 [3] . \left( \frac{[3]}{[2]}\right)^{3} . [4]
            = \; \frac{[4][3]^4}{[2]}
\]


\def\qqt(#1,#2,#3){$q^{#1}[2]^{#2}[3]^{#3}$}
\def\qqf(#1,#2,#3,#4){$q^{#1}[2]^{#2}[3]^{#3}[4]^{#4}$}
\def\qqv(#1,#2,#3,#4,#5){$q^{#1}[2]^{#2}[3]^{#3}[4]^{#4}[5]^{#5}$}

\ignore{{ 
\begin{figure}
\input \xfigdd/bratelli5r2new2.latex
\caption{Specht Gram determinant data on Young graph
  (unfinished - reproduced from paper).
  NB this fig may work in dvi (latex ... compile)
  and not perfectly in pdf (pdflatex ...);
  but can run dvipdf or dvipdfmx on the dvi for a pdf file.
\label{fig:bratelli5r2new}}
\end{figure}
}}

\ignore{\begin{figure} \vspace{-1cm} 
  \input \xfigdd/bratelli5r2newa.pdf_t
  \caption{Specht Gram determinant data on Young graph
    (unfinished - reproduced from paper; some later $q^i$ factors omitted,
    but $q$ is a unit here so these are not relevant;
    the hook lengths shown are also not particularly relevant here).
    \label{fig:stray1}}
\end{figure}}

%

\ignore{\begin{figure} \vspace{-1cm} 
  \input \xfigdd/bratelli5r2new3.pdf_t
  \caption{Coloured streaks indicate fibres projected to single point
    in TL global limit. \label{fig:streaky1}}
\end{figure}}


%


Now let us arrange the TL part of the result so that
the fibres are in vertical array under their $sl_2$ label
(and hence under the appropriate label in the projected Bratelli diagram): 
\[
\begin{array}{ccccccccccc}
  \lambda=0         &  1             &  2      &  3  &  4  &  5   
  \\ \hline
  1 \\ 
                   & 1  \\ 
   {[2]}           &                    & 1 \\ 
                  & [3]                &     & 1   \\ 
  {[ 2]}^2 [3]    &                     & [4] &      &  1  \\
                  & {[2]}^{-1} [3]^4 [4] &     & [5]   &    & 1 \\ 
{[2]}^4 [3]^4 [4]  &                    & {[2]}^{-1}[4]^5[5] &    & [6] & & 1 
\end{array}
\]
It is convenient to augment to also show the module dimensions (in red):
Fig.\ref{fig:gramdets01}.
And in Fig.\ref{fig:gramdets031} we include a schematic indication of the
recursive calculations from (\ref{eq:det-recurs}). 

\newcommand{\du}[2]{\begin{array}{c} \textcolor{red}{{#1}} \\ \hspace{-.321in}{#2}\hspace{-.321in} \end{array}}

\begin{figure}
  \input{tex2/gramdets01-}
  \caption{Standard Gram determinants restricting to the TL weights and then organising by $sl_2$ labels.
    \label{fig:gramdets01}}
\end{figure}

\begin{figure}
  \input{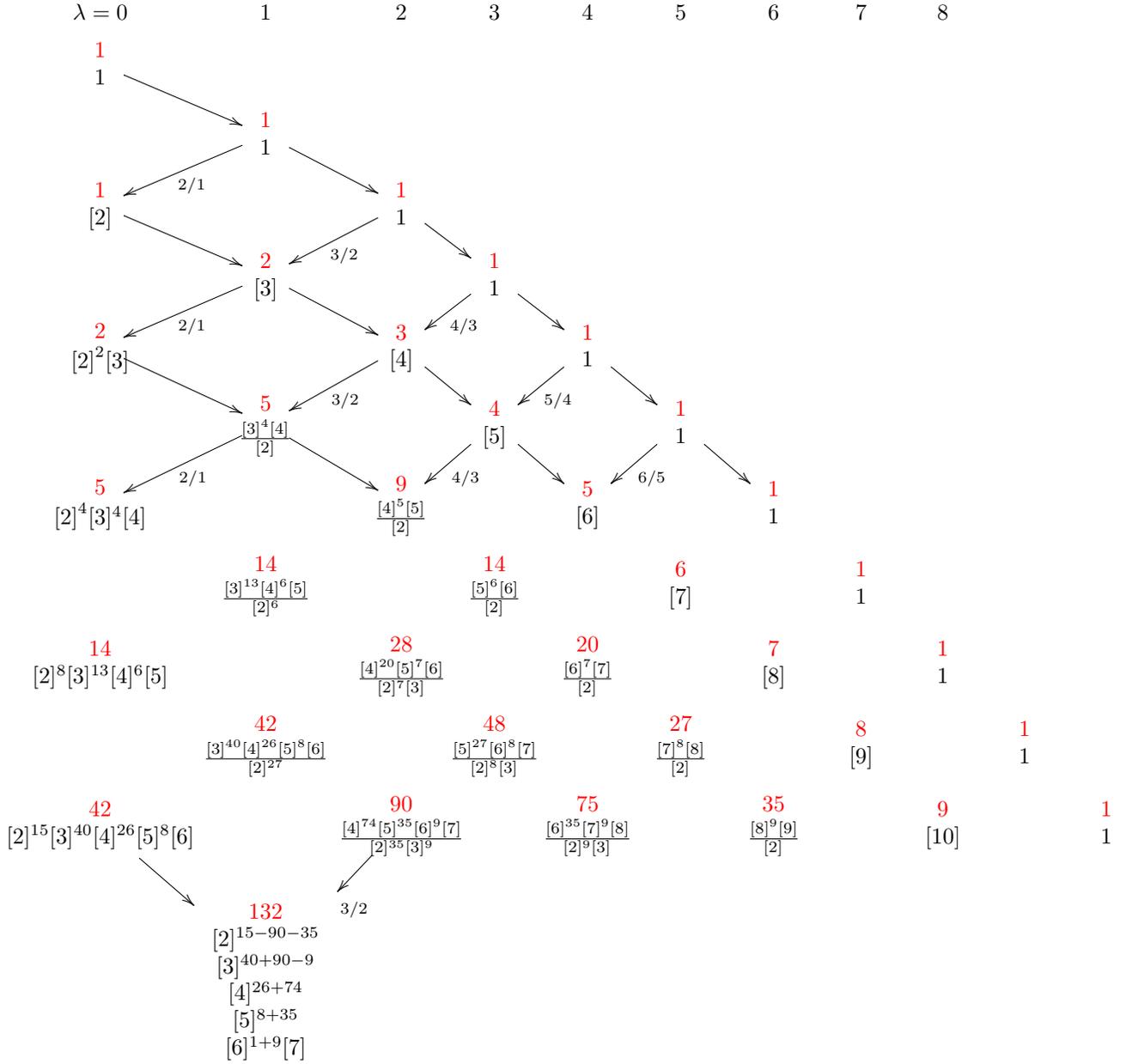}
  \caption{The recursive construction for gram determinants, as in (\ref{eq:det-recurs}).
    The weightings
    $ ( [\lambda\!+\!1]/[\lambda] )^{dim(M_{n-1}(\lambda+1))} $
    are drawn simply as
    $\lambda\!+\!1/\lambda$. 
    We omit drawing edges after the various patterns are established,
    apart from the explicit calculation shown at the bottom. 
    \label{fig:gramdets031}}
\end{figure}

\ignore{\begin{figure}
  \input{tex2/gramdets03-}
  \caption{\ppm{[NB this is a duplicate since we already copied in this fig earlier as \ref{fig:gramdets03--}!]}
  Standard to standard module maps when $[5]=0$.
    This is perhaps best viewed alcove-geometrically - the ambient space is $\R$;
    the origin is at $\lambda=-1$ (it is `$\rho$-shifted');
    and the first affine wall at $\lambda=-1+5=4$.  \label{fig:gramdets03}}
\end{figure}}

At its most succinct this scheme can be reduced to:
\[
\xymatrix@C=45pt{
  0  \ar@<.31ex>[r]^1   &
  1 \ar@<.31ex>[r]^1 \ar@<.21ex>[l]^{{{ } \;\; [2]/1}}    &
  2 \ar@<.31ex>[r]^1 \ar@<.21ex>[l]^{[3]/[2]}   &
  3 \ar@<.31ex>[r]^1 \ar@<.21ex>[l]^{[4]/[3]}  &
  4  \ar@<.21ex>[l]^{[5]/[4]}  &
  \cdots
}
\]
meaning that $det(M_n(\lambda))$ is the product of the $det(M_{n-1}(\mu))$'s
corresponding to the incoming arrows
(the $\lambda,\;\mu$ labels are on the vertices);
together with the factors {\em on}
the incoming arrows raised to the power of $dim(M_{n-1}(\lambda+1))$.

\medskip \vspace{.5in} 


\noindent 
To see how these determinants check against representation theory,
consider for example Fig.\ref{fig:gramdets03}.
Here we focus on the parameter value such that $[5]=0$.
The first determinant that vanishes is then indicated in the figure.
This vanishing indicates that there is a non-isomorphism map from
the corresponding costandard module to the standard module;
and hence that the standard module has a socle.
In this case it is one-dimensional and it is straightforward to deduce that
this corresponds to a map in from the trivial module.
--- And this is what the curved arrow indicates.

In the next layer we see a similar vanishing; and deduce a similar map.

In the next layer we have another. But we also have a globalisation image of
the very first map - where the module mapping in is
the globalisation image of the trivial, and so 
not the trivial.
Note that this has a direct affect on the degree of vanishing of the determinant (recall we have connected this to the Smith form of the gram matrix). 

Hopefully the subsequent pattern will be clear.
Perhaps the next interesting observation is in the last rank shown.
Here we have a map from the trivial module (of course $[5]=0$ implies $[10]=0$);
and we already have an on-map, due to the globalisation property. 
By inspecting the degrees of vanishing of the various gram matrices we can
deduce that the composite map vanishes here.

\vspace{.5cm}
\subsection{What about the Smith form?} 
More than the determinant evaluated over the ground field of interest,
$K$ say,
one cares about the rank of the gram matrix
(evaluated over the ground field of interest).
This is because the rank
 determines the rank of the contravariant form, and hence
the dimension of the simple head of the cell module over $K$. 
Elementary row and column operations expressible as multiplication by 
invertible matrices do not change the rank. Thus, if we are working
over a PID (as we are in the indeterminate-$\delta$ calculations 
if the coefficient ring is a field)
we are interested in the Smith normal form (SNF).

Reduced to the ground field of interest the SNF will give the rank
directly.
(Even if the ground ring is not a PID there may be a form which
directly reveals the rank.)
(The SNF may reveal even more than the rank in general. 
The power of vanishing of
invariant factors vanishing   
at the $\delta$-value 
of interest may reveal details of factor modules deeper in the
Jordan--Holder series of the cell module.) 

Example:
Consider the gram matrix
\[
\mat{cccc} \delta & 1 & 0 \\ 1 & \delta & 1 \\ 0 & 1 & \delta
\tam
\stackrel{R1:R1+R3}{\leadsto}
\mat{cccc} \delta & 2 & \delta \\ 1 & \delta & 1 \\ 0 & 1 & \delta
\tam
\stackrel{R1:R1-\delta R2}{\leadsto}
\mat{cccc} 0 & 2-\delta^2 & 0 \\ 1 & \delta & 1 \\ 0 & 1 & \delta
\tam
\stackrel{C2:C2-\delta C1;...}{\leadsto}
\]\[
\mat{cccc} 0 & 2-\delta^2 & 0 \\ 1 & 0 & 0 \\ 0 & 1 & \delta
\tam
\stackrel{R1:R1+(\delta^2 -2) R3}{\leadsto}
\mat{cccc} 0 & 0 & -\delta(2-\delta^2 )  \\ 1 & 0 & 0 \\ 0 & 1 & \delta
\tam
...
\]
We can compute the SNF in various ways. It is
\[
\mat{cccc} \delta (\delta^2 -2) && \\ &1& \\ &&1 \tam
\]
Method 1: elementary row and column operations.
\\
Method 2: compute minors and take the HCF of all minors at each rank. 
For example the HCF of $1\times 1$ minors in our example is clearly 1;
as is that of $2\times 2$ minors. The final invariant factor is then
forced.





\ignore{{
\vspace{1cm} \ppm{[--------- FIRST DRAFT OF PAPER COULD END HERE! ------]}
\newpage 

\section{\ppm{[Ignore everything from here!]}Towards a proof of \ournotmain\ conjecture \ref{conj:onm0}}

The conjecture is proved in case $l=-1$. One proof uses (the quantum version of) Young's seminormal form, and hence hook lengths, 
so one strategy in general is to try to generalise this - giving more `geometric' 
bases for standard modules. 

\subsection{Formulation of a 
Walk-Diagram Correspondence}
Here we aim to describe the emerging correspondence between walks in the Rollet graph and half-diagrams, describing a basis for a certain $J_{l,n}$ standard module. We will do this in baby steps, starting with the simple 1-cup, 1-D specht case:

\subsubsection{1-cup, 1-D Specht Case}
Here we will treat the case of the $J:=J_{l,n+2}$-module $\Delta:=\Delta_{n,(n)}$ (take $l+1\leq n$ WLOG), knowing that the case $\Delta:=\Delta_{n,(1^n)}$ is obtained by means of symmetry.

As noted in (\ref{de:Rollet}):
Vertices in the Rollet graph  $\Roll_l$ are labelled by a pair $(\l, p)$, where $\l\vdash \min(p,l+2)$ for $p\in \N$, and there is an edge $(\l, p)-(\l', p')$ if and only if (up to interchanging the pairs) $p'=p+1$, and either $\l'$ can be obtained from $\l$ by adding a box, or $\l'=\l$.

\mdef Claim.  
We claim that there exists a correspondence between walks of length $n+2$ between $(\emptyset,0)$ and $((1^n),n)$ (in the Rollet graph of height $l$), and a basis of half diagrams for $\Delta$. Observe that for $\Delta$ a basis is simply described by the position of a single cup, since the corresponding specht module is 1D. We write $e_{i,j}$ for the basis element with a cup joining position $i$ and $j$. 

A walk is then described by a sequence of vertices $(\l_k, p_k)$, $k=-1, \dots, n,$ with $(\l_{-1}, p_{-1})=(\emptyset,0)$, $(\l_0, p_0)=((1),1), (\l_{n}, p_{n})=((1^n), n)$. It seems good to divide the walks into two types: 

A walk $(\l_k, p_k)$ is \textbf{ideal} if for all $k$, $\l_k \subset (1^n)$. We claim that \textbf{ideal} walks correspond to the height -1 diagrams as follows: Such walks are determined by the sequence $(p_l)$ and contain at least one repeated integer. Let $j$ be the maximal repeated integer which may take values $j=1,\dots, n$. The corresponding half diagram is $e_{j,j+1}$.

Now let $(\l_k,p_k)$ be a non-ideal walk. Since the sequence $p_k$ decreases exactly once, it follows that any $\l_k\not\subset (1^n)$, must be of the form $(l,1)$ for some $l$, and furthermore such a $\l_k$ is a partition of $p_k$ (i.e. $p_k$ does not exceed $l+2$). Therefore, a non-ideal sequence is uniquely determined by the following two integers:
\[i=\min\{k \ | \ \l_k\not\subset (1^n)\}, \quad j=\max\{k \ | \ \l_k\not\subset (1^n)\}.\]
Note that $1\leq i \leq j \leq l+1\leq n$. We associate to this walk the diagram $e_{i,j+2}$, which we note has height $p_j-2=j-1$.

Why is this a bijection? even a heuristic argument? 

Examples:

Consider the $J_{2,8}$-module $\Delta_{6,(6)}$, we give the following examples:
\begin{align*}
    0&-(1)-(2)-(3)-(4)-(3)-(4)-(5)-(6), &&\sim e_{4,5} &(ht&=-1)\\
        0&-(1)-(2)-(2,1)-(3,1)-(3)-(4)-(5)-(6), &&\sim e_{2,5} &(ht&=2).\end{align*}

\subsubsection{Many-cup, 1-D Specht Case}

Try two cup first? In the previous case, we made use of the fact that the sequence $p_k$ decreased exactly once. Here it will decrease more than once, will be hard to figure out the config of cups! 

\subsubsection{1-cup, non-1-D Specht Case}

In the 1D specht case, there is a unique ideal walk. Here, given a walk, we will need a way to determine the underlying ideal walk.

\subsubsection{Full Generality}
Hopefully can combine the previous two points

\section{Outpatients} 
Here is nice stuff that might or might not be in the first paper. 

\subsection{On walk bases for $\Delta$-modules for $l\geq 0$}
\label{ss:walkbases0}

We can go further than the combinatorial correspondence in (\ref{pr:walks0}). 
\ppm{[-can we actually?!  there is a lot of colour here.  might minimize this section for the first submission???]}

\mdef   Example. 
Formally, 
the walk-basis for 
$\Delta^{l=1}_{5,((21),0)}$    
is:
\begin{equation}  \label{eq:walkbasis521}
0-1-0-1-(2)-(21),
\;\;\; \;\;\; \hspace{.421cm}
0-1-0-1-(11)-(21),  \hspace{2.29762cm} 
\end{equation}
\[ 
0-1-(2)-1-(2)-(21),
\;\;\; \hspace{.31cm} 
0-1-(2)-1-(11)-(21),   \hspace{.53cm} 
\;
0-1-(11)-1-(2)-(21),
\] 
\[ 
0-1-(11)-1-(11)-(21),  \;\;\;
0-1-(2)-(21)-(2)-(21), \;\;\;
0-1-(2)-(21)-(11)-(21),
\]
\[ 
0-1-(2)-(21)-(21)'-(21), \hspace{1.9753185cm} 
0-1-(11)-(21)-(2)-(21), \hspace{.2cm} 
\]
\[ 
0-1-(11)-(21)-(11)-(21),  \;\;\; \hspace{1.95cm} 
0-1-(11)-(21)-(21)'-(21),
\]
\[ 
0-1-(2)-(3)-(2)-(21), \;\;\; \hspace{.96cm} 
0-1-(11)-(111)-(11)-(21)
\]

\mdef   \label{exa:l0p6la2}
Example. 
Explicit partial bases for  
$\Deltaa^{l=0,n=6}_{(4,(2))}$ is:  
\medskip 

\includegraphics[width=10.9cm]{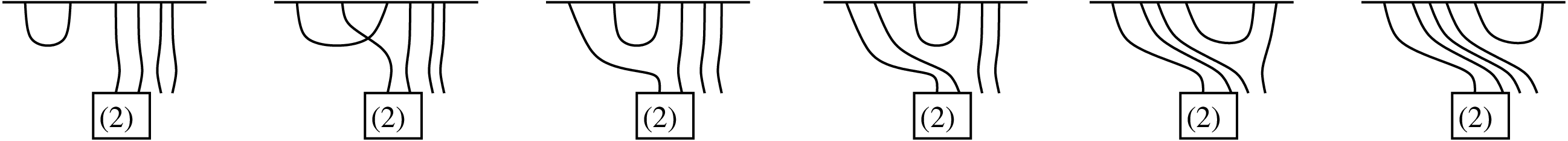}

\ppm{[So far the following is nonsense, but at least there are pics of diagrams. I'll fix it!...]}
The elements expressed as walks (writing $\lambda'''$ to mean $(\lambda,3)$ and so on):
\begin{align*}
    &0-1-0-1-(2)-(2)'-(2)''& &\sim e_{1,2}, \\
&0-1-(2)-1-(2)-(2)'-(2)'' &&\sim e_{2,3}, \\
&0-1-(11)-1-(2)-(2)'-(2)'' &&\sim e_{1,3}, \\
&0-1-(2)-(2)'-(2)-(2)'-(2)''&&\sim e_{3,4}, \\
&0-1-(2)-(2)'-(2)''-(2)'-(2)''&&\sim e_{4,5}, \\
& 0-1-(2)-(2)'-(2)''-(2)'''-(2)'' &&\sim e_{5,6},\end{align*}
(finish me!... \bajm{right? I am also suggesting the above correspondence with diagrams} \ppm{[-interesting. does that represent the application of an algorithm, or so far just a one-off?
For $l=-1$ of course the Rollet graph is just a simple chain $0-1-1'-1''-...$ which we could also label as $0-1-2-3-...$. There is a natural partial order on walks of length $n$ from 0 to (say) $s$, with $0-1-0-1-...-0-1-2-3-...-s$ the lowest and $0-1-2-...-max-...-(s+1)-s$ the highest. One can assign the diagram cup-cup-...-cup-line-line-...line to the lowest; and then construct the others iteratively so that if walk $w'$ is obtained from $w$ by $j-(j-1)-j \;\leadsto\; j-(j+1)-j$ (for some $j$) in position $i$ then the diagram $d'$ is given by $U_i d$. Does that make sense? (There is a more complicated version for Brauer. If we are lucky the general $l$ case (or at least some intermediate $l$ case) might be easy to interpolate.  - In which case this might even make the germ of a neat little paper.)]}\bajm{I think this represents the application of an algorithm which I don't quite understand yet :)... but seems quite accesible}).

Aside. Here is another set of walks (for a different module! \bajm{for $\Delta^{l\geq 0}_{6,((2),0)}$ ?}):
\begin{align*} 
&0-1-0-1-0-1-(2), &&
0-1-0-1-(2)-1-(2),\\
&0-1-0-1-(11)-1-(2),
&&0-1-(2)-1-0-1-(2),
\\ &0-1-(2)-1-(2)-1-(2),
&&0-1-(2)-1-(11)-1-(2),
\\ &0-1-(11)-1-0-1-(2),
&&0-1-(11)-1-(2)-1-(2),
\\&0-1-(11)-1-(11)-1-(2).\end{align*}
(Is this correct?! \bajm{There should be more walks no? None of these walks go past $(2)$ in the rollet graph} \ppm{[-yes agreed. (I'm working in spurts as pain levels allow, so there'll be some mud in my tex-ing waters. hopefully more good than bad overall.  :-) )]})

\medskip  \noindent  Another example is
$\Delta^{l=1}_{4,(2)}$:  \medskip 

\includegraphics[width=10cm]{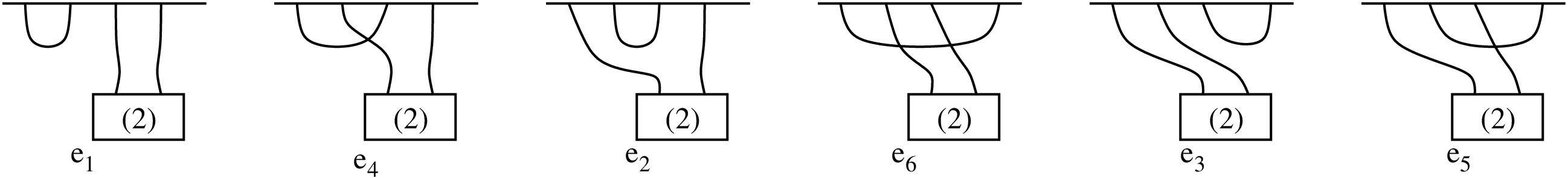}

The elements expressed as walks:  
\begin{align*}
    &0-1-0-1-(2), &&\sim e_{1,2}, &&0-1-(2)-1-(2),  &&\sim e_{2,3}, \\
&0-1-(2)-(3)-(2), &&\sim e_{3,4}, &&0-1-(2)-(21)-(2), &&\sim e_{2,4},\\
&0-1-(11)-1-(2), &&\sim e_{1,3}, &&0-1-(11)-(21)-(2), &&\sim e_{1,4}.
\end{align*}

\bajm{Next look at $\Delta^{l=1}_{6,((3),1)}$, $\Delta^{l=2}_{6,((4),0)}$... See if this sheds light on walk $\to$ diagram map:}
For $\Delta^{l=1}_{6,((3),1)}$, in addition to the walks included in the $l=0$ cases (now replacing $(2)^{'}$ with $(3)$, $(2)''$ with $(3)'$  etc) we also include:
\begin{align*}
    &0-1-(11)-(21)-(2)-(3)-(3)' &&\sim e_{1,4}, \\
&0-1-(2)-(21)-(2)-(3)-(3)'&&\sim e_{2,4},
\end{align*}

For $\Delta^{l=2}_{6,((2),2)}$ we include the previous walks (replacing $(3)'$ with $(4)$) and we include:
\begin{align*}
    &0-1-(11)-(21)-(31)-(3)-(4) &&\sim e_{1,5}, \\
&0-1-(2)-(21)-(31)-(3)-(4)&&\sim e_{2,5},\\
&0-1-(2)-(3)-(31)-(3)-(4)&&\sim e_{3,5},
\end{align*}
\bajm{I seem to be on the verge of being able to formulate a correspondence between walks and diagrams, at least in the case of 1D Specht-modules, by considering the positions of ``turnbacks" and ``deviations" from an ``ideal walk" (ideal walks i am tentatively defining as those which consist only of subdiagrams of the final diagram (need to modify to deal with going to right of final diagram...) ).   }


\section{Young Graph by Alcove Geometry}
This section is a small note, devoted to the idea of realising the Young Graph by alcove geometry. Firstly, consider the type $B_{\infty}$ Weyl group, $G$:
\[ G=\langle s_0, s_1,\dots \ | \ s_k^2=(s_0 s_1)^4=(s_{l}s_{l+1})^3=(s_i s_j)^2=e \ \text{for }l>0, |i-j|>1 \rangle,\]
acting on $\R^\infty$ (the space of finite real sequences) as follows:
\begin{align*}
    s_0(x_1,x_2,\dots,x_n)&=(-x_1,x_2,\dots,x_n),\\
    s_i(x_1,\dots,x_i,x_{i+1},\dots,x_n)&=(x_1,\dots,x_{i+1},x_i,\dots,x_n),
\end{align*}
where we use $(x_1,x_2,\dots,x_n)$ as short hand for the same infinite sequence, padded by 0s on the right. By allowing the $x_k$ to possibly be 0, we can ensure the last definition is always sensible. We can consider the $s_0$ as acting by reflection in the hyperplane defined by $x_1=0$, and the $s_i$ as acting by reflection in the hyperplane defined by $x_i=x_{i+1}$.

Given an element $(x)\in \R^\infty$ we can apply $s_0$ (conjugated by permutations to act in each component) so that all non-zero elements are positive, and then apply a permutation to order all entries in descending order. Thus a fundamental domain, $\mathcal{D}$, for $G\circlearrowright \R^\infty$, is the set of finite, descending, positive sequences:
\[\mathcal{D}= \left\{(x_1,x_2,\dots, x_n,0,\dots) \in \R^\infty \ \Bigg| \ \begin{matrix}\text{ for some } n\in \N \text{ such that }\\ x_1\geq x_2\geq \dots \geq x_n>0\end{matrix} \right\}.\]
Now, let us augment our group to be $\hat{G}$ where we add another generator $r_0$ to $G$ which satisfies the exact same relations as $r_0$ and has no ``braiding" relation with $s_0$. We give a Dynkin diagram for the group $\hat{G}$ below:
\[ \begin{tikzcd}[column sep=tiny,row sep = tiny]
    s_0\ar[dr,equal]\\
    &s_1\ar[r,dash]&s_2\ar[r,dash]&s_3\ar[r,dash]&\dots\\
    r_0\ar[ur,equal]\ar[uu,dash,"\infty"]
\end{tikzcd}\]
Now let $r_0$ act on $\R^\infty$ by the (affine) reflection in the hyperplane defined by $x_1=1$, that is,
\[r_0(x_1,x_2,\dots,x_n)=(2-x_1,x_2,\dots,x_n).\]
We then observe that $r_0 s_0$ acts by $x_1\mapsto x_1+2$ (fixing all other components). For a partition, $\l\vdash n$, we may view $\l=(\l_1,\l_2,\dots, \l_r)\in \mathcal{D}$ with $\l_i\geq \l_{i+1}$ for all $i$. Now define
\[\mathcal{D}_{\l}:=\{x\in \R^\infty \ | \ x-\l \in [0,1)^\infty  \}.\]
Clearly $\mathcal{D}=\sqcup_{\l} \mathcal{D}_{\l}$ and furthermore, we claim that each $\mathcal{D}_{\l}$ is a fundamental domain for $\hat{G}\circlearrowright \R^\infty$ (or an alcove) which shares a face with $\mathcal{D}_{\l'}$, if and only if, $\l'$ is a partition obtained from $\l$ by ``adding or removing a box". Suppose that $\l$ and $\l'=(\l_1,\dots,\l_i+1,\dots ,\l_r)$ are both partitions. Then $\mathcal{D}_{\l}$ and $\mathcal{D}_{\l'}$ (or rather, their closures) share the face
\[\overline{\mathcal{D}_{\l}}\cap\overline{\mathcal{D}_{\l'}}=\left\{\begin{matrix}(\l_1+\xi_1,\dots,\l_i+1,\dots,\l_r+\xi_r,\xi_{r+1},\dots,\xi_k)\in \mathcal{D}  \\ :\text{for some } k \in \N, \text{ where } \xi_l\in [0,1]\end{matrix} \right\}.\]
Therefore, by identifying each alcove $\mathcal{D}_{\l}$ with the underlying partition $\l$, we obtain the Young-graph as the dual graph of this alcove geometry.

}} 

\bibliographystyle{abbrv}
\bibliography{local.bib}
\end{document}